%% file: cubiccoh.tex
\documentclass[10pt]{amsbook}
\usepackage[dvipsnames]{xcolor}
\usepackage{amssymb,amsmath,amsthm,amscd,mathrsfs,graphicx}
\usepackage[cmtip,all,matrix,arrow,tips,curve]{xy}
\usepackage[colorlinks=true]{hyperref}
\usepackage{tikz}
\usepackage{pdfpages}

\usepackage{hyphenat}
\numberwithin{section}{chapter}
\numberwithin{equation}{chapter}
\newtheorem{teo}{Theorem}[chapter]
\newtheorem{pro}[teo]{Proposition}
\newtheorem{lem}[teo]{Lemma}
\newtheorem{cor}[teo]{Corollary}

\theoremstyle{definition}

\newtheorem{exa}[teo]{Example}

\theoremstyle{remark}
\newtheorem{rem}[teo]{Remark}

\newtheorem{conv}[teo]{Convention}

\newcommand{\calA}{\mathcal{A}}
\newcommand{\calB}{\mathcal{B}}
\newcommand{\calD}{\mathcal{D}}
\newcommand{\calE}{\mathcal{E}}
\newcommand{\calG}{\mathcal{G}}

\newcommand{\calP}{\mathcal{P}}
\newcommand{\calH}{\mathcal{H}}
\newcommand{\calO}{\mathcal{O}}
\newcommand{\calM}{\mathcal{M}}
\newcommand{\calN}{\mathcal{N}}

\newcommand{\calR}{\mathcal{R}}

\newcommand{\calT}{\mathcal{T}}
\newcommand{\calF}{\mathcal{F}}

\newcommand{\GL}{\operatorname{GL}}
\newcommand{\PGL}{\operatorname{PGL}}
\newcommand{\gl}{\mathfrak{gl}}
\newcommand{\SL}{\operatorname{SL}}

\newcommand{\Pic}{\operatorname{Pic}}

\newcommand{\Hom}{\operatorname{Hom}}
\newcommand{\Sym}{\operatorname{Sym}}

\newcommand{\Aut}{\operatorname{Aut}}

\newcommand{\diag}{\operatorname{diag}}

\renewcommand{\O}{\operatorname{O}}
\newcommand{\divisore}{\operatorname{div}}
\newcommand{\surf}{{\operatorname{surf}}}
\newcommand{\curv}{{\operatorname{curve}}}
\newcommand{\J}{{\mathfrak{j}}}

\newcommand{\fH}{\mathfrak{H}}

\newcommand{\AAA}{\mathbb{A}}
\newcommand{\ZZ}{\mathbb{Z}}
\newcommand{\QQ}{\mathbb{Q}}
\newcommand{\RR}{\mathbb{R}}
\newcommand{\CC}{\mathbb{C}}
\newcommand{\TT}{\mathbb{T}}
\newcommand{\PP}{\mathbb{P}}
\newcommand{\FF}{\mathbb{F}}
\newcommand{\SSS}{\mathbb{S}}

\newcommand{\Sat}[1][g]{{{\calA}_{#1}^{*}}}
\newcommand{\gquot}{/\!\!/}

\newcommand{\hM}{\widehat{\calM}}

\newcommand{\BG}{(\calB/\Gamma)^*}
\newcommand{\oBG}{\overline{\calB/\Gamma}}

\newcommand{\cDh}{{\calD}_{h}}
\newcommand{\cDn}{{\calD}_{n}}
\newcommand{\GIT}{\calM^{\operatorname{GIT}}}
\newcommand{\MK}{\calM^{\operatorname{K}}}

\newcommand\blfootnote[1]{%
  \begingroup
  \renewcommand\thefootnote{}\footnote{#1}%
  \addtocounter{footnote}{-1}%
  \endgroup
}

\begin{document}
\frontmatter


\title[Cohomology of the moduli of cubic threefolds]{Cohomology of the moduli space of cubic threefolds and its smooth models}

\author[S. Casalaina-Martin]{Sebastian Casalaina-Martin}
\address{University of Colorado, Department of Mathematics, Boulder, CO 80309}
\email{casa@math.colorado.edu}

\author[S. Grushevsky]{Samuel Grushevsky}
\address{Stony Brook University, Department of Mathematics, Stony Brook, NY 11794-3651}
\email{sam@math.stonybrook.edu}

\author[K. Hulek]{Klaus Hulek}
\address{Institut f\"ur Algebraische Geometrie, Leibniz Universit\"at Hannover,  30060 Hannover, Germany}
\email{hulek@math.uni-hannover.de}

\author[R. Laza]{Radu Laza}
\address{Stony Brook University, Department of Mathematics, Stony Brook, NY 11794-3651}
\email{rlaza@math.stonybrook.edu}

\begin{abstract}
We compute and compare the (intersection) cohomology of various natural geometric compactifications of the moduli space of cubic threefolds: the GIT compactification and its Kirwan blowup, as well as the Baily--Borel and toroidal compactifications of the ball quotient model, due to Allcock--Carlson--Toledo. Our starting point is Kirwan's method. We then follow by investigating the behavior of the cohomology under the birational maps relating the various models, using the decomposition theorem in different ways, and via a detailed study of the boundary of the ball quotient model. As an easy illustration of our methods, the simpler case of the moduli of cubic surfaces is discussed in an appendix.
\end{abstract}


\date{\today}

\maketitle
\tableofcontents

\mainmatter
\include{sec1_intro}

\include{sec2_preliminaries}

\include{sec3_HXsemistable}

\include{sec4_HKirwanblowup}

\include{sec5_IHGIT}

\include{sec6_IHball}

\include{sec7_IHtoroidal}

\appendix

\include{sec8_equivcoh}

\include{sec9}

\include{sec10_cub_surface}

\backmatter
\bibliographystyle{amsalpha}
\bibliography{cubiccoh}
\printindex

\end{document}

%% file: sec1_intro.tex
\chapter{Introduction}\label{sec:intro}

\ 
\blfootnote{
2010 \emph{Mathematics Subject Classification}. 14J30, 14J10, 14L24, 14F25, 55N33, 55N25.

Research of the first author is  supported in part by grants from the Simons Foundation (317572) and the NSA (H98230-16-1-0053). Research of the second author is supported in part by NSF grants DMS-15-01265 and DMS-18-02116. Research of the third author is supported in part by DFG grant Hu-337/7-1. Research of the fourth author is supported in part by NSF grants DMS-12-54812 and DMS-18-02128. The first author would like to thank the Institut f\"ur Algebraische Geometrie at Leibniz Universit\"at for support during the Fall Semester 2017. The first and third authors are also grateful to MSRI Berkeley, which is supported by NSF Grant DMS-14-40140,  for providing excellent working conditions in the Spring Semester 2019. 

}
Cubic threefolds and their moduli are one of the most studied objects in algebraic geometry. In previous work we have investigated the relationship among various compactifications   of the moduli space $\calM$
of smooth cubic threefolds, and the purpose of this paper is now to determine the cohomology of these moduli spaces.
The first compactification which one naturally encounters is, as for all hypersurfaces, the GIT compactification $\GIT$ (as studied by Allcock~\cite{allcock} and Yokoyama~\cite{yokoyama}).
 It is interesting to note that recently Liu--Xu~\cite{LiuXu} showed that for cubic threefolds (and also for cubic surfaces) $\GIT$ is  equal to the moduli space  of $K$-stable cubics, thus
providing a differential-geometric perspective on the GIT moduli of cubics.

What makes the case of cubic threefolds especially interesting is the presence of two period maps which lead to further natural compactifications.
The first of these period maps is given by the intermediate Jacobian and was already studied by Clemens--Griffiths~\cite{cg}.  The Torelli theorem holds for this period map for cubic threefolds, and one obtains an immersion $\calM\hookrightarrow \calA_5$ into the moduli space $\calA_5$ of principally polarized abelian varieties of dimension $5$. Taking the closure  $\overline {IJ}\subset \overline\calA_5$ of the locus $IJ$ of intermediate Jacobians in suitable compactifications $\overline\calA_5$  of $\calA_5$, one obtains geometrically meaningful compactifications of $\calM$ (see~\cite{cubics}).

Perhaps even more surprising is that one can construct a $10$-dimensional ball quotient model $\calB/\Gamma$ of $\calM$, by using the periods of cubic fourfolds (cf.~Allcock--Carlson--Toledo~\cite{act}).
This ball quotient admits naturally the Baily--Borel compactification $\BG$ and the (unique) toroidal compactification~$\oBG$, which thus provide two further compactifications of the moduli of smooth cubic threefolds.
It is, in particular, these models which we will study in this paper.
The spaces $\GIT$ and $\BG$ are closely related, as explained in~\cite{act} and~\cite{ls}. Briefly, there exists a space $\widehat\calM$ dominating both $\GIT$ and $\BG$. In fact  $\widehat\calM$ plays two roles: on the
one hand it is the partial Kirwan blowup of the point $\Xi\in\GIT$
corresponding to the chordal cubic, and on the other hand it is the Looijenga $\QQ$-factorialization (cf.~\cite{l1}) associated to the hyperelliptic divisor $\calH_h^*\subset\BG$.
Both compactifications $\GIT$ and $\BG$ are singular. The toroidal compactification~$\oBG$ is a natural (partial) desingularization of~$\BG$, while a natural (partial) desingularization of $\GIT$ is provided by Kirwan's blowup $\MK$, which is smooth up to finite quotient singularities. By construction there is a factorization $\calM^K\to\widehat \calM\to \GIT$, as $\widehat \calM$ is nothing but an intermediary step in the construction of the Kirwan blowup $\calM^K$.
The relationship among these compactifications and $\overline{IJ}$ was the subject of our previous works~\cite{cml,cubics, prym}.

\smallskip

In this paper we investigate this relationship further by turning our attention to the cohomology of these moduli spaces.  More precisely, we determine the (intersection) cohomology of the compactifications $\GIT$, $\widehat\calM$, $\MK$, $\BG$ and $\oBG$:

\begin{teo}\label{teo:betti}
The Betti numbers of $\MK$ and $\oBG$, and the intersection Betti numbers of $\GIT$, $\widehat\calM$, and $\BG$ are as follows:
\begin{equation}
\renewcommand*{\arraystretch}{1.2}
\begin{array}{l|ccccccccccc}
\hskip2cm j&0&2&4&6&8&10&12&14&16&18&20\\\hline
\dim H^j(\MK)&1&4&6&10&13&15&13&10&6&4&1\\
\dim IH^j(\GIT)&1&1&2&3&4&5&4&3&2&1&1\\
\dim IH^j(\widehat\calM)&1&2&3&5&6&8&6&5&3&2&1\\
\dim IH^j(\BG)&1&2&3&5&6&7&6&5&3&2&1\\
\dim H^j(\oBG)&1&4&6&10&13&15&13&10&6&4&1
\end{array}
\end{equation}
while all the odd degree (intersection) cohomology vanishes.
\end{teo}

\begin{conv}
As usual with these type of cohomological computations, the cohomology is always with $\QQ$ coefficients. This will be our convention throughout the paper.
\end{conv}

\begin{rem}
The easier related case of the moduli space of cubic surfaces is discussed in Appendix~\ref{sec:surfaces}. Specifically, as in the case of cubic threefolds, there exists both a GIT model (one of the standard examples in classical Invariant Theory) and a ball quotient model for the moduli space of cubic surfaces (due to Allcock--Carlson--Toledo~\cite{ACTsurf}; see also~\cite{DvGK}). However, in this lower-dimensional case, the two models are isomorphic (cf.~\cite{ACTsurf}). The cohomology of the GIT model and of its partial Kirwan desingularization were worked out by Kirwan as illustrations of her general theory (see~\cite{kirwanhyp}, and also~\cite{ZhangCubic}). On the other hand, to our knowledge, Theorem~\ref{T:CubSurfH}, which computes the cohomology of the associated toroidal compactification of the ball quotient model for cubic surfaces, and Theorem~\ref{T:Naruki}, which computes the cohomology of the Naruki compactification for the surface case, are new, and possibly of independent interest (see also Remark~\ref{rem_possible_iso} below).
\end{rem}

\begin{rem}\label{rem:sing}
Let us briefly comment on the singularities of the various compactifications that occur in our paper. First, by construction, $\MK$ and $\oBG$ have only finite quotient singularities. In particular, their cohomology coincides with their intersection cohomology (with $\QQ$ coefficients). The intermediate space $\widehat\calM$, which resolves the birational map $\GIT\dashrightarrow \BG$, has only toric singularities. In contrast, the two starting points of our analysis, the GIT quotient $\GIT$ and the Baily--Borel compactification $\BG$, have worse singularities. Specifically, the GIT quotient $\GIT$ has at worst finite quotient singularities along the stable locus $\calM^s$, and toric singularities along the GIT boundary, except for the  point $\Xi$ that corresponds to the chordal cubic threefold. Finally, $\BG$ has at worst finite quotient singularities in the interior $\calB/\Gamma$, but the singularities at the two isolated boundary points (the cusps) of $\BG$ are fairly complicated. The precise description of the (partial) resolutions $\MK\to \GIT$ and $\oBG\to \BG$ constitutes an important part of our paper (see esp. Sections~\ref{sec:prelim} and~\ref{sec:toroidal} respectively).
\end{rem}

\begin{rem}\label{rem_possible_iso}
We note that the first and the last row of this table are identical. This is to say, the Betti numbers of the two compactifications that are smooth up to finite quotient singularities --- the Kirwan blowup $\MK$ and the toroidal compactification $\oBG$ of the ball quotient --- coincide. This leads to the natural question of whether these two compactifications are in fact isomorphic. Geometrically, both $\MK$ and $\oBG$ are blowups of $\BG$ at the same two points which are the two cusps of $\BG$. However, it is unclear whether the blowup ideals are the same in both cases (see~\cite[\S5.1]{LOG2} for some related computations).
Similarly, in Appendix~\ref{sec:surfaces} we show that the Betti numbers of the Kirwan blowup and of the toroidal compactification for the moduli of cubic surfaces (see~\cite{ACTsurf}) are also equal (Theorem~\ref{T:CubSurfH}).
Even in this easier case, while we expect that the Kirwan blowup and the toroidal compactification for the moduli of cubic surfaces are in fact isomorphic, some subtle details remain to  be settled.
Answering this question will require methods very different from what we use in this paper, and we plan to return to this question in the future.
\end{rem}

In addition to the fact that the (intersection) Betti numbers of a moduli space are a basic invariant of interest, there are several further reasons for our interest in these numbers.
In particular, our work here provides a better understanding of  the geometry of the birational maps among the various compactifications of
the moduli space of cubic threefolds. In general it is a natural question to ask how different compactifications of a given moduli space, each often arising as the result of a natural compactification process, relate to each other. One way of understanding
such relations can be via the log-MMP with respect to a suitable linear combination of boundary divisors.
This is a very active subject of research, widely known as {\it the Hassett--Keel program}, in the case of the moduli space of curves ${\mathfrak M}_g$ (see eg.~\cite{HH1,HH2}).  The motivation is that the log-MMP allows one to interpolate between a known compactification (such as the Deligne--Mumford compactification $\overline{\mathfrak M}_g$) and a target compactification (such as the canonical model $\mathfrak M_g^{can}$ for $g\ge23$).  More recently, Laza and O'Grady~\cite{LOG2,LOG1} have used a variation of log-models to understand the relationship between the GIT and Baily--Borel compactifications for low degree (esp. quartic) $K3$ surfaces. It is natural to ask whether a similar picture arises for moduli spaces of cubics (see~\cite[Sect. 7]{cml} for some further discussion).
In particular, in this context the question raised by the remark above, of whether $\MK$ and $\oBG$ are in fact isomorphic, seems to be the natural starting point, and resolving it might give some indication of the properties of the log-MMP in this case.

In another direction, our results provide a geometric approach to computing the cohomology of an interesting ball quotient (the Allcock--Carlson--Toledo model $\calB/\Gamma$ for the moduli space of cubic threefolds) and its compactifications. First, since $\calB/\Gamma$ is a locally symmetric variety, there are several interesting questions related to its topology. One natural question is whether its cohomology is generated by arithmetic cycles, i.e., Shimura subvarieties, which in this case will be sub-ball quotients $\calB'/\Gamma'$.   Our results provide a starting point for  identifying some geometrically meaningful candidates for such subvarieties  (e.g., loci corresponding to cubic threefolds with specified singularities, or cubics with specified automorphisms), although we are far from being able to answer this question completely. Analogous questions were considered in the case of orthogonal modular varieties (also known as type IV or $K3$ type) under the heading of the Noether--Lefschetz conjecture. This was verified by Bergeron et al.~\cite{bergeron} who show that the cohomology of locally symmetric varieties of type IV is generated at least up to middle dimension by Shimura subvarieties.

Second, we note that one can also approach the computation of the  intersection cohomology of Baily--Borel compactifications via automorphic representations and trace formulae. This has been advanced very successfully  in the  case of the Satake compactification $\Sat$ of the moduli space of principally polarized abelian varieties, where the intersection cohomology is completely known for $g \leq 7$ (also for intersection cohomology with coefficients in any local system), see~\cite{hulektommasi}. This, as well as the work by Bergeron et al.,  relies on Arthur's endoscopic classification of automorphic representations of the symplectic group. In principle, Arthur's method can also be applied to the unitary group (i.e., the case of ball quotients) as was shown by Mok~\cite{mok}, but to the best of our knowledge the $10$-dimensional case which we treat here has not yet been approached by representation-theoretic methods.

Finally, while there has been some previous work computing the intersection cohomology of Baily--Borel compactifications of ball quotient models, in this paper we work out the cohomology of the \emph{toroidal} compactification.
To our knowledge, this is the first nontrivial example where the
 intersection cohomology of the toroidal compactification of an arithmetic ball quotient model of a moduli space has been computed.
The techniques should be applicable to other examples of interest.
In fact, as our ten-dimensional ball quotient is the largest of the ball quotient models related to natural moduli problems, the results should be immediately applicable in these other situations.  As mentioned above, in Appendix~\ref{sec:surfaces} we for instance apply our techniques to the ball quotient model of the moduli space of cubic surfaces.

Our approach takes as its starting point Kirwan's general theory (see~\cite{kirwan84,kirwanblowup}) of computing the (intersection) cohomology of GIT quotient spaces.
In her paper~~\cite{kirwanhyp} Kirwan uses her techniques to perform the computations for the cases of cubic and quartic surfaces.
Furthermore, Kirwan and her collaborators have done such computations for Baily--Borel compactifications of the moduli space of $K3$ surfaces of degree $2$ (see~\cite{KL1,KL2}) and the Deligne--Mostow ball quotients (see~\cite{KLW}).
Indeed, the largest Deligne--Mostow ball quotient, corresponding to $12$ points in $\PP^1$, is directly related to our analysis, as it corresponds to the hyperelliptic divisor $\calH_h^*$ in $\BG$.
However, our situation is that of the Baily--Borel compactification of the ball quotient $\BG$, which is of dimension $10$, and goes beyond the Deligne--Mostow examples.

While our basic setup is similar to these works, we encounter various new phenomena and complications, which make our computations considerably more intricate, and in particular require a careful analysis of the geometry of our situation. Combining Kirwan's machinery and geometric descriptions of various unstable and polystable loci (some available in the literature, but with further information deduced in this paper) allows us to compute the cohomology of  $\GIT$, $\widehat \calM$, and $\MK$.
Next we compute the cohomology of the Baily--Borel compactification $\BG$ by applying the decomposition theorem to the natural morphism $\MK\to\BG$.
We finally compute the cohomology  of the toroidal compactification~$\oBG$ by applying the decomposition theorem to the natural morphism $\oBG\to \BG$, which is the blowup of the two points which are the cusps of~$\BG$.
We note that, as for all ball quotients, there are no choices involved in the construction of the toroidal compactification.
The computation of the cohomology of the toroidal boundary divisors requires a careful analysis of the arithmetic and the geometry of the two cusps. This involves the theory of Eisenstein lattices and leads to a wealth of new geometric insights.
In particular, we are led to generalize a Chevalley type result due to Looijenga  and Bernstein--Schwarzman~\cite{Lroot, FMW, BS} to the case of Eisenstein lattices (see \S\ref{S:Eies-Chev}).
Furthermore, as an immediate and easy application of our techniques, we can for instance compute the cohomology of the toroidal compactification of the ball quotient model of the moduli space of cubic surfaces (Theorem~\ref{T:CubSurfH}).

Let us briefly go over the content of our paper. We start in Chapter~\ref{sec:prelim} with some preliminaries. Specifically, we first briefly review (\S\ref{subsec:modulicubic}) the work of Allcock~\cite{allcock} and Allcock--Carlson--Toledo~\cite{act} (see also Looijenga--Swierstra~\cite{ls}) on the moduli space of cubic threefolds and its two compact models $\GIT$ and $\BG$. We then review (\S\ref{subsec:kirwanth}) the basic framework of Kirwan's method (and fix the necessary notation). In particular, we introduce the space $\calM^K$, the Kirwan (orbifold) desingularization of the GIT model $\GIT$, that plays a key role in our analysis.

In Sections~\ref{sec:HXss} and~\ref{S:CohKirBl-II} we compute the cohomology of the Kirwan resolution $\calM^K$.
There are two main steps in the computation.
First is the computation of the equivariant cohomology of the semi-stable locus $X^{ss}$ in the Hilbert scheme of cubic threefolds (\S\ref{sec:HXss}). This is done by computing the usual Kempf stratification of the unstable locus, followed by an excision type argument. A key simplifying observation of Kirwan is that, for the purposes of eventually computing the intersection cohomology (or equivalently, cf. Remark \ref{rem:sing}, the cohomology) of $\MK$, one can safely ignore unstable strata of high codimension. In fact,  for the analogous computation in the case of quartic surfaces quartic surfaces discussed in~\cite{kirwanhyp}, all unstable strata can be ignored. To our surprise, this is no longer the case for the strata of unstable cubic threefolds, leading to some additional complications in the computation of $H^\bullet(\MK)$, since the locus of unstable cubic threefolds with a $D_5$ singularity plays a role.  The next step, after computing the equivariant cohomology of the semi-stable locus $X^{ss}$, is to compute some correction terms (\S~\ref{S:CohKirBl-II}) that arise from blowing up the loci of strictly polystable points in $X^{ss}$ in the construction of $\MK$.

Once the computation of the cohomology of $\MK$ is completed, Kirwan's setup allows one to in principle approach the computation of the {\em intersection} cohomology of the GIT compactification~$\GIT$. To do this, Kirwan sets up an appropriate application of a suitable equivariant version of the decomposition theorem. In order to apply this, one needs to solve separate GIT problems for actions on the tangent space of suitable normalizers of stabilizers of strictly semi-stable points. We perform this computation, and turn out to be lucky in that the suitable quotients of strictly semi-stable loci are two points and a $\PP^1$ in our case, which allows the computation of relevant intersection local systems. Along the way, we also determine the intersection cohomology of $\widehat \calM$ as an intermediate step. This is discussed in Chapter~\ref{sec:IHGIT}.

We then further descend the computations from $\MK$ to $\BG$.
To do this, we apply the decomposition theorem directly to the map $\MK\to \BG$. The crucial point here is that the
Kirwan blowup $\MK$ is smooth up to finite quotient singularities and that the map $\MK\to \BG$ is a blowup in two points whose preimages are divisors in $\MK$.
 The decomposition theorem then has a simple description in terms of the cohomology of these exceptional divisors.
Since most of the work in computing the cohomology of those exceptional divisors was already done in the computation of the intersection cohomology of $\GIT$, the computation becomes feasible.
This is discussed in Chapter~\ref{sec:IHball}.

Finally, in Chapter~\ref{sec:toroidal} we compute the intersection cohomology of the toroidal compactification $\oBG$. Since $\oBG$ is a smooth up to finite quotient singularities, this computation is also done by applying directly the decomposition theorem, this time to the morphism $\oBG\to\BG$, which is also a blowup of the two cusps in $\BG$, with the total space smooth (also up to finite quotient singularities). This requires computing the cohomology of the two exceptional toroidal divisors of $\oBG$, which get contracted  to the two cusps of $\BG$.
This turns out to be an interesting question in its own right, whose solution involves the theory of Eisenstein lattices as well as an equivariant version (Proposition~\ref{thm_L_Eis}) of a Chevalley type theorem of Looijenga~\cite{Lroot} and Bernstein--Schwarzman~\cite{BS}.

As  Kirwan's machinery involves computations with equivariant cohomology, for convenience we have summarized in Appendix~\ref{sec:equivcoh} the  properties of equivariant cohomology that we will use. To apply this general machinery, one still needs to determine various stabilizers, normalizers, their fixed point sets, etc. Such computations, though elementary, are quite lengthy and laborious. To streamline the flow of the text, we have gathered all such results in Appendix~\ref{S:Elem}. Finally, Appendix~\ref{sec:surfaces}  discusses the easier case of the moduli space  of cubic surfaces, where we prove that the cohomology of the Kirwan blowup, toroidal, and the Naruki compactifications are all equal.

\section*{Acknowledgements}
We thank Frances Kirwan for answering some questions concerning her work. The first author thanks Jonathan Wise for some discussions on equivariant cohomology. The fourth author thanks Kieran O'Grady and Rahul Pandharipande for some discussions on related issues regarding the cohomology of Baily--Borel compactifications of moduli of low degree K3 surfaces. We thank Mauro Fortuna for carefully reading the manuscript and pointing out some inaccuracies in previous drafts.
We finally thank the referee for a very thorough and extremely helpful reading of our manuscript.

%% file: sec2_preliminaries.tex
\chapter{Preliminaries}\label{sec:prelim}
In this section, we will review some basic facts about the moduli of cubic threefolds (mostly due to Allcock~\cite{allcock} and Allcock--Carlson--Toledo~\cite{act}), and importantly, introduce the two main actors in our paper: the GIT quotient $\GIT$ and the ball quotient model $\BG$ (as well as their common resolution $\widehat \calM$). We then recall  Kirwan's resolution $\MK$ of $\GIT$, and explain its connection with $\widehat \calM$.

\section{Notation and conventions}

\subsection{The general setting}\label{SSS:KirSetUp} In order to keep our presentation consistent with that of~\cite{kirwanblowup, kirwanhyp, GIT}, and in order to discuss some of the details of Kirwan's construction, we first recall the general framework.
We start with a complex projective manifold $X\subseteq \PP^N$,  a complex reductive group $G$ acting algebraically on $X$, and  a $G$-linearization of the action on
 the very ample line bundle $L=\calO_{\PP^N}(1)|_X$.
A complex Lie group $G$ is reductive if and only if it is the complexification of a maximal compact subgroup,
and we fix one such subgroup $K$.
 We assume that the action and the linearization are induced by a faithful representation
$$
\rho:G\longrightarrow \GL(N+1,\CC)
$$
such that $\rho(K)\subset\operatorname{U}(N+1)$.     We fix a maximal algebraic torus $\TT\cong (\CC^*)^{N+1}$ in $G$, and a corresponding maximal compact torus $T$ in $K$, so that $T$ is a maximal compact subgroup of $\TT$.   Let $\alpha_0,\dots,\alpha_{N}\in \mathfrak t^\vee$ be the  weights of the representation of $K$, lying in the dual to the Lie algebra~$\mathfrak t$ of~$T$; if $(x_0:\dots :x_N)$ are the coordinates on $\PP^N=\PP\CC^{N+1}$ diagonalizing the action of $T$, then we associate to $x_i$ the weight $\alpha_i$.
We fix an inner product on the Lie algebra  $\mathfrak k$ of $K$ that is invariant under the adjoint action of $K$ (for example the Killing form), and use its restriction to $\mathfrak t$ to  identify $\mathfrak t=\mathfrak t^\vee$.
We also fix once and for all a positive Weyl chamber $\mathfrak t_+$.

\subsection{The case of  hypersurfaces}\label{S:hypeKirConv}
In this paper we will be  specializing to the case of  hypersurfaces of degree $d$ in $\PP^n$, and eventually to cubic threefolds.  To keep the notation consistent with the previous subsection, and in particular consistent with~\cite{kirwanhyp}, we take $X=\PP H^0(\PP^n,\calO_{\PP^n}(d))=\PP\operatorname{Sym}^d(\CC^{n+1})^\vee$, i.e., $X=\PP^N$ with $N=\binom{n+d}{d}-1$, and we take $G=\SL(n+1,\CC)$ acting via the natural representation on $\operatorname{Sym}^d(\CC^{n+1})^\vee$ induced by the canonical matrix action   on $(\CC^{n+1})^\vee=\mathbb C^{n+1}$.
This induces a linearization of the action for $\calO_{\PP^N}(1)$.
We note that the action of $G$ on $X$ is not faithful: the center of $\SL(n+1,\CC)$, which is isomorphic to $\ZZ/(n+1)\ZZ$, consisting of diagonal matrices with the same $(n+1)$-st root of unity along the diagonal, acts trivially on~$X$.

\begin{rem}
As is typical in this situation, there is some choice involved in picking the group $G$.  The choice of $\operatorname{SL}(n+1,\mathbb C)$ is preferable from the perspective of linearizations and GIT (see~\cite[p.33 and Prop.~1.4]{GIT}).  On the other hand, since the action of $\PP\GL(n+1,\mathbb C)$ on $X$ is faithful, and automorphisms of a hypersurface are identified with the stabilizer of the corresponding point under this action, it can frequently be convenient to work with  $\PP\GL(n+1,\mathbb C)$ when computing stabilizers.  Finally, it turns out that sometimes the stabilizers (and related groups) are easier to describe from the group theoretic perspective as subgroups of  $\GL(n+1,\mathbb C)$.   Since we can easily go back and forth among the various groups, we take $G=\SL(n+1,\mathbb C)$, so as to work well in the GIT setting, and be consistent with Kirwan's conventions.
\end{rem}

In this case $K=\operatorname{SU}(n+1)$, and $T\cong (S^1)^{n}$ is the subgroup of diagonal unitary matrices with determinant $1$.
The root system for $\operatorname{SU}(n+1)$ is of type $A_n$, with Weyl group the symmetric group $S_{n+1}$, and we fix a positive Weyl chamber $\mathfrak t_+$.  The Killing form on $\mathfrak {su}(n+1)$ is given by $A.B=2n\operatorname{tr}(AB)$; thus when restricted to the diagonal traceless matrices of $\mathfrak t$, identified as the hyperplane $\{(a_0,\dots,a_n)\in \RR^{n+1}: \sum a_i=0\} \subseteq \RR^{n+1}$, the inner product on $\mathfrak t$ is $2n$ times the standard inner product.
For simplicity, we will always use the standard inner product.
To describe the weights of the representation of $\operatorname{SU}(n+1)$ concretely,
we take as a basis for $\operatorname{Sym}^d(\CC^{n+1})^\vee$ the monomials of degree $d$.  As usual, we use the notation $x^I:=x_0^{i_0}\dots x_n^{i_n}$, where $I=(i_0,\dots,i_n)$ is a partition of $d$, to index our monomials.   A diagonal matrix $\operatorname{diag}(\lambda_0,\dots,\lambda_n)$ acts on $x^I$ by scaling by $\lambda_0^{i_0}\dots \lambda_n^{i_n}$,
and thus the index $I$ also gives the weight $\alpha_I$ associated to the coordinate $x^I$.
More precisely, the monomials
naturally sit as lattice points in the non-negative  quadrant of  $\ZZ^{n+1}$, and the monomials of fixed degree $d$ can be thought of as the lattice points of a simplex in the affine $n$-space whose defining equation is that the sum of coordinates is $d$.     We make the identification of monomials of degree $d$ with weights in $\mathfrak t\subseteq \RR^{n+1}$ explicit with the assignment $x^I\mapsto \alpha_I:= (i_0-d/(n+1),\dots,i_n-d/(n+1))$.

\subsection{The case of cubic threefolds}
The particular case of interest in this paper is the case of cubic threefolds.
As in the previous subsection, to fix the notation to match~\cite{kirwanhyp}, we set throughout the paper $d=3$ for the degree of the hypersurfaces, $n=4$ for dimension of the ambient $\PP^4$, $X=\PP^{34}=\PP\operatorname{Sym}^3(\CC^5)^\vee$ for the parameter space for cubic threefolds, and $G=\SL(5,\CC)$ for the reductive group acting on $X$ via change of coordinates, with the canonical linearization on $\calO_{\PP^{34}}(1)$.

\subsection{Strictly polystable points}
As before, let $G$ be a reductive group acting on a projective variety $X$ with a $G$-linearized ample line bundle $L$.
A point $x\in X$ is {\it semi-stable} if there exists an invariant section $\sigma\in H^0(X,L^m)^G$ (for some $m\in \ZZ_+$) such that $\sigma(x)\neq 0$.
We denote by $X^{ss}(L)$, or simply $X^{ss}$ if no confusion on $L$ is possible, the set of semi-stable points. A point $x\in X^{ss}(L)$ is {\it polystable} if the orbit $G\cdot x$ is closed in the locus of semi-stable points $X^{ss}(L)$. The stabilizer of a polystable point is a reductive group. We recall that the points of the GIT quotient $X\gquot_L G(=X^{ss}(L)/G)$ are in one-to-one correspondence with the orbits of the polystable points. Finally, a point $x\in X^{ss}(L)$ is {\it stable} if it is polystable with finite stabilizer. We denote by $X^s(L)\subset X^{ss}(L)$ (or simply $X^s$) the open subset of stable points. The quotient $X^s/G$ is a geometric quotient, in particular the points of  $X^s/G$ are in one-to-one correspondence with the $G$-orbits in $X^s(L)$.
We will use the terminology of {\it strictly polystable} points for polystable points that are strictly semi-stable (i.e., the point is polystable, and semi-stable, but is not stable).

The main tool for determining the semi-stable/stable points is Mumford's numerical criterion (e.g.~\cite[\S2.1]{GIT}). For the case relevant here, the cubic threefolds, a complete description of the semi-stable/polystable/stable points was done by Allcock~\cite{allcock} and Yokoyama~\cite{yokoyama}, as reviewed below.

\section[Standard compactifications $\GIT$ and $\BG$]{Moduli space of cubic threefolds and its standard compactifications $\GIT$ and $\BG$}\label{subsec:modulicubic}

\subsection{The GIT compactification $\GIT$}
With $X=\PP^{34}=\PP\operatorname{Sym}^3(\CC^5)^\vee$, the parameter space for cubic threefolds, and $G=\SL(5,\CC)$ acting via change of coordinates, as above,
the natural GIT compactification for the moduli space of cubic threefolds is denoted
$$\GIT:=X\gquot G\,\,.$$
Note that since projective space has Picard rank $1$, and $G=\SL(5,\CC)$, there is essentially a unique choice of linearization for defining the GIT quotient~\cite[Prop.~1.4, p.33]{GIT}.
The open subset parameterizing smooth cubics will be denoted throughout by $\calM$, and the stable locus will be denoted by $\calM^s=X^s/G$. Clearly, one has
$$\calM\subset \calM^s\subset \GIT\,,$$
and $\calM^s$ has at worst finite quotient singularities.

The GIT compactification $\GIT$ for cubic threefolds was analyzed by  Allcock in~\cite{allcock} and Yokoyama in~\cite{yokoyama}.
They showed that semi-stability of a cubic threefold is determined by its singularities (with almost no global information needed; this is quite special to this case).
In particular, all the semi-stable cubics have isolated singularities, with a single exception, the {\it chordal cubic}, i.e. the secant variety of a rational normal curve in $\PP^4$ (see~\eqref{eq:ch} below for an explicit equation). The chordal cubic is polystable, and we denote by $\Xi$ its orbit, which we view as a special point $\Xi\in \GIT$ of the GIT quotient.

For further reference, we summarize  the GIT analysis for cubic threefolds (cf.~\cite[Thms.~1.1 -- 1.4]{allcock}) as follows:

\begin{teo}[{GIT compactification for cubic threefolds,~\cite{allcock}}] \label{T:GITcub}
The following hold:
\begin{itemize}
\item[(1)] A cubic threefold is GIT stable if and only if it has at worst $A_1, \dots ,A_4$-singularities.
\item[(2)] The GIT boundary $\GIT - \calM^s$ consists of a rational curve $\calT$ and an isolated point $\Delta$.
\item[(3a)] The polystable orbit parameterized by  $\Delta$ corresponds to a cubic with $3D_4$-singularities,  given by equation~\eqref{eq:3D4}.
\item[(3b)] Under a suitable identification $\calT\cong \PP^1$, the polystable orbits parameterized by $\calT -\{0,1\}$ correspond to cubics with precisely $2A_5$-singularities (see~\eqref{eq:2A5} below for an explicit parameterization).
\item[(3b')] The special point $0\in \calT$ corresponds to a cubic with $2A_5+A_1$-singularities (i.e. the cubics with $2A_5$ singularities parameterized by $\calT$ can acquire an additional node for a special value of the parameter in $\calT\cong \PP^1$).
\item[(3b'')] The special point $1\in \calT$ corresponds to the  chordal cubic (in this situation, the $2A_5$ cubics specialize to a cubic with non-isolated singularities), i.e. the point $\Xi \in \GIT$ identified above.
\end{itemize}
\end{teo}

\begin{rem}\label{R:Alck-poly-form} In what follows, we will need explicit equations for the cubics in strictly polystable orbits. Specifically, we have the following (cf. ~\cite[Thm.~1.2]{allcock}):
\begin{enumerate}
\item The polystable orbit corresponding to the isolated boundary point $\Delta\in\GIT$ is the orbit consisting of cubics with three isolated $D_4$ singularities (a geometric condition that characterizes it, cf.~\cite[Thm. 5.4]{allcock}); one such cubic is given explicitly by the polynomial
\begin{equation}\label{eq:3D4}
  F_{3D_4}:=x_0x_1x_2+x_3^3+x_4^3,
\end{equation}
with zero set $V(F_{3D_4})$, which we will call the $3D_4$-cubic.
\item The curve $\calT\subset (\GIT -\calM^s)$ parameterizes strictly polystable  orbits given by polynomials of the  form
\begin{equation}\label{eq:2A5}
 F_{A,B}=Ax_2^3+x_0x_3^2+x_1^2x_4-x_0x_2x_4+Bx_1x_2x_3,
\end{equation}
with $A,B$ not simultaneously vanishing.
Specifically, one notes that the zero set $V(F_{A,B})$ is projectively equivalent to $V(F_{\lambda^2A,\lambda B})$ for any $\lambda\in\CC^*$. In fact, $V(F_{A,B})$ is projectively equivalent to $V(F_{A',B'})$ if and only if $A/B^2=A'/B'^2$.  Thus, $C:=4A/B^2$ can be taken as an affine parameter for the rational curve $\calT$.  The factor of $4$ is taken for numerical convenience:  if $C\notin\{0,1\}$, then the cubic $V(F_{A,B})$ has exactly two isolated $A_5$ singularities (a geometric condition that characterizes the cubics $V(F_{A,B})$, cf.~\cite[Thm. 5.7]{allcock}). If $C=0$ (equivalently $A=0$), then the cubic $V(F_{0,B})$ has in addition to the two $A_5$ singularities, an isolated $A_1$ singularity. Finally, if $C=1$ (e.g., $(A,B)=(1,-2)$), then the associated cubic $V(F_{1,-2})$ is the {\it chordal cubic}, i.e., the secant variety of the standard rational normal curve in $ \PP^4$ (which is singular precisely along the rational normal curve). Note that
\begin{equation}\label{eq:ch}
 F_{1,-2}=\det\begin{pmatrix} x_0&x_1&x_2\\ x_1&x_2&x_3\\ x_2&x_3&x_4\end{pmatrix},
\end{equation}
which makes the relationship to the standard rational normal curve in $\PP^4$ more transparent.
\end{enumerate}
\end{rem}

\subsection{The ball quotient model $\BG$} Looij\-enga--Swierstra~\cite{ls} and independently Allcock--Carlson--Toledo~\cite{act} have constructed a ball quotient model $\calB/\Gamma$, where $\calB$ is a $10$-dimensional complex ball, and $\Gamma$ is an arithmetic group acting on $\calB$, via the period map for cubic fourfolds. The following summarizes the essential aspects of the ball quotient model.
\begin{teo}[{The ball quotient model,~\cite{act} and~\cite{ls}}]
 Let $\calB/\Gamma$ be the ball quotient model of~\cite{act}. The following hold:
 \begin{itemize}
\item[(1)] The period map (defined via eigenperiods of cubic fourfolds)
 $$P:\calM\to \calB/\Gamma$$
 is an open embedding with the complement of the image being the union of two irreducible Heegner divisors $D_n:=\cDn/\Gamma$ (called the nodal divisor) and $D_h:=\cDh/\Gamma$ (called the  hyperelliptic divisor), where  $\cDn$ and $\cDh$ are $\Gamma$-invariant hyperplane arrangements.
 \item[(2)] The boundary of the Baily--Borel compactification $\BG$ consists of two cusps (i.e., $0$-dimensional boundary components), which we will call $c_{3D4}$ and $c_{2A_5}$.
 \end{itemize}
\end{teo}

The Baily--Borel compactification $\BG$ of the ball quotient model discussed above gives a projective compactification for the moduli space of cubic threefolds $\calM$. The main result of~\cite{act} and~\cite{ls} is that there is a simple birational relationship between the GIT and Baily--Borel models -- this is an essential result for our analysis. We summarize their results below:

\begin{teo}[{GIT to ball quotient comparison,~\cite{act} and~\cite{ls}}]\label{resgitball}
As above, let $\GIT$ be the GIT compactification of the moduli space of cubic threefolds. Let $\BG$ be the Baily--Borel compactification of the ball quotient model of~\cite{act}. Then there exists a diagram
$$
\xymatrix{
&\widehat{\calM} \ar[ld]_{p} \ar[rd]^{q}&\\
\ \ \GIT\ar@{-->}[rr]^{\overline P} &&  \BG \\
}
$$
resolving the birational map between $\GIT$ and $\BG$ such that:
\begin{itemize}
\item[(1)] $p:\widehat \calM\to \GIT$ is the Kirwan blowup of the point $\Xi\in \GIT$, corresponding to the chordal cubic (see \S\ref{S:KirBlUpDef} below, esp.~\eqref{diag_kirwanblowup}). The exceptional divisor $E:=p^{-1}(\Xi)$
 of this blowup is naturally identified with the moduli space of $12$ unordered points in $\PP^1$.
 \item[(2)] $q:\hM \to  \BG$ is a small semi-toric modification as constructed by Looijenga~\cite{l1}. The morphism $q$ is an isomorphism over the interior $\calB/\Gamma$ and one of the two cusps of  $\BG$, namely $c_{3D4}$.  The preimage under $q$ of the other cusp, $c_{2A5}$, is a curve, which is identified with the strict transform $\widehat {\calT}$ of $\calT\subset \GIT$ under $p$.
 \end{itemize}
In particular note that the period map $P:\calM\to \calB/\Gamma$  extends to a morphism $\overline P:\GIT -\lbrace \Xi\rbrace\to\BG$. Furthermore, the following hold:
 \begin{itemize}
  \item[(3)] Let $E\subset \hM$ be the exceptional divisor of the map $p$. Then the image $q(E)$ is the closure $D_h^*$ in $\BG$ of the hyperelliptic divisor $D_h\subset \mathcal B/\Gamma$,
  while $q$ is an isomorphism over $D_h$ (i.e., $q_{\mid q^{-1}(D_h)}:q^{-1}(D_h)\simeq D_h$).
 \item[(4)] $q$ is an isomorphism over the stable locus $\calM^s$ and in a neighborhood of the point $\Delta$, corresponding to the $3D_4$ cubic. The image under~$q$ of the locus of cubics with $A_1,\dots,A_4$-singularities is $(\cDn-\cDh)/\Gamma$  (equivalently,~$\overline P$ extends over~$\calM^s$ and $\overline P(\calM^s)=(\calB -\cDh)/\Gamma$).
 \item[(5)] $q$ maps $\Delta$ to the cusp $c_{3D4}$ of $\BG$, and the strict transform $\widehat {\calT}$ of the curve $\calT$ to the  cusp $c_{2A5}$.
  \end{itemize}
\end{teo}

\section{The Kirwan blowup $\MK$ of the moduli space of cubic threefolds}\label{subsec:kirwanth}

\subsection{Introduction}
The first step towards understanding the cohomology of the GIT and ball quotient models for the moduli of cubic threefolds is to produce a common resolution (with at worst finite quotient singularities). For GIT quotients, Kirwan~\cite{kirwanblowup} gives a general algorithm that achieves this resolution. Roughly speaking, one considers the GIT boundary $\GIT-\calM^s$($={\calT}\cup\{\Delta\}$ in our situation) and stratifies it in terms of the connected components $R$ of the stabilizers of the associated polystable orbits.
Then, one proceeds by blowing up these strata, starting with the deepest one, in a way that will be explained in detail below.
In our situation, we will see that there are three strata: $\Xi$,  $\Delta$ (which are points) and ${\calT}-\{\Xi\}$ (which is a curve), with associated connected components of the stabilizers being  $\SL(2,\CC)$, $(\CC^*)^2$, and $\CC^*$, respectively.

\subsection{The Kirwan blowup in general}\label{SSS:kirBlUp}

We start with $X$, $G$, and $L$ as in  the general setup of \S~\ref{SSS:KirSetUp}.
Let $\calR$ be a set of representatives for the (finite) set of conjugacy classes of connected components of stabilizers of strictly polystable points in $X^{ss}$. Denote then $r$ the maximal dimension of the groups in $\calR$, and let then $\calR(r)\subseteq \calR$ be the representative of those subgroups that have dimension $r$. For a given $R\in\calR(r)$, we proceed as follows. If $r=0$, then there is nothing to do. Otherwise, set
 \begin{equation}\label{E:ZRss}
Z_R^{ss}:=\{x\in X^{ss}\mid R \textrm{ fixes } x\} \subset X^{ss}.
\end{equation}
Kirwan shows that for all $R\in \calR(r)$, the loci $G\cdot Z^{ss}_R$ are smooth and closed in $X^{ss}$~\cite[Lem.~5.11, Cor.~5.10]{kirwanblowup}. Now let $\hat \pi:\hat X\to X^{ss}$ be the blowup of $X^{ss}$ along  $G\cdot Z^{ss}_R$.
Note that since $G\cdot Z^{ss}_R$ only depends on the conjugacy class of $R$, the same is true for the blowup.

As $G$ acts on the center of the blowup, there is an induced action of $G$ on $\hat X$. Taking $E$ to be the exceptional divisor of the blowup $\hat\pi$, there is a choice of $d\gg 0$ such that $\hat L:=\hat \pi^*L^{\otimes d}\otimes \calO(-E)$ is ample and admits a $G$-linearization that makes the following statements true~\cite[Lem.~3.11, Lem.~6.11]{kirwanblowup} (see also~\cite{reichstein}).  Let  $\hat {\calR}$ be a set of representatives for the set of conjugacy classes of connected components of stabilizers of polystable points in the semi-stable locus $\hat X^{ss}$.  Then, up to replacing elements of $\hat{\calR}$ with conjugates, we have  $\hat{\calR}\subsetneq \calR$~\cite[Lem.~6.1]{kirwanblowup}.

Thus, by induction on the cardinality of the set $\calR$, we obtain the desired space $\pi:\widetilde X^{ss}\to X^{ss}$ by iteratively blowing up with respect to a smooth center, and then restricting to the semi-stable locus.  Moreover, $\widetilde X^{ss}$ is equipped with a  $G$-linearized ample line bundle $\widetilde L$, such that $G$ acts with finite stabilizers.  We define the Kirwan blowup to be the space $\widetilde X^{ss}\gquot_{\widetilde L}G$ ($=\widetilde X^{ss}/G$); up to isomorphism, this is independent of the choices~\cite[Rem.~6.8 and p.64]{kirwanblowup}.  The Kirwan blowup has at worst finite quotient singularities, and there  is a birational morphism~\cite[Cor.~6.7]{kirwanblowup}:
$$
\widetilde X^{ss} \gquot_{\widetilde L}G\longrightarrow X^{ss}\gquot_{ L}G.
$$

\begin{rem}\label{R:Stab+Strict}
For later reference, we recall two further facts regarding the map $\hat \pi:\hat X\to X^{ss}$, and the chosen linearization.
First, if $\hat x\in \hat X - E$, then $\hat x\in \hat X^{ss}$ if and only if $\overline {G\cdot \hat \pi(\hat x)}\cap G\cdot Z_R^{ss}=\emptyset$~\cite{reichstein}.
In other words, outside of the exceptional divisor, the effect of the blowup is to destabilize exactly those strictly semi-stable points that have orbit closure meeting the center of the blowup.
Second, for any $\hat R\in \hat {\calR}$ the locus $\hat Z^{ss}_{\hat R}\subseteq \hat X^{ss}$ is the strict transform of the locus $ Z^{ss}_{\hat R}\subseteq X^{ss}$ defined by viewing~$\hat R$ as an element of~$\calR$~\cite[Rem.~6.8]{kirwanblowup}.
 \end{rem}

\begin{rem}\label{R:pi-r-Def}
We also recall the following fact~\cite[Lem.~8.2]{kirwanblowup}:  If $R_1, R_2\in \calR(r)$ are different groups of maximal dimension among elements of~$\calR$, then $G\cdot Z^{ss}_{R_1}\cap G Z^{ss}_{R_2}=\emptyset$, and any $x$ in $G\cdot Z^{ss}_{R_2}$  remains semi-stable after $X^{ss}$ is blown up along $G\cdot Z^{ss}_{R_1}$. In particular we have~\cite[Cor.~8.3]{kirwanblowup}: the result of successively blowing up $X^{ss}$  along $G\cdot Z^{ss}_R$ for each $R \in  \calR(r)$ is the same as the blowup of $X^{ss}$ along
$\bigcup _{R\in \calR(r)}G\cdot Z^{ss}_R $.  Following the notation in~\cite[Cor.~8.3]{kirwanblowup}, we will denote this blowup by $\pi_r:X_r\to X^{ss}$.
Repeating the above process we obtain a sequence of blowups
$$\widetilde X^{ss}:=X_1^{ss}\stackrel{\pi_1}{\longrightarrow} X_{2}^{ss}\stackrel{\pi_{2}}{\longrightarrow}\dots\stackrel{\pi_{r-1}}{\longrightarrow} X^{ss}_r\stackrel{\pi_r}{\longrightarrow} X^{ss}=:X_{r+1}^{ss}\,\,.$$
Note that we allow some of these blowups to be the identity if there are no relevant subgroups in a given dimension. In short, $\pi_{j}$ is the blowup of the locus determined by the subgroups $R\in \calR$ of dimension $j$; i.e., by all $R\in\calR(j)$.  Note that in contrast,  if  $R_1\in \calR(r_1)$ and $R_2\in\calR(r_2)$ for $r_1\ne r_2$, then it may happen that $G\cdot Z^{ss}_{R_1}\cap G Z^{ss}_{R_2}\ne \emptyset$.
\end{rem}

\subsection{The Kirwan blowup of the moduli space of cubic threefolds}\label{S:KirBlUpDef}
We now implement the steps outlined in the previous subsection to construct the Kirwan blowup of the moduli space of cubic threefolds.
The first step is to enumerate the connected components of the stabilizers of polystable points.  In our situation, this is answered by the following proposition, where as is standard, we write $1$-PS for one-parameter subgroups:

\begin{pro}[The connected components of stabilizers~$R$] \label{pro:stabilizer0}
Let $V$ be a strictly polystable cubic threefold. Then the connected component $\operatorname{Stab}^0(V)$ of the identity in the stabilizer $\operatorname{Stab}(V)\subseteq \SL(5,\CC)$ is one of the following (up to conjugation):
\begin{itemize}
\item[(1)] The $1$-PS with weights $(2,1,0,-1,-2)$:
\begin{equation}\label{E:R2A5}
R_{2A5}:=\operatorname{Stab}^0(V(F_{A,B}))=\operatorname{diag}(\lambda^2,\lambda,1,\lambda^{-1},\lambda^{-2})\cong \CC^*,
\end{equation}
for $4A/B^2\ne 1$.
We have $\operatorname{Stab}^0(V)=R_{2A5}$ (up to conjugation) if and only if $V$ is in the orbit of $V(F_{A,B})$ with $4A/B^2\ne 1$; i.e., if and only if the cubic has exactly two $A_5$ singularities, or exactly two $A_5$ singularities and one $A_1$ singularity.   These are the cubic threefolds corresponding to points on the curve $(\calT-\lbrace \Xi\rbrace)\, \subseteq \GIT$.

\item[(2)] The three-dimensional group
\begin{equation}\label{E:Rc}
R_c:=\operatorname{Stab}^0(V(F_{-1,2}))\cong\PGL(2,\CC),
\end{equation}
given as the copy of $\PGL(2,\CC)$ embedded into $\SL(5,\CC)$ as the image of  the $\SL(2,\CC)$ representation  $\Sym^4(\CC^2)\cong \CC^5$  (see Appendix~\ref{sec:equivcoh} for more details on dealing with equivariant cohomology of $\GL(n+1,\CC)$ versus $\SL(n+1,\CC)$, and related issues).
We have $\operatorname{Stab}^0(V)=R_{c}$ (up to conjugation) if and only if $V$ is in the orbit of $V(F_{A,B})$ with $4A/B^2= 1$; i.e., if and only if the cubic is projectively equivalent to the chordal cubic.   These are the cubic threefolds corresponding to  the point  $\,\Xi \in \GIT$.

\item[(3)] The two-dimensional torus:
\begin{equation}\label{E:R3D4}
R_{3D_4}:=\operatorname{Stab}^0(V(F_{3D_4}))=\diag(s,t,(st)^{-1},1,1)\cong (\CC^*)^2.
\end{equation}
 We have $\operatorname{Stab}^0(V)=R_{3D_4}$ (up to conjugation) if and only if $V$ is in the orbit of $V(F_{3D_4})$; i.e., if and only if the cubic has exactly $3D_4$ singularities.  These are the cubic threefolds corresponding to  the point  $\,\Delta \in \GIT$.
\end{itemize}
Moreover, we have
\begin{equation}\label{E:Rcont}
R_{2A_5}\subset R_c, \ \ R_{c}\cap R_{3D_4}=1,
\end{equation}
with the inclusion on the left corresponding to the fact that  $\Xi\in \calT\subset \GIT$.
\end{pro}
\begin{proof}
From the results of~\cite{allcock} describing polystable cubic threefolds (see Theorem~\ref{T:GITcub} and Remark~\ref{R:Alck-poly-form}, above), it suffices to consider the cubic threefolds of the form $V(F_{A,B})$~\eqref{eq:2A5}, for $A$ and $B$ not simultaneously zero, and $V(F_{3D_4})$~\eqref{eq:3D4}.
It is obvious that each of the groups listed above is connected and stabilizes the corresponding polystable orbit. For instance, $\PGL(2)$ acting on $\PP^4$ via the $\Sym^4$ representation fixes the standard rational normal curve.  Obviously, it will also fix the secant variety of that curve, which is precisely the chordal cubic.

The converse (i.e., the fact that $\operatorname{Stab}^0(V)$ is precisely as listed, and not larger) follows by a routine calculation.  Many straightforward computations with matrices will be relegated to Appendix~\ref{S:Elem}.  For the results here, see in particular  Proposition~\ref{P:App-R=SL2}, and Propositions~\ref{P:App-R=C*p1},\ref{P:App-R=C*p2},~\ref{P:App-R=C*2}.  The relationships~\eqref{E:Rcont}  among the $R$ are straightforward from the descriptions of the groups.
\end{proof}

Utilizing the notation from~\eqref{E:R2A5},~\eqref{E:Rc}, and~\eqref{E:R3D4}, it follows that for cubic threefolds we may take
\begin{equation}\label{def_calr}
\calR:=\{R_{2A_5},R_{3D_4},R_c\}\longleftrightarrow \{\CC^*, (\CC^*)^2,\PGL(2,\CC)\}
\end{equation}
as a set of representatives for the set of conjugacy classes of connected components of stabilizers of strictly  polystable cubic threefolds. For each $R\in \calR$, we have the corresponding fixed locus $Z^{ss}_R$, defined in~\eqref{E:ZRss}.
These loci can be described more explicitly:

\begin{pro}[The strata $Z_{R}^{ss}$]\label{P:ZRss}
For cubic threefolds, the fixed loci $Z_R^{ss}$~\eqref{E:ZRss} can be described as follows:
\begin{enumerate}
\item $Z^{ss}_{R_{2A_5}}$ is the set of cubic threefolds defined by the cubic  forms:
\begin{equation}\label{eq:ZR2A5}
F=a_0x_2^3+a_1x_0x_3^2+a_2x_1^2x_4+a_3x_0x_2x_4+a_4x_1x_2x_3,
\end{equation}
with  $a_1,a_2,a_3\ne 0$, $(a_0,a_4)\ne (0,0)$.   For $(A,B)\ne (0,0)$ we have $V(F_{A,B})\in  Z^{ss}_{R_{2A_5}}$, and conversely every cubic in $Z^{ss}_{R_{2A_5}}$ is  projectively equivalent to a cubic of the form $V(F_{A,B})$ with $(A,B)\ne (0,0)$.

\item $Z^{ss}_{R_c}=\{V(F_{1,-2})\}$, the chordal cubic in standard coordinates.

\item $Z^{ss}_{R_{3D_4}}$ is the set of cubics defined by equations of the form
$$
x_0x_1x_2+P_3(x_3,x_4)
$$
where $P_3(x_3,x_4)$ is an arbitrary homogeneous cubic with three distinct roots.
\end{enumerate}
Moreover,  we have the following relationships among the fixed loci:
\begin{equation}\label{E:ZssR-Rel}
Z^{ss}_{R_c}\subset Z^{ss}_{R_{2A5}}, \ \ \ \  Z^{ss}_{R_{2A_5}} \cap Z^{ss}_{R_{3D4}}=\emptyset.
\end{equation}
\end{pro}

\begin{proof}
It is immediate to check that the groups $R_{2A_5}$, $R_{3D_4}$, and $R_{c}$ fix the corresponding loci $Z^{ss}_{R_{2A_5}}$, $Z^{ss}_{R_{3D_4}}$, and $Z^{ss}_{R_{c}}$, respectively.   It is a straightforward check that these are in fact the full fixed loci; see also Propositions~\ref{P:App-R=C*p1},\ref{P:App-R=C*p2},\ref{P:App-R=SL2},~\ref{P:App-R=C*2}.
The relationships~\eqref{E:ZssR-Rel} among the $Z^{ss}_R$ are a straightforward consequence of the descriptions above.  See also Corollary~\ref{C:App-ZssR-rel}.
\end{proof}

For the Kirwan blowup, we are actually interested in the orbits
$$
G\cdot Z^{ss}_R;
$$
in other words the loci of cubic threefolds that are projectively equivalent to the cubics in a given stratum.

\begin{cor}[The orbits $G\cdot Z^{ss}_R$]\label{C:ZRss}
For cubic threefolds, the orbits of the fixed loci $Z_R^{ss}$   can be described as follows:
\begin{enumerate}
\item $G\cdot Z^{ss}_{R_{2A_5}}$ is the set of polystable cubics   projectively equivalent to a $2A_5$ cubic, a $2A_5+A_1$ cubic, or a chordal cubic; i.e., projectively equivalent to a cubic of the form $V(F_{A,B})$ with $(A,B)\ne (0,0)$.

\item $G\cdot Z^{ss}_{R_c}$ is the set of polystable cubics   projectively equivalent to the chordal cubic; i.e., projectively equivalent to $V(F_{1,-2})$.

\item $G\cdot Z^{ss}_{R_{3D_4}}$ is the set of polystable cubics with $3D_4$ singularities; i.e.,   projectively equivalent to $V(F_{3D_4})$.
\end{enumerate}
Moreover,
we have the following relationships among the orbits:
\begin{equation}\label{E:GZssR-Rel}
G\cdot Z^{ss}_{R_c}\subset G\cdot Z^{ss}_{R_{2A5}}, \ \ \ \   G\cdot Z^{ss}_{R_{2A_5}} \cap G\cdot  Z^{ss}_{R_{3D4}}=\emptyset.
\end{equation}
\end{cor}

\begin{proof} (1)--(3) follow directly from Proposition~\ref{P:ZRss}(1)--(3).
The first inclusion of~\eqref{E:GZssR-Rel} follows directly from that of~\eqref{E:ZssR-Rel}.    The equality on the right  follows from (1)--(3), since  the cubics in question are not projectively equivalent.
\end{proof}

Now recall that the  Kirwan desingularization process consists of successively blowing up $X^{ss}$ along the (strict transforms of the) loci $G\cdot Z_R^{ss}$  in order of $\dim R$, to obtain a smooth space $\widetilde X^{ss}$, and then taking the induced  GIT quotient  $\widetilde{X}^{ss}\gquot_{\widetilde L} G$ with respect to a particular linearization. We denote the resulting desingularization~$\MK$ and refer to it as {\it the Kirwan blowup} of $\GIT$.
Concretely, in our situation, this translates into a diagram:

\begin{equation}\label{diag_kirwanblowup}
\resizebox{\textwidth}{!}{
\xymatrix@C=1em{
\widetilde X^{ss}\ar@{=}[d]\\ (\operatorname{Bl}_{ {G\cdot Z^{ss}_{R_{2A_5},2}}}(X_2^{ss}))^{ss}\ar[r]\ar@{->}[d] & X_2^{ss}=(\operatorname{Bl}_{{G\cdot Z^{ss}_{R_{3D_4}}}}( X_3^{ss}))^{ss}\ar[r]\ar@{->}[d]& X^{ss}_3=(\operatorname{Bl}_{G\cdot Z^{ss}_{R_{c}}}(X^{ss}))^{ss}\ar[r] \ar@{->}[d]& X^{ss}\ar[d]\\
\MK \ar[r]& \widehat {\widehat {\calM}\,\,} \ar[r]& \widehat{\calM} \ar[r]&\GIT
}
}
\end{equation}

Here ${{G\cdot Z^{ss}_{R_{2A_5},2}}}$ is the strict transform of the orbit ${{G\cdot Z^{ss}_{R_{2A_5}}}}$.

The Kirwan blowup $\MK$ is obtained by first blowing up the point $\Xi\in \GIT$ corresponding to the chordal cubic, followed by blowing up the point $\Delta$ (which is not affected by the first blowup), and then finally blowing up the strict transform $\widehat {\calT}$ of ${\calT}\subset \GIT$.
To be precise, we must  specify the blowup ideals corresponding to the blowups on the lower line of \eqref{diag_kirwanblowup}. These are obtained by descent modulo the action of $G$ from $X^{ss}$ of the reduced ideals defining the blowup $\widetilde {X}^{ss}\to X^{ss}$.
Note that the last two blowups commute (thus their order is irrelevant). Also, the blowup of $\Xi$ (i.e., the first blowup) coincides
with the blowup $\widehat \calM$ constructed by Allcock--Carlson--Toledo~\cite{act} in order to resolve the birational period map $\overline P:\GIT\dashrightarrow \BG$ (i.e., the space discussed above in  Theorem~\ref{resgitball}).   Indeed, in the Kirwan blowup, in light of Corollary~\ref{C:ZRss}(2), the first step is to blowup $X^{ss}$ along the orbit of the chordal cubic, and then take the GIT quotient with respect to a particular linearization, which is exactly the construction in~\cite[\S 3]{act}.
The space $\widehat {\widehat {\calM}\,\,}$ is an auxiliary space from our perspective.

\section{The toroidal compactification}
As with any locally symmetric space, the ball quotient $\calB/\Gamma$ has not only the Baily--Borel compactification $\BG$, but also a toroidal compactification, which is thus another
natural birational model of $\calM$. While typically
the construction of toroidal compactifications depends on certain choices, this is not the case for ball quotients. Recall that the cusps are in $1:1$ correspondence
with $\Gamma$-orbits of rational isotropic subspaces of the vector space on which the group $\Gamma$ acts. Since ball quotients are related to hermitian forms
of signature $(1,n)$, the only possibility is given by isotropic lines. This means on the one hand that the Baily--Borel compactification $\BG$ is obtained from the ball quotient $\calB/\Gamma$
by adding finitely many (in our case -- two) points, that is $0$-dimensional cusps, as we have discussed above. On the other hand, from a toric point of view, we are in a $1$-dimensional situation, which
allows no choices. We shall denote the (unique) toroidal compactification by $\oBG$. It comes with a natural morphism $\oBG \to \BG$.
We shall discuss this in more detail in Chapter~\ref{sec:toroidal}.

In summary we have the following diagram illustrating the relationships among all the models of the moduli space of cubic threefolds we have discussed so far:
\begin{equation}\label{E:BirDiagMod}
\xymatrix{
&\MK\ar[ldd]_\pi\ar[d]^f\ar[rdd]^g\ar@{<-->}[r]^{}&\oBG\ar[dd]\\
&\widehat\calM\ar[ld]^{p}\ar[rd]_q\\
\GIT \ar@{-->}[rr]^{\overline P} &&\BG.
}
\end{equation}
While $\MK$ and $\oBG$ can both be viewed as blowups of the two points in $\BG$ corresponding to the two cusps of the Baily--Borel compactification,
we do not know whether the Kirwan blowup $\MK$ and the toroidal compactification $\oBG$ are isomorphic (see Remark~\ref{rem_possible_iso}). This seems to us an interesting question in its own right, which we plan to revisit in the future.

%% file: sec3_HXsemistable.tex
\chapter[Equivariant cohomology of the semi-stable locus]{The cohomology of the Kirwan blowup, part I: \\ equivariant cohomology of the semi-stable locus}\label{sec:HXss}
Following Kirwan, we will compute the intersection cohomology of the GIT quotient $\GIT$ by first computing the cohomology of the Kirwan blowup $\MK$.  The first step in computing the cohomology of the Kirwan blowup is to compute the equivariant cohomology of the semi-stable locus.  This is accomplished by constructing an equivariantly perfect stratification ~\cite[p.17]{kirwan84} of the unstable locus, and then using the Thom--Gysin sequence.  We review the precise setup in this section, and perform this step for the case of cubic threefolds.

\section[The equivariantly perfect stratification]{The equivariantly perfect stratification and the equivariant cohomology of the semi-stable locus in general}\label{S:PXss}

\subsection{Defining the equivariantly perfect stratification $S_\beta$}  \label{S:E-P-S-def}
We return to the general setup of \S~\ref{SSS:KirSetUp}, and
recall Kirwan's equivariantly perfect stratification of the unstable locus in $X$, which will allow us to compute the equivariant cohomology of the semi-stable locus.
Our presentation follows~\cite[Ch.8 \S 7]{GIT}, and serves primarily to fix notation.
In addition, one of the main  points of the review in this  section is that it is difficult to explain the terms in Kirwan's formulas in the case of cubic threefolds without describing the construction, and partially explaining the proofs.

To define the stratification we first  define an indexing set  $\calB$. This consists of the points in the closure $\overline{\mathfrak t}_+$ of the positive Weyl chamber that can be characterized as
follows: they are the closest point to the origin of the convex hull of a nonempty set of the weights $\alpha_0,\dots,\alpha_N$~\cite[Def.~3.13, and \S8 p.59]{kirwan84}.
Using the inner product on $\mathfrak t$ (fixed in \S~\ref{SSS:KirSetUp}), and the corresponding  norm $||\cdot ||$, we  define   for each $\beta \in \calB$~\cite[p.173]{GIT},~\cite[Exa.~3.11, Thm.~12.26]{kirwan84}:
\begin{align}
 \label{E:Zbeta} Z_\beta&:=\{(x_0:\dots :x_N)\in X\subseteq \PP^N: x_j=0 \text { if } \alpha_j.\beta \ne ||\beta ||^2\}\\
\label{E:Ybeta}Y_\beta &:=\{(x_0:\dots :x_N)\in X\subseteq \PP^N: x_j=0 \text { if } \alpha_j.\beta < ||\beta ||^2,\\ \nonumber 
&\ \ \ \ \ \ \ \ \ \  \text{ and  } \exists \ x_i \ne 0 \text{ s.t. } \alpha_i.\beta =||\beta ||^2\}.
\end{align}
Since $Z_\beta$ sits in projective space, for any point $(x_0:\dots :x_N)\in Z_\beta$ there exists some $x_i\ne 0$ with $\alpha_i.\beta =||\beta ||^2$.  Thus we have $Z_\beta\subseteq Y_\beta$,
and in fact there is a retraction $$p_\beta :Y_\beta \to Z_\beta$$ that sends $x_i$ to $0$ if $\alpha_i.\beta >||\beta ||^2$ (see~\cite[p.42, Def.~12.18]{kirwan84} and~\cite[p.173]{GIT}).

\begin{rem}  To get a geometric sense of the spaces $Z_\beta$ and $Y_\beta$, it can be helpful to consider the special case of hypersurfaces of degree $d$ in $\mathbb P^n$. 
This case is described in detail in~\S~\ref{Sec-Hyp-2}.
\end{rem}

For each $\beta\in \calB$ we set  $K_\beta$ to be the stabilizer of $\beta$ under the adjoint action of the maximal compact subgroup  $K$ on its Lie algebra  $\mathfrak k$ (recall $\beta \in \mathfrak t\subseteq \mathfrak k$)~\cite[Def.~4.8]{kirwan84},~\cite[p.169]{GIT}.   There is an action of $K_\beta$ on $Z_\beta$~\cite[p.25]{kirwan84}, and a particular linearization of the action of the complexification of $K_\beta$ on $Z_\beta$ that is  defined in~\cite[\S 8.11]{kirwan84}, and with respect to which we obtain a semi-stable locus $Z^{ss}_\beta$.  One defines~\cite[p.173]{GIT},~\cite[(11.2), Def.~12.20]{kirwan84}:
\begin{align}
Y_\beta^{ss}&:=p_\beta ^{-1}(Z_\beta^{ss})\\
\label{E:SbetaDef} S_\beta &:= G\cdot Y_\beta ^{ss}.
\end{align}
It is a fact  that
\begin{equation}\label{E:GbPbYb}
S_\beta  \cong G\times_{P_\beta }Y_{\beta}^{ss}
\end{equation}
where $P_\beta$ is the parabolic subgroup of $G$ that is the product of the stabilizer $K_\beta$ and the Borel subgroup $B$ associated to the choice of $T$ and $\mathfrak t^+$~\cite[p.173]{GIT},~\cite[Lem.~6.9 and \S 12]{kirwan84}.
 In fact, the parabolic subgroup  $P_\beta$ can also be described as the subgroup of $G$ that preserves $Y_\beta^{ss}$~\cite[Lem.~13.4]{kirwan84}.

An equivalent algebraic definition of $Z_\beta^{ss}$, and hence of  $Y_\beta^{ss}$ and   $S_\beta$, is given in~\cite[Def.~12.8, Def.~12.14, Def.~12.20]{kirwan84}.  For any $x=(x_0:\dots:x_N)\in X\subseteq \PP^N$, we denote by $C(x)\subseteq \mathfrak t$ the convex hull of the collection of weights $\alpha_i$ such that  $x_i\ne 0$; we define $\beta(x)$ to be the closest point to the origin in $C(x)$.   Then for $\beta\ne 0$ we have the following description, summarizing and slightly rephrasing the discussion of~\cite[\S 12]{kirwan84}:
\begin{align}\label{E:HesseZss}
Z_\beta^{ss}&=\left\{x\in Z_\beta :   \beta(x)=\beta,    \text{ and for all } g \in G,     \ ||\beta(gx)||\le  ||\beta|| \right\}.
\end{align}

We will also use the fact  that~\cite[Lem.~12.13]{kirwan84}:
\begin{equation}\label{E:S0P0}
S_0=X^{ss}
\text{ and }
P_0=G.
\end{equation}
Finally it is shown in~\cite[Lem.~12.15, 12.16]{kirwan84} that the $S_\beta$ define a  $G$-equivariant stratification
\begin{equation}\label{E:SbStrat}
X=\bigsqcup_{\beta\in\calB}S_\beta=X^{ss}\sqcup \bigsqcup_{0\ne \beta\in\calB}S_\beta.
\end{equation}

 We end by observing that one can use~\eqref{E:GbPbYb} to conclude that, if nonempty,  $S_\beta$ has  dimension
\begin{equation}\label{E:DimSb}
\dim S_\beta = \dim G/P_\beta+\dim Y_\beta.
\end{equation}
We call the right hand side of~\eqref{E:DimSb} the expected dimension of $S_\beta$, and denote this as $\dim_{\operatorname{exp}} S_\beta$.

\begin{rem}\label{R:SbPOSET}
We order the strata $S_\beta$  as a  POSET in the usual way, via inclusions of  closures; i.e., $S_{\beta'}\le S_{\beta }$ if $\overline S_{\beta'} \subseteq \overline S_{\beta}$.
The maximal stratum is $S_0=X^{ss}$, if it is nonempty.  More generally, we can make a POSET out of $\calB$ by setting $\beta'\le \beta$ if $\overline Y_{\beta'}\subseteq \overline Y_{\beta}$, and then if $S_{\beta }$ is nonempty, the inequality  $\beta'<\beta$ implies $S_{\beta'} <S_{\beta }$.   Indeed, $S_{\beta}$ nonempty implies that $Y^{ss}_{\beta }$ is a dense open subset of $\overline Y_{\beta}$,   and consequently $\overline Y_{\beta'}\subseteq \overline {Y_{\beta}^{ss}}$, so  that $\overline S_{\beta'} \subseteq  \overline S_{\beta}$.   Note also that if $\beta '<\beta$, then since $S_{\beta '}\subseteq G \overline Y_{\beta '}\subseteq G\cdot \overline Y_\beta$, we have $\dim S_{\beta'}\le \dim G/P_\beta+\dim Y_\beta$.    In other words, we can say that if $\beta '<\beta$, then $\dim_{\operatorname{exp}}S_{\beta '} \le \dim_{\operatorname{exp}}S_{\beta}$, and if $S_{\beta}$ is nonempty, the inequality is strict.
\end{rem}

\subsection{Equivariant cohomology of the semi-stable locus}
The Thom--Gysin sequence relating the cohomology of a manifold $Y$, a closed submanifold $Z$ of $Y$, and its complement $Y-Z$, is a long exact sequence of the form
$$
\dots \to H^{i-d}(Z;\QQ)\to H^i(Y;\QQ)\to H^i(Y-Z;\QQ)\to H^{i+1-d}(Z;\QQ)\to \dots
$$
where $d$ is the codimension of $Z$ in $Y$.  The existence of such a sequence implies the following identity for Poincar\'e polynomials:
$$
t^dP_t(Z)-P_t(Y)+P_t(Y-Z)=(1+t)Q(t)
$$
where $Q(t)\in \QQ[t]$ has nonnegative coefficients.
Applying this to the stratification~\eqref{E:SbStrat} we obtain the following identities for Poincar\'e polynomials and  equivariant Poincar\'e polynomials (i.e., the Poincar\'e polynomials for equivariant cohomology):
\begin{align*}
P_t(X)&=\sum_{\beta}t^{2d(\beta)}P_t(S_{\beta})-(1+t)Q(t),\\
P^G_t(X)&=\sum_{\beta}t^{2d(\beta)}P_t^G(S_{\beta})-(1+t)Q^G(t),
\end{align*}
where the polynomials $Q(t),Q^G(t)\in \QQ[t]$  have nonnegative coefficients, and
\begin{equation}\label{E:d(beta)Def}
d(\beta):=\operatorname{codim}_{\CC}S_\beta =\dim X-(\dim G-\dim P_\beta+\dim Y_{\beta}),
\end{equation}
where the equality on the right holds provided $S_\beta$ is nonempty.
One then shows that the stratification is $G$-equivariantly perfect~\cite[p.17]{kirwan84}, implying that $Q^G(t)=0$, so that we have:
\begin{align}\label{E:KD-(3.1)-0}
P^G_t(X)&=\sum_{\beta}t^{2d(\beta)}P_t^G(S_{\beta}).
\end{align}
The key point in showing that the stratification is equivariantly perfect is to consider a degenerate Morse function $f:X\to \RR$ given as the composition of the induced moment map $\mu:X\to \mathfrak k^\vee$~\cite[(2.7)]{kirwan84}, with the modulus $||-||:\mathfrak k^\vee \to \RR$, induced by the Killing form.    The strata $S_\beta$, $Y_\beta$, and $Z_\beta$ then have interpretations with respect to the gradient flow to the  critical sets for $f$~\cite[Thm.~12.26]{kirwan84}, and one then uses techniques from Morse theory and symplectic geometry to establish that the stratification is equivariantly perfect~\cite[Thm.~6.18]{kirwan84}.

Finally we observe that equation~\eqref{E:KD-(3.1)-0} can be rewritten as~\cite[Eq.~3.1]{kirwanhyp}:
\begin{equation}\label{E:KD-(3.1)}
P_t^G(X^{ss})=P_t(X)P_t(BG)-\sum_{0\neq \beta \in \calB} t^{2d(\beta)}P_t^G(S_{\beta}),
\end{equation}
using~\eqref{E:S0P0}, and a result of Kirwan~\cite[Prop.~5.8]{kirwan84} on equivariant cohomology with respect to compact Lie groups acting on symplectic manifolds   (see  formulas~\eqref{E:K-EC-1} and~\eqref{E:AS-EC-1}),  to write $P_t^G(X)=P_t(X)P_t(BG)$.  Note that if $S_\beta$ is empty, our convention is that $t^{2d(\beta)}P_t^G(S_\beta)=0$.

\begin{rem}\label{R:PD-deg}  The computation of $P_t^G(X^{ss})$ is
an intermediate step in computing the intersection cohomology of the GIT quotient $X/\!\!/_{\calO(1)}G$, and the cohomology of the Kirwan blowup.  Both of these cohomology theories satisfy Poincar\'e duality, and therefore in these applications it suffices to compute $P_t^G(X^{ss})$ up to degree equal to the complex dimension of the GIT quotient.
Thus, estimating the dimensions  via~\eqref{E:DimSb}, one may in some cases ignore many if not all  of the  strata $S_\beta$ in~\eqref{E:KD-(3.1)}.
\end{rem}

\section[The equivariant cohomology of the semi-stable locus]{The equivariant cohomology of the locus of semi-stable cubic threefolds} \label{S:ECSScubics}

\subsection{Some observations for hypersurfaces}
\label{Sec-Hyp-2}

Before moving to the case of cubic threefolds, we start by making a few observations that hold for all hypersurfaces.
  We continue using the notation from  \S \ref{S:hypeKirConv}.  Recall that $X=\PP\operatorname{Sym}^d(\CC^{n+1})^\vee$ is the Hilbert space of hypersurfaces of degree $d$ in $\mathbb P^n$, we have identified the Lie algebra of the maximal torus of $\operatorname{SU}(n+1)$  as $\mathfrak t=\{(a_0,\dots,a_n)\in \RR^{n+1}: \sum a_i=0\} \subseteq \RR^{n+1}$, the inner product on $\mathfrak t$ is taken to be the standard inner product, and we make the identification of monomials of degree $d$ with weights in $\mathfrak t\subseteq \RR^{n+1}$ via the assignment $x^I\mapsto \alpha_I:= (i_0-d/(n+1),\dots,i_n-d/(n+1))$.  The Weyl group of $\operatorname{SU}(n+1)$ is the symmetric group  $S_{n+1}$ acting on $\mathfrak t$ by its generators, the reflections in the coordinate hyperplanes. 
 The indexing set
$\calB\subseteq \overline{\mathfrak t}_+$ consists of the points in $\overline{\mathfrak t}_+$ that can be described as the closest point to the origin of the convex hull of a nonempty set of the weights $\alpha_0,\dots,\alpha_N$, which  themselves can be viewed as lattice points in a simplex.

The sets $Z_\beta$ and $Y_\beta$ are defined as sets of polynomials (up to scaling) where only certain monomials are allowed to appear with non-zero coefficients.    More precisely, $Z_\beta$ is the linear subspace of $\PP\operatorname{Sym}^d(\CC^{n+1})^\vee$ determined by the vanishing of the coefficients of the monomials $x^I$ whose weight $\alpha_I$ does not lie in the affine space orthogonal to $\beta$ (i.e., the coefficient of $x^I$ is zero if $\alpha_I.\beta\ne ||\beta^2||$), and $Y_\beta$ is an open subset of the linear subspace of $\PP^N$ determined by the vanishing of the coefficients of the monomials $x^I$ whose weight does not lie on the positive side of the affine space orthogonal to $\beta$.
Said another way, $Z_\beta$ is the linear span of the monomials $x^I$ with weights $\alpha_I$ lying in the affine space orthogonal to $\beta$ (i.e., the span of the monomials $x^I$ with $\alpha_I.\beta= ||\beta^2||$), and $Y_\beta$ is the set of polynomials that are linear combinations of the $x^I$ with weights $\alpha_I$ lying on the non-negative side of the affine space orthogonal to $\beta$, and have at least one  monomial $x^I$ appearing with non-zero coefficient that has weight $\alpha_I$ lying in the affine space orthogonal to $\beta$. A similar description of $Z^{ss}_\beta$ follows from~\eqref{E:HesseZss}.

 We observe also that for hypersurfaces, from the definition of a parabolic subgroup,
 it follows that the parabolic subgroup  $P_\beta$ can  be described as the subgroup of $G$ that preserves the linear subspace of $\PP\operatorname{Sym}^d(\CC^{n+1})^\vee$  that is the closure of  $Y_\beta$; this can make the explicit computation of $P_\beta$ easier.

 We define $d(\beta):=\operatorname{codim}_{\CC}S_\beta$,  so that if $S_\beta$ is nonempty, we obtain the convenient combinatorial dimension count from~\eqref{E:DimSb}:
\begin{align}\label{E:Sbcodim}
d(\beta) &=n(\beta)-\dim G/P_\beta
\end{align}
where
$n(\beta)=\dim \PP\operatorname{Sym}^d(\CC^{n+1})^\vee-\dim Y_\beta$ is    the number  of weights $\alpha_I$ such that
$\beta.\alpha_I <||\beta ||^2$~\cite[p.47]{kirwanhyp}; i.e., the number of weights lying on the negative side of the affine space orthogonal to $\beta$.  In other words, the expected codimension of $S_\beta$ is $d_{\operatorname{exp}}(\beta)=n(\beta)-\dim G/P_\beta $.

\begin{rem}[Estimating $\dim P_\beta$]\label{R:dimPbbest}
Clearly a key point is to estimate the dimension of $P_\beta$.
To this end, recall that if there is a decomposition of the vector space  $\CC^{n+1}=W\oplus W'$, and a parabolic subgroup $P$ of $\SL(n+1,\CC)$ contains the subgroup $\SL(W)\oplus \operatorname{Id}_{W'}$,
 then the flag associated to $P$ has as its smallest vector space a vector space of dimension at least $\dim W$.   In other words, in appropriate bases, $P$ must be block upper-diagonal with a block of size at least  $\dim W$ so that the dimension of $P$ must be at least $\binom{\dim W}{2}$ more than the dimension of the Borel subgroup of upper triangular matrices in $\SL(n+1,\CC)$ (which has dimension $\binom{n+2}{2}-1$).
\end{rem}

\subsection{The case of cubic threefolds}
We now implement all this in the case of cubic threefolds.  For dimension estimates as in Remark~\ref{R:dimPbbest}, note that
the Borel subgroup of upper triangular matrices has dimension $14$ in this case.
\begin{pro}[Equivariant cohomology of the semi-stable locus]\label{P:PtGXss}
For the moduli of cubic threefolds, the only unstable stratum $S_\beta$ that contributes to formula~\eqref{E:KD-(3.1)}, modulo $t^{11}$, is the complex codimension $5$ stratum corresponding to general  $D_5$ cubics
 (corresponding to the case (b) in~\cite[Lem.~3.1]{allcock}), which only contributes its equivariant $H^0$, so that finally
\begin{equation}\label{eq:cohxss_our}\begin{aligned}
P_t^G(X^{ss})&\equiv 1+t^2+2t^4+3t^6+5t^8+6t^{10}  &\mod t^{11}.
\end{aligned}
\end{equation}
\end{pro}
\begin{proof}
We are claiming that for the moduli of cubic threefolds, the only unstable stratum $S_\beta$ that contributes to formula~\eqref{E:KD-(3.1)}, modulo $t^{11}$, is the complex codimension $5$ stratum corresponding to general  $D_5$ cubics
 (corresponding to the case (b) in~\cite[Lem.~3.1]{allcock}), which only contributes its equivariant $H^0$, so that we have 
\begin{equation*}
\begin{aligned}
P_t^G(X^{ss})&\equiv P_t(X)P_t(B\SL(5,\CC))-t^{10}&\mod t^{11}\, \\
&\equiv (1-t^2)^{-1}(1-t^4)^{-1}(1-t^6)^{-1} (1-t^{8})^{-1}(1-t^{10})^{-1}-t^{10} &\mod t^{11}\,\\
&\equiv 1+t^2+2t^4+3t^6+5t^8+6t^{10}  &\mod t^{11}.
\end{aligned}
\end{equation*}

We now explain this.  To begin, recalling that
$P_t(X)=P_t(\PP^{34})\equiv (1-t^2)^{-1}\mod t^{11}$ and that   $ P_t(B\SL(5,\CC)) =  (1-t^4)^{-1}(1-t^6)^{-1} (1-t^{8})^{-1}(1-t^{10})^{-1} $ (e.g., Example~\eqref{E:BSUn}),
we can write~\eqref{E:KD-(3.1)} as
\begin{equation}\label{eq:cohxssProp}
P_t^G(X^{ss})\equiv (1-t^2)^{-1}(1-t^4)^{-1} \dots (1-t^{10})^{-1} -\sum_{0\neq \beta \in \calB} t^{2d(\beta)}P_t^G(S_{\beta})
\ \ \mod t^{11}.
\end{equation}
We will show that  the strata $S_\beta$, $\beta\ne 0$,  have complex codimension $d(\beta)$ at least $5$, and that there is exactly one stratum of complex codimension $5$.  This stratum can then  only contribute its equivariant $H^0$, which we will see is  $1$-dimensional, completing the proof.

The basic tool we will use is the dimension count for the $S_\beta$ given in~\eqref{E:Sbcodim}, and for  convenience we rewrite  this with the specific numerics we have here.  If we set $r(\beta)$ to be the number of weights $\alpha$ such that $\beta .\alpha \ge ||\beta||^2$, and set $p(\beta )=\dim P_\beta$, then it follows from~\eqref{E:Sbcodim}  that if $S_\beta$ is nonempty, then
\begin{align}\label{E:dbetacubic}
d(\beta)=(35 -r(\beta))-\dim G+p(\beta)&=11+p(\beta)-r(\beta)\\
&\ge 25-r(\beta). \label{E:dbetacubic>}
\end{align}
From~\eqref{E:dbetacubic>}, if $r(\beta)< 20$, then $d(\beta)>5$, so that $S_\beta$ cannot contribute to~\eqref{eq:cohxssProp}.
As before, we call the right hand side of~\eqref{E:dbetacubic} the expected codimension of $S_\beta$, i.e., the codimension of $S_\beta$, provided it is nonempty, and denote it by $d_{\operatorname{exp}}(\beta)$.

For our analysis, we will proceed to estimate $d_{\operatorname{exp}}(\beta)$, starting from the maximal $\beta$;  i.e., we partially order the elements of $\calB$ by setting  $\beta'\le \beta $ if $\overline Y_{\beta'}\subseteq \overline Y_\beta$, and  start  with the  (possibly empty)  strata  $S_\beta$, $\beta \ne 0$,  such that the associated linear spaces $\overline  Y_\beta$ are maximal (among the $\overline Y_\beta$ with $\beta \ne 0$); see also Remark~\ref{R:SbPOSET}.    These   maximal  $\overline  Y_\beta$  can  be described  as maximal linear spaces spanned by monomials destabilized by some $1$-PS,   and are classified by
Allcock     in~\cite[Lem.~3.1]{allcock}.  In terms of Allcock's notation, $r(\beta)$ is the  number  of black dots in the corresponding diagram in~\cite[Lem.~3.1]{allcock}, and the linear space $\overline  Y_{\beta}$ is given by the span of the monomials corresponding to those black dots. We now compute the  expected codimension $d_{\operatorname{exp}}(\beta)$   for all the cases in~\cite[Lem.~3.1]{allcock}, enumerating in the same way as in the reference:

\begin{enumerate}
\item[(a)] Let $S_\beta$ correspond to~\cite[Fig.~3.1(a)]{allcock}. One computes the number of black dots in~\cite[Fig.~3.1(a)]{allcock}   to be $r(\beta) = 21$. For the dimension of the parabolic subgroup, it is also easy to see from
\cite[Fig.~3.1(a)]{allcock} that
if we write $(\CC^5)^\vee=\CC\langle x_1,x_2,x_3\rangle\oplus \CC\langle x_0,x_4\rangle$, then
 the parabolic subgroup $P_\beta$ contains the subgroup $\SL(\CC\langle x_1,x_2,x_3\rangle)\oplus \operatorname{Id}_{\CC\langle x_0,x_4\rangle}$, so that $\dim P_\beta$ is at least $14+\binom{3}{2}=17$ (Remark~\ref{R:dimPbbest}).
 Thus from~\eqref{E:dbetacubic} we have that   $d_{\operatorname{exp}}(\beta)\ge 11+17-21=7$.

\item[(b)]  Now let $S_\beta$ correspond to~\cite[Fig.~3.1(b)]{allcock}.  Here one computes $r = 21$, while the parabolic subgroup can permute $x_3$ and $x_4$, and thus has a $2\times 2$ block on the diagonal, so that $p\ge 14+1=15$.  Thus  $d_{\operatorname{exp}}(\beta)\ge 5$.

\item[(c)] Here $r = 20$, while the parabolic subgroup $P_\beta$ contains the subgroup $\SL(\CC\langle x_2,x_3\rangle) \oplus \operatorname{Id}_{\CC\langle x_0,x_1,x_4\rangle}$.  Thus we have $p\ge 15$,
so that  from~\eqref{E:dbetacubic} we have $d_{\operatorname{exp}}(\beta)\ge 6$.

\item[(d)]  Here $r= 18$,  so that as pointed out above in~\eqref{E:dbetacubic>}, the minimal  estimate  $p\ge 14$ suffices to give $d_{\operatorname{exp}}(\beta)\ge 7$.

\item[(e)] Here $r=22$, while the parabolic subgroup can permute $x_2,x_3,x_4$, so that $p\ge 17$, and thus $d_{\operatorname{exp}}(\beta)\ge 6$.

\item[(f)]  Here $r = 19$, so again we can use the minimal estimate $d_{\operatorname{exp}}(\beta)\ge 6$ in~\eqref{E:dbetacubic>}.  (Considering the parabolic subgroup more carefully, one can see that  one can permute $x_1$ and $x_2$, so that $p\ge 15$ and $d_{\operatorname{exp}}(\beta )\ge 7$.)
\end{enumerate}
We now turn our attention to the expected codimension of the strata $S_{\beta '}$, $\beta'\ne 0$, that do not arise in the list above.    The point is that for any $0\ne \beta'\in \calB$, we have $\beta '\le \beta $ for one of the $\beta$ in the list above, and
if $\beta'<\beta$, then $d(\beta' )\ge d_{\operatorname{exp}}(\beta')\ge d_{\operatorname{exp}}(\beta)$ (see, e.g.,  Remark~\ref{R:SbPOSET}).

For clarity, we summarize what we have shown, and what we will prove to establish the proposition:
\begin{enumerate}
\item For any $0\ne \beta'\in \calB$ we have shown that $d(\beta')\ge 6$, unless $\beta'\le \beta$ with $\beta$ as in case  (b) above, in which case we have $d(\beta')\ge 5$.

\item We claim that there is exactly one $\beta '\le \beta$ with $\beta$ as in case (b) above such that $d(\beta')=5$.

\item For the stratum $S_{\beta'}$ in  (2), we claim that the general point corresponds to a $D_5$ cubic,  and that $\dim H^0_G(S_{\beta'})=1$.
\end{enumerate}

If $S_\beta$ were known to be non-empty in case (b), then (2) would follow trivially from (1).  We find it is easier to establish the weaker statement (2) directly and  argue  geometrically.
Let $\beta$ be as in~\cite[Fig.~3.1(b)]{allcock}. This corresponds to the case \cite[Thm.~3.3(iii)]{allcock}: the cubic  threefold contains  a singularity of nullity $2$ and Milnor number $\ge 5$. In others words, the cubic has a double point whose projectivized tangent cone (a quadric) has corank $2$. This excludes the possibility of $A_k$ singularities (as they have corank $1$). Since the Milnor number is at least $5$, we also exclude the $D_4$ case, leading to $D_5$ singularities or worse.  In fact, it is immediate to see\footnote{Namely, this case corresponds to singularities of cubic threefolds with an affine equation $x_3^2+x_4^2+x_1x_2^2+(h.o.t.)$, where higher order terms is with respect to weights $(1/3,1/3,1/2,1/2)$ (see \cite[p. 216]{allcock}). We then note that $x_3^2+x_4^2+x_1x_2^2+2x_3x_1^2$ is analytically equivalent to $(x_3')^2+x_4^2+x_1x_2^2-x_1^4$, which is precisely the normal form for $D_5$ in $4$ variables.} that a generic point of $\overline Y_\beta$ (corresponding to a  generic linear combination of the given set of monomials) gives a cubic threefold with an isolated  $D_5$ singularity.  Moreover, $G \cdot\overline Y_\beta$ contains an open subset of  the cubics with an isolated $D_5$ singularity; i.e., a general cubic with an isolated $D_5$ singularity is projectively equivalent to one in $\overline Y_\beta$.     Using the versal deformation space of a $D_5$ singularity  and checking that locally around a $D_5$ cubic the space of cubics maps surjectively onto the versal deformation space (see e.g.,~\cite[Fact 3.12]{cubics}),
 it follows that  the locus $G\cdot \overline Y_\beta$  has complex codimension $5$ in the space of cubics.  Thus there is a codimension $5$ locus in the non-semi-stable locus $X-X^{ss}$.
 As the $S_\beta$ with $\beta\ne 0$, stratify the non-semi-stable locus,
and have codimension at most $5$, a Zariski open subset of $G\cdot\overline Y_\beta$  must be an open subset in
some stratum, say $S_{\beta'}$.   For dimension reasons, the only possibility is $\beta'\le \beta$ with $\beta$ as in (b) in the list above.  For this stratum we have $d(\beta')=5$,    and clearly since $G\cdot\overline Y_\beta$ is connected of codimension $5$ there can be no other stratum $S_{\beta ''}$, $\beta''\ne \beta'$,  with $d(\beta'')=5$.

The only thing left to show is (3), that $\dim H^0_G(S_{\beta'})=1$.
We have that $Y_{\beta'}^{ss}$ is connected, being a Zariski open subset of a projective space.  Consequently,  $S_{\beta'} =G\cdot Y_{\beta'}^{ss}$ is connected.  Setting $EG\to BG$ to be the universal principal bundle,
   since
$S_{\beta'} \times _G EG $ is the quotient of the connected space    $S_{\beta'} \times EG\sim_{\text{hom}}S_{\beta'}$, and is therefore itself connected, it follows finally that
$\dim H^0_G(S_{\beta'})=1$.
\end{proof}

%% file: sec4_HKirwanblowup.tex
\chapter{The cohomology of the Kirwan blowup, part II}\label{S:CohKirBl-II}

The next, and final, step in computing the cohomology of the Kirwan blowup is to compute some ``correction'' terms arising from the blowups.  The key point is that there is an equivariant version of the formula for the cohomology of a blowup, which inductively reduces the problem to the setup of the previous section, namely, computing the equivariant cohomology of semi-stable loci.  The subtle point is that these computations all reduce to computations on the exceptional divisors, and then with some more work, to computations on a general normal fiber to each stratum that is blown up.  This allows one to compute everything essentially on the original space, making the process feasible in examples. We start by reviewing the general setup, and then specialize to the case of cubic threefolds.  One of the main  points of the review in this  section is that it is difficult to explain the terms in Kirwan's formulas in the case of cubic threefolds without describing the construction, and partially explaining the proofs.

\section{The correction terms in general}

\subsection{The correction terms for a single blowup}\label{SSS:CorTermsI}
It is notationally much easier to explain the correction terms after a single blowup.  We start with this case, and then in the next subsection explain what the formulas are for multiple blowups.

We start here in the situation of \S~\ref{SSS:kirBlUp}, where we have fixed a maximal dimensional connected component $R\in \calR$ of the stabilizer of a strictly polystable point, taken the blowup
\begin{equation}\label{E:PiHatDef-4.1}
\hat \pi:\hat X\to X^{ss}
\end{equation}
along the locus $G\cdot Z^{ss}_R$~\eqref{E:ZRss}, and chosen a linearization of the action on an ample line bundle $\hat L$ on $\hat X$,  as described in \S~\ref{SSS:kirBlUp}.
For simplicity, we further assume that $Z_R^{ss}$ is connected.

The first observation is that  from the standard argument about cohomology of blowups of smooth loci~\cite[p.605]{GrHa78}, adapted to the $G$-equivariant setting, one has~\cite[p.67]{kirwanblowup}:
\begin{equation}\label{E:KirYXss}
P_t^G(\hat X)=P^G_t(X^{ss})+P^G_t(\PP\calN)-P^G_t(G\cdot Z_R^{ss})
\end{equation}
where $\PP\calN$ is the projectivization of
\begin{equation}\label{E:NNbundDef}
\calN:=\text{the normal bundle to the orbit } G\cdot Z^{ss}_R
\end{equation}
 (i.e., $\PP\calN$ is the exceptional divisor).
Next, the standard Leray spectral sequence argument for the cohomology of a projective bundle~\cite[Prop.~p.606]{GrHa78}, adapted to the $G$-equivariant setting, gives
 $P_t^G(\PP \calN)=P^G_t(G\cdot Z_R^{ss})P_t(\PP^{\operatorname{rk}(\calN)-1})$~\cite[p.67]{kirwanblowup}, so that we have
~\cite[Lem.~7.2]{kirwanblowup}
\footnote{There is a typo in~\cite[Lem.~7.2]{kirwanblowup}: $P_t^G(Z_R^{ss})$ should be $P_t^G(G\cdot Z_R^{ss})$.}
\begin{align}\label{E:KirL7.2}
P_t^G(\hat X)=P^G_t(X^{ss})+P_t^G(G\cdot Z_R^{ss})(t^2+\dots +t^{2(\operatorname{rk}\calN-1)})
\end{align}
where $\operatorname{rk}\calN$ is the complex rank of the vector bundle, which is equal to the codimension of $G\cdot Z_R^{ss}$.

Using~\eqref{E:KD-(3.1)} applied to $\hat X$, with indexing set $\calB_{\hat X}$, and strata $S_{\hat X,\hat \beta}$,
and substituting
 into~\eqref{E:KirL7.2}, we obtain
\begin{align}
P_t^G(\hat X^{ss})&=P^G_t(X^{ss})+P_t^G(G\cdot Z_R^{ss})(t^2+\dots +t^{2(\operatorname{rk}\calN-1)})-\!\!\! \sum_{0\ne \hat \beta \in \calB_{\hat X}}\!\!t^{2d(\hat X,\hat \beta)}P^G_t(S_{\hat X,\hat \beta}),
\end{align}
where the  complex codimension of the stratum $S_{\hat X,\hat \beta}$ in $\hat X$
is given by $d(\hat X,\hat \beta)$  as defined in~\eqref{E:d(beta)Def}, but now applied to the blowup $\hat X$.
Now let
\begin{align}
\label{E:NDef} N&:=N(R)\\
\label{E:StabbDef} \operatorname{Stab}_G\hat \beta&:=G_{\hat \beta}
\end{align}
be the normalizer of $R$ in $G$, and the stabilizer of $\hat \beta$ in $G$ under the adjoint action, respectively.
Then $G\cdot Z^{ss}_R=G\times_N Z^{ss}_R$, and $S_{\hat X,\beta}=G\times_{N\cap \operatorname{Stab}_G\hat \beta }(Z_{\hat \beta}^{ss}\cap \hat \pi^{-1}Z^{ss}_R)$~\cite[p.72]{kirwanblowup}, where $Z_{\hat \beta}^{ss}\subseteq \hat X$ is defined as in~\eqref{E:HesseZss}.    Therefore we obtain (see~\eqref{E:AB-EC-4})
\cite[(7.13)]{kirwanblowup}
\footnote{There are two typos in \cite[(7.13)]{kirwanblowup}: the formula is missing a $-P_t^N(Z^{ss}_R)$,  and the sign before the sum in the formula should be negative.  The former is noted in \cite[p.50]{kirwanhyp}, while the latter is essentially noted in \cite[(3.4)]{kirwanhyp}.}
\begin{align}
P_t^G(\hat X^{ss})&=P^G_t(X^{ss})\\
&+P_t^N(Z_R^{ss})(t^2+\dots +t^{2(\operatorname{rk}\calN-1)})\\
\label{E:ExtraTerm00} &- \sum_{0\ne \hat \beta \in \calB_{\hat X}}t^{2d(\hat X,\hat \beta)}P^{N\cap \operatorname{Stab}_G\hat \beta }_t(Z_{\hat \beta}^{ss}\cap \hat \pi^{-1}Z^{ss}_R).
\end{align}

We now come to the more subtle point of relating the final term~\eqref{E:ExtraTerm00} to the representation on the normal slice to the orbit.
Let
\begin{equation}\label{E:xinZssR}
x\in Z^{ss}_R
\end{equation}
be a generic point, and let $\calN_x$ be the normal space to $G\cdot Z^{ss}_R$ in $X^{ss}$ at $x\in G\cdot Z^{ss}_R$; i.e., the fiber of $\calN$ at $x$.  Then we obtain a representation
\begin{equation}\label{E:rhoDef}
\rho:R\to \GL(\calN_x).
\end{equation}
We take $\TT_R$ to be the restriction of the maximal torus $\TT$ of $G$ under the inclusion  $R\subseteq G$, with maximal compact tori $T_R$ and $T$, respectively.  This gives an inclusion of (real) Lie algebras $\mathfrak t_R\subseteq \mathfrak t$, and we use the metric on $\mathfrak t_R$ \emph{induced from that of $\mathfrak t$}. (Recall that originally in the setup we were allowed to take any $\operatorname{Ad}$-invariant metric on $\mathfrak t$, but here we \emph{must} take the induced metric on $\mathfrak t_R$.)
Let
\begin{equation}\label{E:BrhoDef}
\mathcal B(\rho)
\end{equation}
be the indexing set for the induced stratification of $\mathbb P\mathcal N_x$: that is, $\mathcal B(\rho)$ is the set of all $\beta$ in a fixed positive Weyl chamber in  $\mathfrak t_R$ such that $\beta$ is the closest point to $0$ of the convex hull of a nonempty set of weights of the representation $\rho$.

Let $S_{\beta'}(\rho)$ for
\begin{equation}\label{E:BBrhoDef}
\beta'\in \calB(\rho)
\end{equation}
be the associated $R$-stratification of $\PP\calN_x$ (as defined in \S~\ref{S:E-P-S-def} and~\eqref{E:SbetaDef}).
Kirwan shows
that for $\hat \beta \in \calB_{\hat X}$, naturally in $\mathfrak t$, we may actually take $\hat \beta \in \mathfrak t_R\subseteq \mathfrak t$~\cite[Proof of Lem.~7.9, p.70]{kirwanblowup}, and  that there is a surjective
\footnote{
For surjectivity of the map \eqref{E:Brho-Bhat}, we are assuming that $\hat X$ is the full Kirwan blow-up.  Otherwise,  \eqref{E:Brho-Bhat} will not be surjective, and later, in \eqref{E:finExtraT1}, we would find that there was a further sum of the form in \eqref{E:ExtraTerm00}, corresponding to those $\hat \beta\in \mathcal B_{\hat X}$ that are not in the image of \eqref{E:Brho-Bhat}.  This is further clarified in \eqref{eq:ARcontribution} and \eqref{E:finExtraT}.
}
 map~\cite[p.73, Lem.~7.6, Lem.~7.9]{kirwanblowup}
\begin{equation}\label{E:Brho-Bhat}
\calB(\rho)\to \calB_{\hat X}
\end{equation}
taking $\beta'\in \calB(\rho)$ to the unique element $\hat \beta\in \calB_{\hat X}$ with $\beta'$ in its Weyl group orbit $W(G)\cdot \hat \beta$; here we are identifying  $W(G)=W(K)$, where $K$ is the maximal compact subgroup.   Note that we prefer to work with Weyl group orbits in the Lie algebra $\mathfrak t$ of the (real) maximal torus $T$, whereas Kirwan prefers to work with the equivalent notion of adjoint orbits in the Lie algebra $\mathfrak k$ of  $K$.
Given $\hat \beta\in \calB_{\hat X}$, the Weyl group orbit $W(G)$ of $\hat \beta$ decomposes into a finite number of $W(R)$ orbits.  There is a unique $\beta'\in \calB(\rho)$ in each $W(R)$ orbit contained in $W(G)\cdot \hat \beta$~\cite[Proof of Lem.~7.9, p.71]{kirwanblowup}.
 We let
 \begin{equation}\label{E:wbRGDef}
 w( \beta',R,G)
\end{equation} be the number of  $\beta' \in \calB (\rho)$ that lie in $W(G)\cdot \hat \beta$, i.e., the number of elements in the fiber of~\eqref{E:Brho-Bhat} containing $\beta'$, which is also  equal to  the number of $W(R)$ orbits  contained in the $W(G)$ orbit of~$\hat \beta$~\cite[p.68]{kirwanblowup}\footnote{There is a typo in~\cite[\S 3, p.49]{kirwanhyp}; the definition there is meant to read: $w(\beta,R,G)$ is the number of $R$-adjoint orbits contained in the $G$-adjoint orbit of $\beta$.}.

The fiber  $\hat \pi^{-1}(x)$ of $\hat \pi: \hat X \to  X^{ss}$  can be naturally identified with the projective space
$\PP\calN_x$. For $\hat \beta\in \calB_{\hat X}$, we have~\cite[Lem.~7.9]{kirwanblowup}
$$
S_{\hat X,\hat \beta}\cap \PP\calN_x=\bigcup_{\beta '\in \calB(\rho)\cap  \operatorname{Ad}(G)\hat \beta} S_{\beta'}(\rho);
$$
i.e., the union is over the $\beta'$ in the fiber of~\eqref{E:Brho-Bhat} over $\hat \beta$.
Kirwan proves in~\cite[Lem.~7.11]{kirwanblowup} that the codimensions $d(\hat X,\hat \beta)$ and $d(\PP\calN_x,\beta')$ of the associated strata $S_{\hat X,\hat \beta}$ and $S_{\beta'}(\rho)$ are equal.

Now given $\beta'\in \calB(\rho)$, let
$\hat \beta$ be the unique element in  $\calB_{\hat X}$ with $\beta'$ in its  $W(G)$ orbit;  i.e., the image of $\beta '$ under~\eqref{E:Brho-Bhat}.  Let
\begin{equation}\label{E:StabGb'}
\operatorname{Stab}_G\beta':=G_{\beta'}
\end{equation}
be the stabilizer of $\beta'$ under the \emph{$G$-adjoint action}.
In general $N\cap \operatorname{Stab}_G\hat \beta \ne N\cap \operatorname{Stab}_G\beta '$; they may differ by conjugation by an element of the Weyl group $W(G)$.  However, as both groups are conjugate subgroups of $N=N(R)$, they have well-defined actions on $Z^{ss}_R$ (since $N(R)$ preserves $Z^{ss}_R$).  Moreover, replacing $Z^{ss}_{\hat \beta}$ ($=:Z^{ss}_{\hat X,\hat \beta}$) with the isomorphic locus
\begin{equation}\label{E:ZsshatXbb'}
Z^{ss}_{\hat X,\beta'}
\end{equation}
defined by the element $\beta'\in W(G)\cdot\hat \beta$ (i.e., via~\eqref{E:HesseZss}), we obtain a well-defined action of $N\cap \operatorname{Stab}_G\beta'$ on $Z^{ss}_{\hat X,\beta'}\cap \hat \pi^{-1}Z^{ss}_R$, such that
$$
P^{N\cap \operatorname{Stab}\hat \beta }_t(Z_{\hat \beta}^{ss}\cap \hat \pi^{-1}Z^{ss}_R)=P^{N\cap \operatorname{Stab}_G \beta' }_t(Z_{\hat X, \beta'}^{ss}\cap \hat \pi^{-1}Z^{ss}_R).
$$

In summary, we have~\cite[(7.15)]{kirwanblowup},~\cite[(3.2), (3.4)]{kirwanhyp}:

\begin{flalign}
P_t^G(\hat X^{ss})
\label{E:finHss}&=P^G_t(X^{ss})& \!\!\!\!\!\!\!\!\!\!\!\!\!\!\!\text{``semi-stable locus''}\\
\label{E:finMainT1} &+P_t^N(Z_R^{ss})(t^2+\dots +t^{2(\operatorname{rk}\calN-1)}) & \text{``main term"}\\
\label{E:finExtraT1} &-\!\!\!\!\sum_{0\ne \beta'\in \calB(\rho)}\!\!\frac{1}{w(\beta',R,G)}t^{2d(\PP\calN_x,\beta')}P^{N\cap \operatorname{Stab}_G \beta' }_t(Z_{\hat X , \beta'}^{ss}\cap \hat \pi^{-1}Z^{ss}_R)\!\!\!\!\!\!
&\text{``extra term"}
\end{flalign}
For the formula above, recall that $Z^{ss}_R$ is defined in~\eqref{E:ZRss}, and the rest of the terms are defined above in this subsection: $\hat X^{ss}$~\eqref{E:PiHatDef-4.1},  $N$
\eqref{E:NDef}, $\calN$~\eqref{E:NNbundDef}, $\rho$~\eqref{E:rhoDef}, $\beta '$~\eqref{E:BBrhoDef}, $\calB(\rho)$~\eqref{E:BrhoDef}, $w(\beta ',R,G)$~\eqref{E:wbRGDef}, $x\in Z^{ss}_R$~\eqref{E:xinZssR}, $\operatorname{Stab}_G{\beta'}$~\eqref{E:StabGb'}, and $Z^{ss}_{\hat X,\beta'}$~\eqref{E:ZsshatXbb'}.

The goal in deriving  the formulas~\eqref{E:finHss},~\eqref{E:finMainT1},~\eqref{E:finExtraT1}  above is to try to reduce the computation of the equivariant cohomology of the blowup $\hat X^{ss}$ to certain computations on $X$.  The remark below can be quite helpful in this regard.

\begin{rem} \label{R:PiFibration}
The restriction $\hat \pi : (Z_{\hat X, \beta'}^{ss}\cap \hat \pi^{-1}Z^{ss}_R) \to Z^{ss}_R$   is an $(N\cap \operatorname{Stab}_G \beta')$-equivariant
fibration with fibers isomorphic to  $Z^{ss}_{\beta'}(\rho)$ (as defined in~\eqref{E:Zbeta},~\eqref{E:HesseZss} for the representation $\rho$)~\cite[8.11]{kirwan84}~\cite[p.50]{kirwanhyp}.
Moreover, if for instance  $N\cap \operatorname{Stab}_G  \beta'$ acts transitively on $Z^{ss}_R$,
then letting  $(N\cap \operatorname{Stab}_G \beta')_x$ be the stabilizer of the general point $x\in Z^{ss}_R$~\eqref{E:xinZssR}
in $N\cap \operatorname{Stab}_G \beta'$, we have   (e.g.,~\eqref{E:PGfib})
$$
P_t^{N\cap \operatorname{Stab}_G \beta'}(Z^{ss}_{\hat X, \beta'}\cap \hat \pi^{-1}Z^{ss}_R)=P^{(N\cap \operatorname{Stab}_G \beta')_x}(\hat \pi^{-1}(x))=P^{(N\cap \operatorname{Stab}_G \beta')_x}(Z^{ss}_{\beta'}(\rho)).
$$
Note that with the transitive group action on $Z^{ss}_R$, we may take any point  $x\in Z^{ss}_R$ to make our computation, since this will only change the computations up to conjugate groups, which will not affect the final outcome.
\end{rem}

\subsection{The correction terms in general} Having reviewed the case of a single blowup, we now   give the formulas  for the cohomology of the full inductive blowup, $\widetilde X^{ss}$.
We use the notation from \S~\ref{SSS:kirBlUp} and especially Remark~\ref{R:pi-r-Def}.
The relevant formula for computing the cohomology of $\widetilde X^{ss}$, generalizing~\eqref{E:finHss} to the full blowup,  is now~\cite[Eq.~3.2]{kirwanhyp}
\begin{equation}\label{eq:ARcontribution}
  P_t^G(\widetilde X^{ss})=P_t^G(X^{ss})+\sum_{R\in \calR} A_R(t).
\end{equation}
For any $R\in\calR$ the term $A_R(t)$ in~\eqref{eq:ARcontribution} records the change of the Betti numbers under the blowup $\pi_{\dim R}:X_{\dim R}\to X_{\dim R+1}$, as defined in Remark~\ref{R:pi-r-Def}; more precisely, one decomposes $\pi_{\dim R}$ into individual blowups of the loci $G\cdot Z^{ss}_{R,\dim R+1}$, where $Z^{ss}_{R,\dim R+1}$ is the strict transform of $Z^{ss}_R$ in $X^{ss}_{\dim R+1}$, and these are the correction terms for that blowup.

The explicit formula for these terms, generalizing~\eqref{E:finMainT1} and~\eqref{E:finExtraT1},  is given by~\cite[Eq.~3.4]{kirwanhyp}
\begin{align}
&A_R(t):=\\
\label{E:finMainT} &P_t^{N(R)}(Z_{R,\dim R+1}^{ss})(t^2+\dots +t^{2(\operatorname{rk}\calN_R-1)}) & \text{``main term"}\\
\label{E:finExtraT} &- \sum_{0\ne \beta'\in \calB_R(\rho)}\frac{1}{w(\beta',R,G)}t^{2d(\PP\calN_x,\beta')}P^{N(R)\cap \operatorname{Stab}_G \beta' }_t(Z^{ss}_{\beta ',R})& \text{``extra term"}
\end{align}
where $Z^{ss}_{\beta ',R}:=Z_{X_{\dim R}, \beta'}^{ss}\cap  \pi_{\dim R}^{-1}Z^{ss}_{R,\dim R+1}$.
These terms, as well as all of the other  terms  above, are described  in \S~\ref{SSS:CorTermsI} (see especially the references after~\eqref{E:finExtraT1}, and also Remark~\ref{R:PiFibration}, and  Remark~\ref{R:pi-r-Def}).

\section[The main correction terms for cubic threefolds]{The \emph{main correction} terms for cubic threefolds}
We now compute the main terms~\eqref{E:finMainT} for the case of cubic threefolds.  We have taken  $\calR=\{R_{2A_5}\cong \CC^*, R_{3D_4}\cong (\CC^*)^2, R_c\cong \PGL(2,\CC)\}$.
Since we have also already worked out the loci $Z^{ss}_R$ in Proposition~\ref{P:ZRss}, the main point is to understand the normalizers $N(R)$, and their action on the $Z^{ss}_R$.
We will work in the order of descending dimension of $R$, following the Kirwan blowup process.

\subsection{The main correction term for $R_c\cong \PGL(2,\CC)$, the chordal cubic case}
As we have seen, the first step in the Kirwan blowup process is to blow up the locus corresponding to chordal cubics.  We start by describing the main term~\eqref{E:finMainT}  for this blowup.

\begin{pro}[Main term for the chordal cubic]\label{P:MT-ChC}  For the group $R_c\cong \PGL(2,\CC)$, the main term~\eqref{E:finMainT} is given by
\begin{align*}
P^{N(R_c)}_t(Z^{ss}_{R_c})(t^2+\dots +t^{2(\operatorname{rk}\calN_{R_c}-1)}) &=(1-t^4)^{-1}(t^2+\dots+t^{24})\\
&= t^2+t^4+2t^6+2t^8+3t^{10}\mod t^{11}.
\end{align*}
\end{pro}
This will follow directly from the following lemma:
\begin{lem}[Proposition~\ref{P:App-R=SL2}]\label{L:RcNorm}
For $R_c$, the group $\PGL(2,\CC)$ embedded in $\SL(5,\CC)$ via its $\Sym^4\mathbb C^2$ ($\cong \mathbb C^5$) representation,
the normalizer $N(R_c)$ of $R_c$ in $\SL(5,\CC)$  is a split  central extension
\begin{equation}\label{E:NRcCent}
1\to\mu_5\to N(R_c)\to \PGL(2,\CC)\to 1,
\end{equation}
where $\mu_5$ is the group of $5$-th roots of unity. \qed
\end{lem}
The proof of the lemma is elementary, with all necessary computations given in Proposition~\ref{P:App-R=SL2}.

\begin{proof}[Proof of Proposition~\ref{P:MT-ChC}]
We saw in Proposition~\ref{P:ZRss} that $Z^{ss}_{R_c}$ consists of a single point, $V(F_{-1,2})$; i.e., the chordal cubic.
The stabilizer of $V(F_{-1,2})$  has connected component equal to $R_c$, so that the dimension of $G\cdot Z^{ss}_{R_c}=\dim G-3=21$.
Thus the rank of the normal bundle to the orbit  $G\cdot Z^{ss}_{R_c}$ is  $\operatorname{rk}\calN_{R_c}=34-21=13$.

Next we compute $P^{N(R_c)}_t(Z^{ss}_{R_c})$.
Since $Z^{ss}_{R_c}$ is a point, we have
\begin{align*}
H^\bullet_{N(R_c)}(Z^{ss}_{R_c})&=H^\bullet(B(N(R_c)))\\
&=H^\bullet (B\mu_5)\otimes H^\bullet (B\PGL(2,\CC))& (\eqref{E:NRcCent}, \eqref{E:K-AS-EC-cent})\\
&= H^\bullet (B\PGL(2,\CC))\\
&=\QQ[c] & (\text{Example~\ref{E:BSUn}, \ref{E:H-PGL}})
\end{align*}
where $\deg c=4$.   In other words $P^{N(R_c)}_t(Z^{ss}_{R_c})=(1-t^4)^{-1}=1+t^4+t^8+\dots$.
\end{proof}

\subsection{The main correction term for $R_{3D_4}\cong (\CC^*)^2$, the $3D_4$ case}
\begin{pro}[Main term for the $3D_4$ cubic]\label{P:MT-3D4}
For the group $R_{3D_4}\cong (\CC^*)^2$,  the main term~\eqref{E:finMainT} is given by
\begin{align*}
P^{N(R_{3D_4})}_t(Z^{ss}_{R_{3D_4,3}})(t^2+\dots +t^{2(\operatorname{rk}\calN_{R_{3D_4}}-1)})&= (1-t^4)^{-1}(1-t^6)^{-1}(t^2+\dots +t^{22})\\
&=t^2+t^4+2t^6+3t^8+4t^{10} \mod t^{11}.
\end{align*}
\end{pro}
Similarly to the chordal cubic case, this will follow directly from the elementary, but laborious, computations leading to the descriptions of the geometry involved. We will record the results here, while the full proofs are given in Proposition~\ref{P:App-R=C*2} of the Appendix. We first recall the notation $\SSS_n$ for the ``generalized permutation matrices of size $n$'', which explicitly are the matrices one obtains in $\GL(n,\CC)$ by permuting the columns of some diagonal matrix. Moreover, we adopt the convention that when we write an explicit form of a collection of matrices, and then write that it lies in a certain group, that this may impose an extra condition (eg.,~for a $5\times 5$ matrix, we may write $\in\SL$ to impose that it has determinant one, if it is not automatic from the form of the matrix). We finally record that $R_{3D_4}$ is isomorphic to $(\CC^*)^2$, and given in coordinates by
\begin{equation}\label{E:R3D4main}
R_{3D_4}=\operatorname{diag}(s,t,s^{-1}t^{-1},1,1)\cong (\CC^*)^2\,.
\end{equation}
\begin{lem}[{Proposition~\ref{P:App-R=C*2}}]\label{L:R3D4Norm1}
In the notation above:
\begin{enumerate}
\item The normalizer  $N(R_{3D_4})$ of $R_{3D_4}$ in $\SL(5,\CC)$ is
$$
N(R_{3D_4})=
\left\{\left(
\begin{array}{c|c}
\SSS_3&\\ \hline
&\GL_2\\
\end{array}
\right)\in \SL(5,\CC)\right\}\,.
$$
\item The fixed  locus $Z_{R_{3D_4}}^{ss}$ is unchanged under the first blowup:
$Z_{R_{3D_4,3}}^{ss}=Z_{R_{3D_4}}^{ss}$.

\item The normalizer acts on $Z_{R_{3D_4}}^{ss}$ transitively: $Z^{ss}_{R_{3D_4}}=N(R_{3D_4}) \cdot \{V(F_{3D_4})\}$. \qed
\end{enumerate}
\end{lem}
We will moreover need to know various other stabilizer groups. We denote by $\operatorname{Stab}(V(F_{3D4}))\subset\SL(5,\CC)$
the stabilizer of the cubic with equation $F_{3D4}$, denote $\operatorname{Aut}(V(F_{3D4}))$ its stabilizer in $\PGL(5,\CC)$, and let $\GL_{V(F_{3D4})}$ be the stabilizer in $\GL(5,\CC)$. We furthermore denote $D:=\{\diag(\lambda _0,\lambda _1,\lambda_2,\lambda_3,\lambda_4): \lambda_0\lambda_1\lambda_2=\lambda_3^3=\lambda_4^3\}$ an auxiliary group for these computations, which can be explicitly written as the direct product
\begin{equation}\label{E:D=Txmu}
D=\TT^3\times \mu_3
\end{equation}
of the torus $\TT^3=\diag(\lambda_0,\lambda_1,\lambda_0^{-1}\lambda_1^{-1}\lambda_3^3,\lambda_3,\lambda_3)\cong (\CC^*)^3$ and the group $\mu_3=\diag(1,1,1,1,\zeta^i)\cong \ZZ/3\ZZ$ where $\zeta$ is a primitive $3$-rd root of unity.
\begin{lem}[{Proposition~\ref{P:App-R=C*2}}]\label{L:R3D4Norm2}
The groups defined above are as follows.
\begin{enumerate}
\item The group $\operatorname{Stab}(V(F_{3D4}))$ is equal to
\begin{equation}\label{E:App-GF3D4}
\operatorname{Stab}(V(F_{3D4}))=
\left\{\left(
\begin{array}{c|c}
\SSS_3&\\ \hline
&\SSS_2\\
\end{array}
\right)\in \SL(5,\CC): \lambda_0\lambda_1\lambda_2=\lambda_3^3=\lambda_4^3\right\},
\end{equation}
where the $\lambda_i$ is the unique non-zero element in the $i$-th row.

\item There are central extensions
\begin{equation}\label{E:St3D4Ct}
1\to \mu_5\to \operatorname{Stab}(V(F_{3D_4})) \to \operatorname{Aut}(V(F_{3D_4}))\to 1\,,
\end{equation}
\begin{equation}\label{E:GLSt3D4Ct}
1\to \CC^*\to \GL_{V(F_{3D_4})}\to \operatorname{Aut}(V(F_{3D_4}))\to 1\,.
\end{equation}
\item There is an isomorphism
\begin{equation}\label{E:GLStDs3s2}
\GL_{V(F_{3D_4})}\cong D\rtimes (S_3\times S_2)\,,
\end{equation}
where the action of $S_3\times S_2$ on $D$ is to permute the entries; $S_3$ permutes the first three entries $\lambda_0,\lambda_1,\lambda_2$, and $S_2$, the last two, $\lambda_3,\lambda_4$. \qed
\end{enumerate}
\end{lem}
The proofs of the two lemma above are by (long) direct computations, given in Proposition~\ref{P:App-R=C*2} in the Appendix.

\begin{rem}
The reason for introducing $\GL_{V(F_{3D_4})}$ is that while there is a short exact sequence $1\to D\to \operatorname{Stab}(V(F_{3D_4}))\to S_3\times S_2\to 1$, this sequence does not split.  For the purposes of computing equivariant cohomology, it is just as easy to work with central extensions, see~\eqref{E:K-AS-EC-cent}, and so we work with $\GL_{V(F_{3D_4})}$, where the surjection splits, giving an easy semi-direct product with which to work.
\end{rem}
\begin{proof}[Proof of Proposition~\ref{P:MT-3D4}]
For brevity, we write $R=R_{3D_4}$ and $N=N(R)$.  Since by Lemma~\ref{L:R3D4Norm1}(1) the group~$N$ acts transitively on $Z^{ss}_R$, we have $\dim G\cdot Z^{ss}_R=\dim G\cdot \{V(F_{3D_4})\}=24-2=22$.
Thus the rank of the normal bundle to the orbit  $G\cdot Z^{ss}_{R}$ is  $\operatorname{rk}\calN_{R}=34-22=12$.

Next we compute $P^N_t(Z^{ss}_{R})$.
 From Lemma~\ref{L:R3D4Norm1}(3), we have
\begin{align*}
H^\bullet_N(Z^{ss}_{R})&=H^\bullet (B\operatorname{Stab}(V(F_{3D_4})))\,.
\end{align*}
At the same time we have
\begin{align*}
H^\bullet(B\CC^*)\otimes H^\bullet (B&\operatorname{Stab}(V(F_{3D_4})))\\
&=H^\bullet(B\CC^*)\otimes  H^\bullet (B\operatorname{Aut}(V(F_{3D_4})))&(\eqref{E:St3D4Ct},~\eqref{E:K-AS-EC-cent})\\
&=H^\bullet (B \GL_{V(F_{3D_4})})&(\eqref{E:GLSt3D4Ct})\\
&=H^\bullet(BD)^{S_3\times S_2}& (\eqref{E:GLStDs3s2}, \ \eqref{E:AS-EC-2})\\
&=(H^\bullet(B(\TT^3\times \mu_3)))^{S_3\times S_2} & (\eqref{E:D=Txmu})\\
&=H^\bullet(B(\TT^3))^{S_3\times S_2} & (\eqref{Exa:SemDirExa})\\
&=\QQ[c_1^{(1)},c_1^{(2)},c_1^{(3)}]^{S_3\times S_2},&
\end{align*}
with degree $c_1^{(i)}=2$.
The action of $S_3\times S_2$ is given as follows. First, we observe that the action is obtained from the action of $S_3\times S_2$ on the torus $\TT^3$ (e.g., Example~\ref{Exa:SemDirExa}), via the identifications~\eqref{E:D=Txmu} and~\eqref{E:GLStDs3s2}.
Concretely, $\TT^3=\diag(\lambda_0,\lambda_1,\lambda_0^{-1}\lambda_1^{-1}\lambda_3^3,\lambda_3,\lambda_3)\cong (\CC^*)^3=\diag(\lambda_0,\lambda_1,\lambda_3)$.
The action of $S_3\times S_2$ on $\mathbb T^3$ is to permute the entries; $S_3$ permutes the first three entries, and $S_2$, the last two.
Consequently, the $S_2$ factor acts trivially on $\mathbb T^3$.
To describe the action of the $S_3$ factor on $\mathbb T^3\cong (\CC^*)^3=\diag(\lambda_0,\lambda_1,\lambda_3)$, let us denote $S_3=\langle \delta,\gamma : \delta^2=\sigma^3=1,\ \delta \sigma =\sigma^2\delta\rangle$ the standard presentation of $S_3$.
Then
  $\delta(\lambda_0,\lambda_1,\lambda_3)=(\lambda_1,\lambda_0,\lambda_3)$, and $\sigma(\lambda_0,\lambda_1,\lambda_3)=(\lambda_0^{-1}\lambda_1^{-1}\lambda_3^3,\lambda_0,\lambda_3)$.
The action of $S_3$ on the symmetric algebra $\QQ[c_1^{(1)},c_1^{(2)},c_1^{(3)}]$ is induced by the action of $S_3$ on the vector space $\QQ\langle c_1^{(1)},c_1^{(2)},c_1^{(3)}\rangle$, and so we see that $\delta$ and $\sigma$ act by
$$
\delta =
\left(
\begin{array}{ccc}
0&1&0\\
1&0&0\\
0&0&1
\end{array}
\right)\ \ \hbox{and}
\ \ \sigma =
\left(
\begin{array}{ccc}
-1&1&0\\
-1&0&0\\
3&0&1
\end{array}
\right).
$$
At this point, one may use Molien's formula, or simply observe via the characters that the representation of $S_3$ given by the matrices above is isomorphic to the standard representation, which decomposes as the direct sum of the trivial representation and the representation of $S_3$ as the dihedral group acting on the plane.
In any case,
we obtain the generating function for $\QQ[c_1^{(1)},c_1^{(2)},c_1^{(3)}]^{S_3\times S_2}$ to be
$$
(1-t^2)^{-1}(1-t^4)^{-1}(1-t^6)^{-1}
$$
Thus putting everything together, we have
\begin{align*}
P^N_t(Z^{ss}_{R})&=(1-t^2)\cdot (1-t^2)^{-1}(1-t^4)^{-1}(1-t^6)^{-1} =(1-t^4)^{-1}(1-t^6)^{-1}\\
&\equiv 1+ t^4+t^6+t^8+t^{10} \mod t^{11}.
\end{align*}
\end{proof}

\subsection{The main correction term for $R_{2A_5}\cong \CC^*$, the $2A_5$ case}
\begin{pro}[Main term for $2A_5$ cubics]\label{P:MT-2A5}  For the group $R_{2A_5}\cong \CC^*$, the main term~\eqref{E:finMainT} is given by
\begin{align*}
P^{N(R_{2A_5})}_t(Z^{ss}_{R_{2A_5,2}})(t^2+\dots +t^{2(\operatorname{rk}\calN_{R_{2A_5}}-1)})&=(1-t^4)^{-1}(1+t^2)(t^2+\dots+t^{18})\\
&\equiv t^2+2t^4+3t^6+4t^8+5t^{10} \mod t^{11}.
\end{align*}
\end{pro}
This will follow directly from the following lemma, which will be proven by direct elementary computations, given in Propositions~\ref{P:App-R=C*p1} and~\ref{P:App-R=C*p2} in the Appendix. We recall  from Proposition~\ref{P:ZRss} that $Z^{ss}_{R_{2A_5}}$ is the set of semi-stable cubics defined by equations of the form $a_0x_2^3 + a_1x_0x_3^2 + a_2x_1^2x_4 +a_3 x_0x_2x_4 + a_4x_1x_2x_3=0$.
\begin{lem}[{Propositions~\ref{P:App-R=C*p1} and~\ref{P:App-R=C*p2}}]\label{L:R2A5Norm} For $R_{2A_5}=\operatorname{diag}(\lambda^{2},\lambda ,1,\lambda ^{-1},\lambda ^{-2})\cong \CC^*$:
\begin{enumerate}
\item  The normalizer $N(R_{2A_5})$ of $R_{2A_5}$ in $\SL(5,\CC)$ is equal to
the subgroup  of $\SL(5,\CC)$ that is the semi-direct product
\begin{equation}\label{E:LR2A5Norm}
N(R_{2A_5})\cong \TT^4\rtimes \ZZ/2\ZZ
\end{equation}
of the maximal torus $\TT^4$, and the involution $\tau:x_i\mapsto x_{4-i}$, with the semi-direct product given by the homomorphism
$$ \tau \mapsto \left(\diag(\lambda_0,\lambda_1,\lambda_2,\lambda_3,\lambda_4)\mapsto \diag(\lambda_4,\lambda_3,\lambda_2,\lambda_1,\lambda_0)\right)\,\,.$$

\item The orbit of the chordal cubic $G\cdot Z^{ss}_{R_c}$ meets $Z^{ss}_{R_{2A_5}}$
precisely in the divisor defined by the equation $4a_0a_1a_2-a_3a_4^2=0$.
Thus the strict transform $Z_{R_{2A_5},2}^{ss}$ is isomorphic to $Z^{ss}_{R_{2A_5}}$.

\item The quotient $Z_{R_{2A_5}}^{ss}/\TT^4$ is isomorphic to $\PP^1$. \qed
\end{enumerate}
\end{lem}
For the last item we note that $\TT^4/\CC^*$ acts on $Z_{R_{2A_5}}^{ss}$ with finite stabilizers, and so the quotient $Z_{R_{2A_5}}^{ss}/\TT^4=Z_{R_{2A_5}}^{ss}/(\TT^4/\CC^*)$ is a well-defined variety.

\begin{proof}[Proof of Proposition~\ref{P:MT-3D4}]
For brevity, write $R=R_{2A_5}$ and $N=N(R)$.  In Lemma~\ref{L:R2A5Norm}(2), we saw that $Z^{ss}_{R_{},2}=Z^{ss}_{R_{}}$, and that $Z^{ss}_{R_{}}$ has dimension $4$.   Now consider the subgroup $G'\subseteq \SL(5,\CC)$ consisting of those $g$ such that $g\cdot Z^{ss}_R\subseteq Z^{ss}_R$.  This group $G'$ has dimension $4$.  Indeed, we have $N\subseteq G'$, so that $\dim G'\ge 4$.  On the other hand, the stabilizer of a general point of $Z^{ss}_R$ is $1$-dimensional (it has connected component equal to $R$), so that
if $\dim G'\ge 5$, then the dimension of the orbit of a general point would be $\ge 5-1=\dim Z^{ss}_R$.  But then there would be a Zariski dense subset of $Z^{ss}_R$ corresponding to projectively equivalent cubics.
It follows that   $\dim G\cdot Z^{ss}_R=\dim Z^{ss}_R+\dim G -\dim G' =4+24-4=24$.
Thus the rank of the normal bundle to the orbit  $G\cdot Z^{ss}_{R}$ is  $\operatorname{rk}\calN_{R}=34-24=10$.

Next we compute $P^N_t(Z^{ss}_{R})$.
 From Lemma~\ref{L:R2A5Norm}, we have
\begin{align*}
H^\bullet_N(Z^{ss}_{R})&=\left(H^\bullet _{\TT^4}(Z^{ss}_R)\right)^{\ZZ/2\ZZ}& (\eqref{E:LR2A5Norm}, \ \eqref{E:AS-EC-2})\\
&=(H^\bullet(BR)\otimes H^\bullet_{\TT^4/R}(Z^{ss}_R))^{\ZZ/2\ZZ} & (\eqref{E:K-AS-EC-cent}) \\
&=(H^\bullet(BR)\otimes H^\bullet(\PP^1))^{\ZZ/2\ZZ} & (\text{Lemma~\ref{L:R2A5Norm}}(3),\ \eqref{E:AS-EC-FinG})\\
&=(\QQ[c_1] \otimes \QQ[h]/h^2)^{\ZZ/2\ZZ} &
\end{align*}
where $\deg c_1=\deg h=2$.
Now one must trace through the constructions to find the action of $\ZZ/2\ZZ=\langle \tau\rangle$ on the polynomial ring.  The action on the cohomology of $BR$ is induced by the action on $R$, which can easily be seen to be given by $\lambda\mapsto \lambda^{-1}$.  Thus the action of $\tau$ on $c_1$ is given by $\tau c_1=-c_1$.  The action of $\tau$ on the cohomology of $\PP^1$  is induced by the action on $\PP^1$.  The action of $\tau$ on $Z^{ss}_R$ is given by $\tau (a_0:\dots :a_4)=(a_0:a_2:a_1:a_3:a_4)$.  Using the locus $\{V(F_{A,B}):(A,B)\ne (0,0)\}\subseteq Z^{ss}_R$ (i.e., $a_1=a_2=-a_3=1$), one sees that the action on the quotient $\PP^1=Z^{ss}_R/\TT^4$ is trivial.  Thus the action of $\tau$ on $h$ is trivial.  Thus we have $ (\QQ[c_1] \otimes \QQ[h]/h^2)^{\ZZ/2\ZZ} =\QQ[c_1^2]\otimes \QQ[h]/(h^2)$.  Thus $P_t^N(Z^{ss}_R)=(1-t^4)^{-1}(1+t^2)$.
\end{proof}

\section[The extra correction terms for cubic threefolds]{The \emph{extra correction terms} for cubic threefolds}
Having computed the {\em main terms} of the contributions $A_R(t)$ given by~\eqref{E:finMainT}, to finish the computation of $H^\bullet (\MK)$ following Kirwan's method it thus remains to compute the \emph{extra terms} given by~\eqref{E:finExtraT}.  A key point is to describe for each $R$ the representation $\rho:R\to \operatorname{Aut}(\calN_x)$ on the normal slice to the orbit $G\cdot Z^{ss}_R$ at a generic point $x\in Z^{ss}_R$.  We start in \S~\ref{S:TanOrb} by reviewing a general approach to computing the tangent space to an orbit for the case of hypersurfaces.  We then utilize this in the case of cubic threefolds, and consequently obtain the extra terms.

\subsection{Tangent spaces to orbits for hypersurfaces}\label{S:TanOrb}
Let $F\in H^0(\PP^n,\calO_{\PP^n}(d))$ the form defining a hypersurface $V(F)\subseteq \PP^n$.  We wish to describe the tangent space to the orbit $\GL(n+1,\CC)\cdot F$.

\begin{rem} We will ultimately be interested in the normal space to the orbit $\SL(n+1,\CC)\cdot \{V(F)\}$ in $\PP H^0(\PP^n,\calO_{\PP^n}(d))$.  However, since the normal space of any submanifold  $Y$ in projective space $\PP(V)$ can, via the
Euler sequence,  be identified with the  normal space to its cone $C(Y)$ in $V$, we may instead consider the $\GL(n+1,\CC)$ orbit of $F$ in  $H^0(\PP^n,\calO_{\PP^n}(d))$, rather than the $\SL(n+1,\CC)$ orbit of $V(F)$ in $\PP H^0(\PP^n,\calO_{\PP^n}(d))$.
\end{rem}

To compute the tangent space to the $\GL(n+1,\CC)$ orbit of $F$ we  work with the Lie algebra $\gl(n+1,\CC)$ and use the exponential map $\exp: \gl(n+1,\CC) \to \GL(n+1,\CC)$.
If $e \in \gl(n+1,\CC)$ then taking the derivative $\frac{d}{dt}\left.(\exp(te)F)\right|_{t=0}$ gives a tangent vector $t_e$  in the tangent space to the orbit $\GL(n+1,\CC)\cdot F$. Taking a basis of $\gl(n+1,\CC)$ we then obtain generators for the tangent space
to the orbit $\GL(n+1,\CC)\cdot F$.
Concretely, with respect to the coordinates $(x_0:\dots:x_n)$, numbering the rows and columns of matrices from $0$ to $n$, we then take as generators for $\gl(n+1,\CC)$ the elementary matrices $e_{ij}$ for all $0\le i, j\le n$, where $e_{ij}$ is the matrix with all zero entries except for the $ij$-the entry which is one. Given a form $F$, we then denote
$$
(DF)_{ij}:=\frac{d}{dt}\left.\left(\exp(te_{ij})F\right)\right|_{t=0}.
$$
We denote by $DF$ the associated matrix with entries $(DF)_{ij}$.  Finally, we conclude that the tangent space to the orbit $\GL(n+1,\CC)\cdot F$ is given by the span of the entries of the matrix $DF$.

We implement this now in the case of polystable cubic threefolds:
\begin{exa}[Tangent spaces to the orbits of strictly polystable cubic threefolds]\label{Exa:D}
For a strictly polystable cubic threefold defined by a cubic form $F$, the tangent space to the orbit $\GL(5,\CC)\cdot F$ is given by the span of the entries of the matrix $DF$.  In particular we have:

\begin{enumerate}
\item For $F=F_{A,B}$, $(A,B)\ne (0,0)$, the matrix $DF_{A,B}$ is given by
$$
 DF_{A,B}=
 $$
$$
\scriptstyle
\hskip-1cm \left(\begin{smallmatrix}
 x_0x_3^2-x_0x_2x_4&x_1x_3^2-x_1x_2x_4&x_2x_3^2-x_2^2x_4&x_3^3-x_2x_3x_4&x_3^2x_4-x_2x_4^2\\
 2x_0x_1x_4+Bx_0x_2x_3&2x_1^2x_4+Bx_1x_2x_3&2x_1x_2x_4+Bx_2^2x_3&2x_1x_3x_4+Bx_2x_3^2&2x_1x_4^2+Bx_2x_3x_4\\
 3Ax_0x_2^2-x_0^2x_4+Bx_0x_1x_3&3Ax_1x_2^2-x_0x_1x_4+Bx_1^2x_3&3Ax_2^3-x_0x_2x_4+Bx_1x_2x_3&3Ax_2^2x_3-x_0x_3x_4+Bx_1x_3^2&3Ax_2^2x_4-x_0x_4^2+Bx_1x_3x_4\\
 2x_0^2x_3+Bx_0x_1x_2&2x_0x_1x_3+Bx_1^2x_2&2x_0x_2x_3+Bx_1x_2^2&2x_0x_3^2+Bx_1x_2x_3&2x_0x_3x_4+Bx_1x_2x_4\\
 x_0x_1^2-x_0^2x_2&x_1^3-x_0x_1x_2&x_1^2x_2-x_0x_2^2&x_1^2x_3-x_0x_2x_3&x_1^2x_4-x_0x_2x_4
 \end{smallmatrix}\right)
 $$
To quickly determine all linear equations satisfied by the entries of the matrix $DF$, we note that since $F_{A,B}$ is preserved by the action of $\CC^*=R_{2A_5}$, any relation decomposes under the action, given by~\eqref{E:R2A5}, into linear equations among monomials of the same weight. By inspection the weights in the above matrix under the above action range from $+4$ in the bottom left to corner to $-4$ in the top right corner, with entries along each diagonal going down-and-right having the same weight. Then within each diagonal to determine possible linear relations among the entries one first checks if any monomial is repeated. One easily sees then that for $4A/B^2\ne 1$ and $(A,B)\ne (0,0)$ the only possible relation could be in weight $0$, i.e.~among the five entries on the main diagonal. By looking at which monomials repeat in which entries, one finally sees that for $4A/B^2\ne 1$ and  $(A,B)\ne (0,0)$ the only linear relation satisfied by the entries of $DF_{A,B}$ is 
\begin{equation}\label{eq:rel*}
  2(DF_{A,B})_{00}+(DF_{A,B})_{11}-(DF_{A,B})_{33}-2(DF_{A,B})_{44}=0\,.
\end{equation}

\item For $F_{1,-2}$, i.e., $A=1$ and $B=-2$, in addition to the linear relation~\eqref{eq:rel*}, one sees that the entries of $DF_{1,-2}$ satisfies two additional linear relations in weights $1$ and $-1$:
\begin{equation}\label{eq:relSL}\begin{aligned}&(DF_{1,-2})_{10}+2(DF_{1,-2})_{21}+3(DF_{1,-2})_{32}+4(DF_{1,-2})_{43}=0; \\ &(DF_{1,-2})_{34}+2(DF_{1,-2})_{23}+3(DF_{1,-2})_{12}+4(DF_{1,-2})_{01}=0\,,
\end{aligned}
\end{equation}
and no further relations.

\item For $F=F_{3D_4}$, the matrix $DF_{3D_4}$ is given by
$$
 DF_{3D_4}=\begin{pmatrix}
 x_0x_1x_2&x_1^2x_2&x_1x_2^2&x_1x_2x_3&x_1x_2x_4\\
 x_0^2x_2&x_0x_1x_2&x_0x_2^2&x_0x_2x_3&x_0x_2x_4\\
 x_0^2x_1&x_0x_1^2&x_0x_1x_2&x_0x_1x_3&x_0x_1x_4\\
 3x_0x_3^2&3x_1x_3^2&3x_2x_3^2&3x_3^3&3x_3^2x_4\\
 3x_0x_4^2&3x_1x_4^2&3x_2x_4^2&3x_3x_4^2&3x_4^3
 \end{pmatrix}\,.
$$
Since all entries of this matrix are monomial, the only possible linear relations are pairwise equalities, up to a constant factor. One sees that the only monomial that repeats more than once is $x_0x_1x_2$, and thus the set of linear relations satisfied by the entries of $DF_{3D_4}$ is
$$
  (DF_{3D_4})_{00}=(DF_{3D_4})_{11}=(DF_{3D_4})_{22}.
$$
\end{enumerate}
\end{exa}

\subsection{The extra correction term for $R_c\cong \PGL(2,\CC)$, the chordal cubic case}

\begin{pro}[Extra term for the chordal cubic] \label{P:ET-ChC}

 For the group $R_c\cong \PGL(2,\CC)$, the extra term~\eqref{E:finExtraT} is given by
\begin{align}\label{E:PET-ChC}
\sum_{0\ne \beta'\in \calB(\rho)}\frac{1}{w(\beta',R_c,G)}t^{2d(\PP\calN_x,\beta')}P^{N(R_c)\cap \operatorname{Stab} \beta' }_t(Z_{\beta ',R_c}^{ss})&=0 \mod t^{11}.
\end{align}
\end{pro}

Recall that in the formula above $x$ is a general point of $Z^{ss}_{R_c}$ (in this case, since $Z^{ss}_{R_c}$ consists of a single point, $x$ corresponds to the chordal cubic), $\calN_x$ is the fiber of the normal bundle to the orbit $G\cdot Z^{ss}_{R_c}$ at $x$, and $\rho:R_c\to \operatorname{Aut}(\calN_x)$ is the induced representation.   The proof of the proposition will consist of showing that the codimension $d(\PP\calN_x,\beta')$ of any stratum $S_{\beta'}(\rho)$ for $0\ne \beta'\in \calB(\rho)$ is at least $6$.  This will follow from the following lemma, describing the representation $\rho$.

\begin{lem}\label{L:Rc-Nx-Rep}
For $R_c=\PGL(2,\CC)$, $\dim\calN_x=13$, and the representation $\rho$ of $R_c$ on $\calN_x$ is the one induced by the $\SL(2,\CC)$-representation $\Sym^{12}\CC^2$, where $\CC^2$ is the standard two-dimensional representation.  Consequently the weights of the action of the maximal torus $T\cong \CC^*$ in $R_c$ are
$$
-12,-10,-8,-6,-4,-2,0,2,4,6,8,10,12.
$$
\end{lem}

\begin{proof}
It suffices to determine the restriction of $\rho$  to the  maximal torus $\TT$ in $\SL(2,\CC)$ (induced by the homomorphism $\SL(2,\CC)\to \PGL(2,\CC)$).
Recall that $Z^{ss}_{R_c}=\{V(F_{1,-2})\}$, and so to describe $\calN_x$, we must simply describe the normal space to the orbit $G\cdot V(F_{1,-2})$ at $V(F_{1,-2})$.

The maximal torus  $\TT=\diag(t,t^{-1})$ in $\SL(2,\CC)$ acts on coordinates $(x_0:\dots:x_4)$ diagonally by $(t^4:t^2:1:t^{-2}:t^{-4})$. Thus it multiplies each cubic monomial by some power of $t$, so that each monomial is thus an eigenspace for the action of $\TT$.
Thus $T_x\CC^{35}=\CC^{35}$
decomposes as a sum of one-dimensional representations of $\TT$ with the following multiplicities of weights
$$
(\pm 12) \times 1, (\pm 10) \times 1, (\pm 8) \times 2, (\pm 6) \times 3, (\pm4) \times 4, (\pm 2) \times 4, (0) \times 5.
$$
The tangent space to the orbit $G\cdot V(F_{1,-2}) $ is generated by the entries of the matrix in Example~\ref{Exa:D}(1).
Each binomial spans an eigenspace for the action of $\TT$, and  weights of the action of $\TT$ on these generators can be computed  directly to be equal to
\begin{equation*}
(\pm 8) \times 1, (\pm 6) \times 2, (\pm4) \times 3, (\pm 2) \times 4, ( 0) \times 5.
\end{equation*}
Now the relation $(\ref{eq:rel*})$ is among the weight $0$ generators, and thus we may drop one of them in forming a basis of the tangent space.  The two relations~\eqref{eq:relSL} are among generators of weights $2$ and $-2$, respectively, so we can also drop one generator of weight $2$ and $-2$.
In summary, the weights for $\TT$ on the tangent space to the orbit are given by
\begin{equation}\label{equ:w2}
(\pm 8) \times 1, (\pm 6) \times 2, (\pm4) \times 3, (\pm 2) \times 3, ( 0) \times 4.
\end{equation}
Taking the complement of the set of weights of the representation on the tangent space to the orbit in the set of weights of the representation on $\CC^{35}$ gives the weights of the action on the normal space, proving the lemma.
\end{proof}

\begin{rem}
We note that in fact this result already follows from the geometry as described in~\cite{act}, where it was shown that the exceptional divisor in $\MK$ corresponding to the chordal cubic is in fact the locus of Jacobians of hyperelliptic curves of genus five, which is thus the moduli space of twelve points on $\PP^1$, which is exactly the GIT quotient for the $12$-th symmetric power of the standard representation of $\SL(2,\CC)$ on $\CC^2$.
\end{rem}

\begin{proof}[Proof of Proposition~\ref{P:ET-ChC}]
From the description of the weights of $\rho$ in Lemma~\ref{L:Rc-Nx-Rep}, we see that we can take $\calB(\rho)=\{0,2,4,6,8,10,12\}$.
We can estimate the codimension $d(\beta')$ for $\beta'\in \calB(\rho)$ using~\eqref{E:Sbcodim}; i.e., $d(\beta')=n(\beta')-\dim ( R_{c}/P_{\beta'})$, where $n(\beta')$ is the number of weights less than $\beta'$, namely $6+\beta'/2$, and $P_{\beta'}$ is the associated parabolic subgroup.   One can check that $P_{\beta'}$ is equal to the $2$-dimensional Borel subgroup consisting of upper triangular matrices; however, it suffices for our purposes to observe that $P_{\beta'}$ contains the Borel.   Thus $d(\beta')\ge (6+\beta'/2)-3+2\ge 6$.
Thus the terms in~\eqref{E:PET-ChC} begin in degree $\ge 2d(\beta')=12$, and are zero modulo $t^{11}$.
\end{proof}

\subsection{The extra correction term for $R_{3D_4}\cong (\CC^*)^2$, the $3D_4$ case}

\begin{pro}[Extra term for the $3D_4$ cubic] \label{P:ET-3D4}

 For the group $R_{3D_4}\cong (\CC^*)^2$, the extra term~\eqref{E:finExtraT} is given by
\begin{align*}
\scriptstyle
-\sum_{0\ne \beta'\in \calB(\rho)}\frac{1}{w(\beta',R_{3D_4},G)}t^{2d(\PP\calN_x,\beta')}P^{N(R_{3D_4})\cap \operatorname{Stab} \beta' }_t(Z_{\beta', R_{3D_4}}^{ss})&=-t^8-2t^{10} \mod t^{11}.
\end{align*}
\end{pro}

Recall that in the formula above, $x$ is a general point of $Z^{ss}_{R_{3D_4}}$ (in this case, since $Z^{ss}_{R_{3D_4}}$ consists of a single $N(R_{3D_4})$ orbit, so we may take $x=V(F_{3D_4})$), $\calN_x$ is the fiber of the normal bundle to the orbit $G\cdot Z^{ss}_{R_c}$ at $x$, and $\rho:R_c\to \operatorname{Aut}(\calN_x)$ is the induced representation.  We start with the following lemma, describing the representation $\rho$.

\begin{lem}\label{L:R3D4-Nx-Rep}
For $R_{3D_4}=\diag(\lambda_0,\lambda_1,\lambda_2,1,1)\cap \SL(5,\CC)\cong (\CC^*)^2$, and $x=V(F_{3D_4})$, we have $\dim\calN_x=12$,  with an explicit basis given by
\begin{equation}\label{E:3D4-Nx-Bas}
x_0^3,\  x_1^3,\  x_2^3,\
x_0^2x_3,\  x_1^2x_3, \ x_2^2x_3,\
x_0^2x_4, \ x_1^2x_4, \ x_2^2x_4,\
x_0x_3x_4, \ x_1x_3x_4, \ x_2x_3x_4.
\end{equation}
Each element of this basis is an eigenvector for the action of $R$, so that under the inclusion $\mathfrak t_R=\{(\alpha_0,\alpha_1,\alpha_2,0,0): \sum \alpha_i=0\}\subseteq \mathfrak t=\{(\alpha_0,\dots ,\alpha_4): \sum \alpha_i=0\}\subseteq  \RR^5$, and the identification $\mathfrak t_R=\mathfrak t_R^\vee$  induced by the metric from $\mathfrak t$, the weights of the representation $$\rho:R_{3D_4}\to \operatorname{Aut}(\calN_x)$$ in the order of the basis elements above, are equal to (see also Figure~\ref{F:curve-weights}):
\begin{equation}
\begin{array}{ccc}
(2,-1,-1,0,0),&  (-1,2,-1,0,0),  & (-1,-1,2,0,0),\\
(4/3,-2/3,-2/3,0,0),& (-2/3,4/3,-2/3,0,0),& (-2/3,-2/3,4/3,0,0) \\
 (4/3,-2/3,-2/3,0,0),& (-2/3,4/3,-2/3,0,0),&
  (-2/3,-2/3,4/3,0,0), \\
   (2/3,-1/3,-1/3,0,0), & (-1/3,2/3,-1/3,0,0),& (-1/3,-1/3,2/3,0,0).
\end{array}
\end{equation}
\end{lem}

\begin{proof}
The basis for $\calN_x$ comes directly from Example~\ref{Exa:D}(3).   The identification of $\mathfrak t_R=\mathfrak t_R^\vee$ is given by the composition $\mathfrak t_R^\vee \hookrightarrow \mathfrak t^\vee \stackrel{\sim}{\to}\mathfrak t \twoheadrightarrow \mathfrak t_R$, where the last map is the orthogonal projection, and the rest is immediate.  This gives the weights above.  Indeed, for a basis monomial $x^I$, writing the group as $\diag(e^{i\alpha_0},e^{i\alpha_1},e^{i\alpha_2},1,1)$, the associated weight as a linear map  (viewed as either a linear map in $\mathfrak t_R^\vee$ or $\mathfrak t^\vee$) is given by $I.\alpha$.  The orthogonal projection is given by   $(\alpha_0,\dots,\alpha_4)\mapsto (\alpha_0-\frac{1}{3}\sum_{i=0}^2\alpha_i,\alpha_1-\frac{1}{3}\sum_{i=0}^2\alpha_i,\alpha_2-\frac{1}{3}\sum_{i=0}^2\alpha_i,0,0)$.  For instance, the monomial $x_0^2x_3$ has index $I=(2,0,0,1,0)$, and the orthogonal projection is then $(\frac{4}{3},-\frac{2}{3},-\frac{2}{3},0,0)$.
\end{proof}

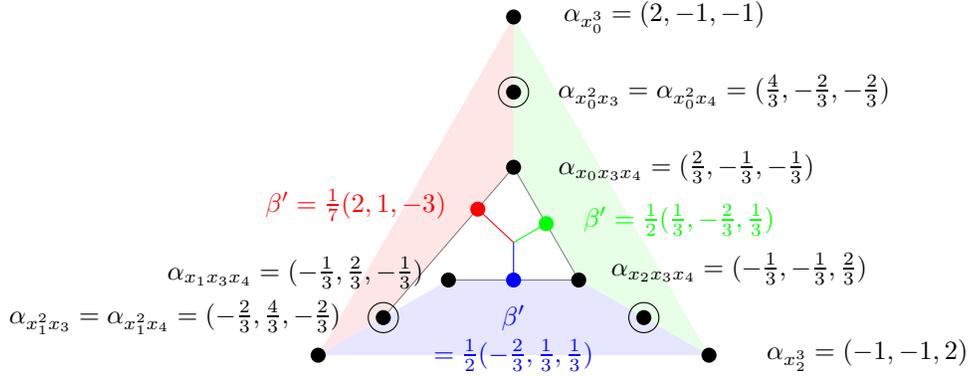
\begin{figure}[!htb]
  \begin{tikzpicture}
  \draw [black,opacity=.5] (90:1cm) -- (330:1cm);
  \draw [black,opacity=.5] (90:1cm) -- (210:2cm);
  \draw [black,opacity=.5] (330:1cm) -- (210:1cm);
   \draw [green,opacity=1] (0,0) -- (30: .5cm);
    \draw [blue,opacity=1] (0,0) -- (-90: .5cm);
   \draw [red,opacity=1] (0,0) -- (137: .65cm);
   \fill[color=red, opacity=0.1] (90:1cm)--(90:3cm)--(210:3cm)--(210:2cm);
   \fill[green, opacity=0.1] (90:3cm)--(330:3cm)--(330:1cm)--(90:1cm);
   \fill[blue, opacity=0.1]  (330:3cm)--(330:1cm)--(210:1cm)--(210:3cm);

 \foreach\ang in {90,210,330}{
    \fill [black,opacity=1] (\ang:1cm) circle (.1 cm); }
 \foreach\ang in {90,210,330}{
    \fill [black,opacity=1] (\ang:2cm) circle (.1 cm);
    \draw  (\ang:2cm) circle (.2 cm);
    }
    \foreach\ang in {90,210,330}{
    \fill [black,opacity=1] (\ang:3cm) circle (.1 cm); }

    \fill [green,opacity=1] (30: .5cm) circle (.1 cm);
      \fill [red,opacity=1] (137: .65cm) circle (.1 cm);
       \fill [blue,opacity=1] (-90: .5cm) circle (.1 cm);

    \node at (2.2cm, .3cm)  (dd)    {\textcolor{green}{$\beta'=\frac{1}{2}(\frac{1}{3},-\frac{2}{3},\frac{1}{3})$}};
    \node at (0cm, -1cm)  (dd)    {\textcolor{blue}{$\beta'$}};
      \node at (0cm, -1.5cm)  (dd)    {\textcolor{blue}{$=\frac{1}{2}(-\frac{2}{3},\frac{1}{3},\frac{1}{3})$}};
    \node at (-2.1cm, .5cm)  (dd)    {\textcolor{red}{$\beta'=\frac{1}{7}(2,1,-3)$}};
     \node at (2cm, 3cm)  (dd)    { $\alpha_{x_0^3}=(2,-1,-1)$};
      \node at (2.8cm,2cm)  (dd)    { $\alpha_{x_0^2x_3}=\alpha_{x_0^2x_4}=(\frac{4}{3},-\frac{2}{3},-\frac{2}{3})$};
          \node at (2.3cm, 1cm)  (dd)    { $\alpha_{x_0x_3x_4}=(\frac{2}{3},-\frac{1}{3},-\frac{1}{3})$};
        \node at (3cm, -.35cm)  (dd)    { $\alpha_{x_2x_3x_4}=(-\frac{1}{3},-\frac{1}{3},\frac{2}{3})$};
      \node at (4.7cm, -1.5cm)  (dd)    { $\alpha_{x_2^3}=(-1,-1,2)$};
      \node at (-2.9cm,-.4cm)  (dd)    { $\alpha_{x_1x_3x_4}=(-\frac{1}{3},\frac{2}{3},-\frac{1}{3})$};
      \node at (-4.5cm,-1cm)  (dd)    { $\alpha_{x_1^2x_3}=\alpha_{x_1^2x_4}=(-\frac{2}{3},\frac{4}{3},-\frac{2}{3})$};

\end{tikzpicture}

\caption{
A sample codimension~$4$ element is given above in blue as $\textcolor{blue}{\beta'}= \frac{1}{2}(-\frac{2}{3},\frac{1}{3},\frac{1}{3})\in \calB$.
Another sample codimension~$4$ element is given in green as $\textcolor{green}{\beta'}= \frac{1}{2}(\frac{1}{3},-\frac{2}{3},\frac{1}{3})\in \calB$.   A sample codimension~$5$ element is given in red as $\textcolor{red}{\beta'}= \frac{1}{7}(2,1,-3)\in \calB$.  We have dropped the last two $0$-coordinates for brevity.
 }
 \label{F:curve-weights}
\end{figure}
We now move to describe the indexing set $\calB(\rho)$ associated to the representation $\rho$, as well as various groups and loci associated to the elements $\beta'\in \calB(\rho)$. First we note that since $R$ is the two-dimensional torus, the Weyl chamber is all of $\RR^2$. By construction, the indexing set $\calB(\rho)$ associated to the representation $\rho:R_{3D_4}\to \operatorname{Aut}(\calN_x)$ is then the set of points $\beta'\in\mathfrak t_R=\{(\alpha_0,\alpha_1,\alpha_2,0,0): \sum \alpha_i=0\}$
that can be described as the closest point to the origin (with respect to the standard metric in $\RR^5$) in a convex hull of the weights. The codimension of the associated stratum $S_{\beta'}$ is then  equal to the number of weights lying on the same side as the origin from the orthogonal complement to $\beta'$. The situation is described by the following lemma, and the corresponding sets and weights are depicted in Figure~\ref{F:curve-weights}.
\begin{lem}\label{L:R(rho)3D4p1}
For the group $R=R_{3D_4}\cong (\CC^*)^2$, all codimension $4$ and $5$ strata are as follows.
\begin{enumerate}
\item[(a)]  There are $3$ codimension~$4$ elements $\beta'\in \calB$: $\frac{1}{2}(-\frac{2}{3},\frac{1}{3},\frac{1}{3})$ (in \color{blue}{blue }\color{black}  in Figure~\ref{F:curve-weights});
$\frac{1}{2}(\frac{1}{3},-\frac{2}{3},\frac{1}{3})$ (in \color{green}green \color{black} in the figure); and
$\frac{1}{2}(\frac{1}{3}, \frac{1}{3},-\frac{2}{3})$ (not shown in the figure). For each of these, $w(\beta ',R,G)=3$;

\item[(b)] There are $6$ codimension~$5$ elements $\beta'\in \calB'$:  $\frac{1}{7}(1,2,-3)$;  $\frac{1}{7}(2,-3,1)$;  $ \frac{1}{7}(-3,1,2)$; $\frac{1}{7}(-3,2,1)$; $ \frac{1}{7}(1,-3,2)$;
 $\frac{1}{7}(2,1,-3)$ (the last of them shown in \color{red}{red }\color{black}  in Figure~\ref{F:curve-weights}). For each of these, $w(\beta',R,G)=6$.
\end{enumerate}
Moreover, in each of these two cases, all the elements $\beta'$ are in the same orbit of the Weyl group of $G=\SL(5,\CC)$.
\end{lem}
\begin{proof}
To find $\calB(\rho)$ one observes that since $R$ is a torus, the Weyl chamber is all of $\RR^2$.   It is easy to check from~\eqref{E:HesseZss} that since $R$ is a torus, the strata $S_{\beta'}$ are non-empty for the weights as given and shown in the picture. The fact that all 3 elements in each case lie in the same orbit of the Weyl group is also immediate, since $W(G)=S_5$ acts by permuting the entries, which preserves $\mathfrak t_R$ only for the subgroup $S_3$ permuting the first three entries.  The weights also then easily follow.
\end{proof}
We now further describe the relevant fixed point sets and the action of the stabilizers. Since all $\beta'$ in the case (a) or in case (b) lie in the same orbit of the Weyl group, it is enough to work with one representative for each case.
\begin{lem}[{Lemmas~\ref{L:App-R(rho)3D4a} and \ref{L:App-R(rho)3D4b}}]\label{L:R(rho)3D4p2} In the notation above:
\begin{enumerate}
\item  For $\beta'=\frac{1}{2}(-\frac{2}{3},\frac{1}{3},\frac{1}{3})$ (case (a)), we have
$$
Z^{ss}_{\beta'}=\{[a:b]\in \PP\CC\langle  x_1x_3x_4,x_2x_3x_4\rangle: a\ne 0, b\ne 0 \}\cong \CC^*.
$$

\item For $\beta'= \frac{1}{7}(2,1,-3)$ (case (b)), we have
$$
Z^{ss}_{\beta'}=\{[a:b:c]\in \PP\CC\langle  x_0x_3x_4,x_1^2x_3, x_1^2x_4\rangle: a\ne 0,\ \text{and}\  (b,c)\ne (0,0) \}\cong \AAA^2-\{0\}.
$$

\item For either $\beta'$, the group $N\cap \operatorname{Stab}_G\beta'$ acts transitively on $Z^{ss}_R$.

\item The action of $(N\cap \operatorname{Stab}_G\beta')_x$ on $Z^{ss}_{\beta'}$ is induced by change of coordinates, via the inclusion $(N\cap \operatorname{Stab}_G\beta')_x\subseteq \SL(5,\CC)$ and the description of the loci above in terms of cubic forms.
 For $\beta '=\frac{1}{2}(-\frac{2}{3},\frac{1}{3},\frac{1}{3})$ (case (a)), the group $(N\cap \operatorname{Stab}_G\beta')_x$ acts transitively on $Z^{ss}_{\beta '}$, and the stabilizer of the point $(1:1)\in Z^{ss}_{\beta'}$ is given explicitly
 as $(\CC^*\times \mu_{15})\times(S_2\times S_2)$. \qed
\end{enumerate}
\end{lem}
We note that for case (b), more details on the action of $(N\cap \operatorname{Stab}_G\beta')_x$ on $Z^{ss}_{\beta'}$ will not be needed.

The proof of the lemma is elementary, with all necessary computations given in Lemmas~\ref{L:App-R(rho)3D4a} and \ref{L:App-R(rho)3D4b} in the Appendix, for the cases (a) and (b), respectively.

\begin{proof}[Proof of Proposition~\ref{P:ET-3D4}]
Recall that we are trying to compute the extra terms~\eqref{E:finExtraT} contributed by the $3D_4$ locus:
\begin{align*}
 &- \sum_{0\ne \beta'\in \calB(\rho)}\frac{1}{w(\beta',R,G)}t^{2d(\PP\calN_x,\beta')}P^{N\cap \operatorname{Stab}_G \beta' }_t(Z^{ss}_{\beta',R}).
\end{align*}
As we saw in Lemma~\ref{L:R(rho)3D4p1}, there are $3$ elements $\beta'\in \calB(\rho)$ with $d(\PP\calN_x,\beta')=4$, and $6$ with $d(\PP\calN_x,\beta')=5$.  The other $\beta'\ne 0$ have $d(\PP\calN_x,\beta')\ge 6$, and will not contribute modulo $t^{11}$, and so we ignore them.    We also saw that the $3$ codimension~$4$ (resp.~$6$ codimension~$5$) elements $\beta'$ were all in the same Weyl group of $G=\SL(5,\CC)$ orbit, and so we are free to work with one representative from each orbit.    Finally, we showed  for the $\beta'$ of codimension~$4$ and $5$ that $N\cap \operatorname{Stab}_G\beta'$ acts transitively on $Z^{ss}_R$, so that by Remark~\ref{R:PiFibration},
$$
P^{N\cap \operatorname{Stab}_B\beta'}_r(Z^{ss}_{\beta',R})=P_t^{(N\cap \operatorname{Stab}_B\beta')_x}(Z^{ss}_{\beta'}).
$$
Let us consider first $P_t^{(N\cap \operatorname{Stab}_B\beta')_x}(Z^{ss}_{\beta'})$ for the codimension~$5$ loci.  Since we have from Lemma~\ref{L:R(rho)3D4p2}(2) that $Z^{ss}_{\beta'}$ is connected, it follows that  $P_t^{(N\cap \operatorname{Stab}_B\beta')_x}(Z^{ss}_{\beta'})=1+O(t)$.
Thus the codimension $5$ extra term is $-\sum_{i=1}^6 \frac{1}{6}t^{10} (1+\dots)$.
 Now let us consider  $P_t^{(N\cap \operatorname{Stab}_B\beta')_x}(Z^{ss}_{\beta'})$ for the codimension~$4$ loci.  Take the representative $\beta'=\frac{1}{2}(-\frac{2}{3},\frac{1}{3},\frac{1}{3})$.  By Lemma~\ref{L:R(rho)3D4p2}(4), the action of $(N\cap \operatorname{Stab}_B\beta')_x$ on $Z^{ss}_{\beta'}$ is transitive, with the stabilizer equal to $(\CC^*\times \mu_{15})\times(S_2\times S_2)$.  Thus $P_t^{(N\cap \operatorname{Stab}_B\beta')_x}(Z^{ss}_{\beta'})=P_t(B\CC^*)=(1-t^2)^{-1}$.
In summary, we have
\begin{align*}
(\text{``extra term'' codim 4})\ \ \ \ \ &-\sum_{i=1}^3 \frac{1}{3}t^{8} (1+t^2+\dots) \\
(\text{``extra term'' codim 5})\ \ \ \ \ &-\sum_{i=1}^6 \frac{1}{6}t^{10} (1+\dots)
\end{align*}
This completes the proof of the proposition.
\end{proof}

\subsection{The extra correction term for $R_{2A_5}\cong \CC^*$, the $2A_5$ case}

\begin{pro}[Extra term for the $2A_5$ cubics] \label{P:ET-2A5}
For the group $R_{2A_5}\cong \CC^*$, the extra term~\eqref{E:finExtraT} is given by
\begin{align*}
-\sum_{0\ne \beta'\in \calB(\rho)}\frac{1}{w(\beta',R_{2A_5},G)}t^{2d(\PP\calN_x,\beta')}P^{N(R_{2A_5})\cap \operatorname{Stab} \beta' }_t(Z_{\beta',R_{2A_5}}^{ss})&\equiv -t^{10} \mod t^{11}.
\end{align*}
\end{pro}

This will follow from the next lemma.

\begin{lem}\label{L:R2A5-Nx-Rep}
For $R_{2A_5}=\diag(\lambda^2,\lambda,1,\lambda^{-1},\lambda^{-2})\cong \CC^*$, $\dim\calN_x=10$, and the weights of the representation $\rho$ of $R_{2A_5}$ on $\calN_x$ are
$$
-6,-5,-4,-3,-2,2,3,4,5,6.
$$
\end{lem}
\begin{proof}
The proof is essentially the same as that of Lemma~\ref{L:Rc-Nx-Rep}.
The vector space $T_x\CC^{35}=\CC^{35}$
decomposes as a sum of one-dimensional representations of $R=R_{2A_5}$ with the following multiplicities of weights
$$
(\pm 6) \times 1, (\pm 5) \times 1, (\pm 4) \times 2, (\pm 3) \times 3, (\pm2) \times 4, (\pm 1) \times 4, (0) \times 5.
$$
The tangent space to the orbit  $\GL(5,\CC)\cdot F_{A,B} $ of a general $F_{A,B}$ is generated by the entries of the matrix in Example~\ref{Exa:D}(1).
Each binomial spans an eigenspace for the action of $R$, and  weights of the action of $R$ on these generators can be computed  directly to be equal to
\begin{equation*}
(\pm 4) \times 1, (\pm 3) \times 2, (\pm 2) \times 3, (\pm 1) \times 4, ( 0) \times 5.
\end{equation*}
Now the relation $(\ref{eq:rel*})$ is among the weight $0$ generators, and thus we may drop one of them in forming a basis of the tangent space.
In summary, the weights for $R$ on the tangent space to the orbit  $\GL(5,\CC) \cdot F_{A,B}$ are given by
$$
(\pm 4) \times 1, (\pm 3) \times 2, (\pm 2) \times 3, (\pm 1) \times 4, ( 0) \times 4.
$$

 Recall, however, that the relevant normal space $\calN_x$ is the normal space in $X$ to the orbit $G\cdot Z_R^{ss}$ of the
fixed set $Z_R^{ss}$.   We know
that $Z_R^{ss}/G$ is one-dimensional (corresponding to the curve $\calT$ of cubics). Thus the tangent space $T_x(G\cdot Z_R^{ss})$, when lifted to $\CC^{35}$,  is the sum of $T_{F_{A,B}}(\GL(5,\CC)\cdot F_{A,B})$ together with a tangent vector representing the direction along $Z_R^{ss}/G$; in other words a tangent vector which comes from varying $4A/B^2$. As such a tangent vector we can take the deformation which simply deforms the cubic $F_{A,B}$ by changing the coefficient $A$. Clearly the derivative in this direction is equal to $\frac{d}{dA}F_{A,B}=x_2^3$.
This is weight $0$ (and, as expected,  does not lie in the span of the weight~$0$ space of the orbit). Thus the lift to $\CC^{35}$ of the tangent space to the orbit  $G\cdot Z^{ss}_R$ is given by a space with weights
$$
(\pm 4) \times 1, (\pm 3) \times 2, (\pm 2) \times 3, (\pm 1) \times 4, ( 0) \times 5.
$$
Taking the complement of the set of weights of the representation on the tangent space in the set of weights of the representation on $\CC^{35}$ gives the weights of the representation on the normal space.
\end{proof}

\begin{proof}[Proof of Proposition~\ref{P:ET-2A5}]
From the description of the weights of $\rho$ in Lemma~\ref{L:R2A5-Nx-Rep}, we see that we can take $\calB(\rho)= \{\pm 2,\pm 3,\pm 4,\pm 5,\pm 6\}$.
We can compute  $d(\beta'):=\operatorname{codim}_{\mathbb C}\dim S_{\beta'}$ for $\beta'\in \calB(\rho)$ using the definition~\eqref{E:SbetaDef};
i.e.,  $d(\beta')=n(|\beta'|)$, where $n(|\beta'|)$ is the number of weights less than $|\beta'|$, namely $5+|\beta'|-2\ge 5$.
Thus the terms in the formula in Proposition~\eqref{P:ET-2A5} begin in degree $\ge 2d(\beta')=10$, and  we must only compute for $\beta'=\pm 2$.

 One immediately obtains $w(\beta',R_{2A_5},G)=2$.
Finally, it is easy to see that $Z^{ss}_{\beta',R}$ is connected, since $Z^{ss}_R$ is connected, and one can check that $Z^{ss}_{\beta'}\cong \CC$ is connected.  Thus we have $P^{N(R_{2A_5})\cap \operatorname{Stab}_G\beta'}_t(Z^{ss}_{\beta',R})=1+\dots$, completing the proof.
\end{proof}

\section{Putting the terms together to compute the cohomology of $\calM^K$}\label{S:MK-Coh-sum}

We now put together the results in the previous sections to complete  the proof of Theorem~\ref{teo:betti} for $\MK$.  Recall that $\MK$  has only finite quotient singularities, so the cohomology satisfies Poincar\'e duality.  Consequently, as $\dim \MK=10$, it suffices to compute $P_t(\MK)$ modulo $t^{11}$.  We have:

\begin{align*}
&P_t(\calM^K)=P_t^G(\widetilde X^{ss}) \equiv\\
& 1+t^2+2t^4+3t^6+5t^8+6t^{10}&\text{(Semi-stable locus, Prop.~\ref{P:PtGXss})}\\
&\ +t^2+t^4+2t^6+2t^8+3t^{10}&\text{(Main term, chordal
cubic, Prop.~\ref{P:MT-ChC})}\\
&\ +t^2+t^4+2t^6+3t^8+4t^{10} &\text{(Main term, $3D_4$ cubic, Prop.~\ref{P:MT-3D4})}\\
&\ +t^2+2t^4+3t^6+4t^8+5t^{10}&\text{(Main term, $2A_5$ cubics, Prop.~\ref{P:MT-2A5})}\\
&\ -0&\text{(Extra term, chordal cubic, Prop.~\ref{P:ET-ChC})}\\
&\  -t^8-2t^{10}&\text{(Extra term, $3D_4$ cubic, Prop.~\ref{P:ET-3D4})}\\
&\ -t^{10}&\text{(Extra term, $2A_5$ cubics, Prop.~\ref{P:ET-2A5})}\\
\equiv& 1+4t^2+6t^4+10t^6+13t^8+15t^{10} \mod t^{11}.\hskip-1cm
\end{align*}

This completes the proof of Theorem~\ref{teo:betti} for $\MK$.

%% file: sec5_IHGIT.tex
\chapter{The intersection cohomology of the GIT moduli space $\GIT$}\label{sec:IHGIT}
Our next goal is to compute the intersection cohomology of the GIT moduli space $\GIT$ by comparing this with the (intersection) cohomology of the Kirwan blowup $\MK$. We recall that $\MK$ is smooth up to finite quotient singularities, which  implies that
cohomology and intersection cohomology coincide. The starting point lies in Kirwan's techniques~\cite{kirwanrational1,kirwanhyp}, which in turn use the decomposition theorem in a subtle way. To carry this out
requires a thorough understanding of the geometric situation, and it turns out that this analysis is rather involved.
However, these geometric details will come in handy also  in Chapter~\ref{sec:IHball} where we compute the
intersection cohomology of the Baily--Borel compactification $\BG$.

\section[The Kirwan blowup to the GIT quotient, in general]{Obtaining the intersection cohomology of the GIT quotient from the cohomology of the Kirwan blowup, in general}
\subsection{Intersection cohomology for a single blowup}\label{SSS:IC1}
As before, it is notationally easier to explain the formulas after a single blowup.  We start with this case, and then in the next subsection explain what the formulas are for a sequence of blowups.

We start again in the situation of \S~\ref{SSS:kirBlUp}, where we have fixed a maximal dimensional connected component $R\in \calR$ of the stabilizer of a strictly polystable point, taken the blowup
\begin{equation}\label{E:piHatDef}
\hat \pi:\hat X\to X^{ss}
\end{equation}
along the locus $G\cdot Z^{ss}_R$~\eqref{E:ZRss}, and chosen a linearization of the action on an ample line bundle $\hat L$ on $\hat X$,  as described in \S~\ref{SSS:kirBlUp}.
For simplicity, we further assume that $Z_R^{ss}$ is connected (which is the case for all $R$ for the moduli of cubic threefolds, as computed in the previous section); see~\cite[Rem.~1.19]{kirwanrational1} for the necessary modifications if $Z_R^{ss}$ is disconnected.
Considering GIT quotients, we have the following diagram~\cite[Diag.~1]{kirwanrational1} summarizing the situation:
$$
\xymatrix@C=1em@R=1em{
\hat X^{ss}\ar@{->>}[rrrrr] \ar@{->>}[ddddd]_{\hat \pi}&&&&&\hat X/\!\!/_{\hat L}G \ar@{->>}[ddddd]^{\hat \pi_G}\\
&E^{ss}\ar@{^(->}[lu] \ar@{->>}[rrr] \ar@{->>}[ddd]_{\hat \pi}&&&E/\!\!/_{\hat L}G \ar@{^(->}[ru] \ar@{->>}[ddd]^{\hat \pi_G}&\\
&&\hat {\calN}\ar@{->>}[lu] \ar@{->>}[r] \ar@{->>}[d]_{T(\hat \pi)}&\hat{\calN}/\!\!/_{\hat L}G \ar@{->>}[ru] \ar@{->>}[d]^{T(\hat \pi)_G} &&\\
&&\calN\ar@{->>}[ld] \ar@{->>}[r]&\calN/\!\!/_LG\ar@{->>}[rd]&&\\
&G\cdot Z^{ss}_R \ar@{^(->}[ld] \ar@{->>}[rrr]&&&Z_R/\!\!/_{L}N \ar@{^(->}[rd]&\\
X^{ss}\ar@{->>}[rrrrr]&&&&&X/\!\!/_LG\\
}
$$

Most of the notation in the diagram has been introduced before, but here we recall some of the definitions, and explain the remaining notation.  We have set $\calN$ to be the normal bundle to  $G\cdot Z^{ss}_R$ in $X^{ss}$, we defined $E$ to be the exceptional divisor of the blowup $\hat \pi:\hat X\to X^{ss}$, $E^{ss}$ to be the intersection of $E$ with $\hat X^{ss}$, and set   $\hat {\calN}$ to be the normal bundle to $E^{ss}$ in $\hat X^{ss}$.  The morphism $T(\hat \pi):\hat{\calN}\to \calN$ is induced by the differential of $\hat \pi$.  The $G$-actions extend naturally to all of the spaces in the diagram, and the linearizations are induced via pull-back along the respective morphisms.
The group $N$ is defined to be the normalizer of $R$, and we have identified $G\cdot Z^{ss}_R/\!\!/_LG=Z^{ss}_R/\!\!/_LN$ via the identification $G\cdot Z^{ss}_R=G\times_ N Z^{ss}_R$~\cite[p.~72]{kirwanblowup}.

The goal is to use the decomposition theorem  of Beilinson, Bernstein, Deligne, and Gabber~\cite{bbdg} to compare the intersection cohomology groups  $IH^\bullet(\hat X/\!\!/_{\hat L}G)$ and $IH^\bullet(X/\!\!/_LG)$.  The interaction between the general theory and  Kirwan's results is outlined in~\cite[Rem.~2.2, p.~484]{kirwanrational1}.  Here we will review the outline of Kirwan's argument, as this helps clarify the meaning of terms appearing in the formulas, particularly when we move to the specific case of cubic threefolds.

Immediately from the decomposition theorem, since $\hat \pi_G$ is a birational morphism, one has that $IH^\bullet(X/\!\!/_LG)$ is a direct summand of $IH^\bullet(\hat X/\!\!/_{\hat L} G)$, and thus the goal is to determine the extra summands precisely.  In other words, denoting $IP_t$ the intersection Poincar\'e polynomial, we can write
\begin{equation}\label{E:BR(t)-intro}
IP_t(X/\!\!/_LG)=IP_t(\hat X/\!\!/_{\hat L}G)-B_R(t),
\end{equation}
for some polynomial $B_R(t)$ with non-negative integral coefficients, and our aim is to compute this polynomial.

The first observation is the following.  Given a fibred product diagram
\begin{equation}\label{E:MV-setup}
\xymatrix{
E\ar@{^(->}[r] \ar@{->}[d]&\hat U \ar@{^(->}[r] \ar[d]&\hat V \ar[d]^f\\
C\ar@{^(->}[r]&U\ar@{^(->}[r]&V\\
}
\end{equation}
where  $f:\hat V\to V$ is a birational  morphism of projective varieties that is an isomorphism on the complement of a closed subvariety $C\subseteq V$, and  $U\subseteq V$ is an open neighborhood of $C$, then~\cite[Lem.~2.8]{kirwanrational1}:
\begin{equation}\label{E:IH-Lem1}
\dim IH^i(V)=\dim IH^i(\hat V)-\dim IH^i(\hat U)+\dim IH^i(U).
\end{equation}
We note that this is a slightly more general statement than \cite[Lem.~2.8]{kirwanrational1}, but that proof goes through unchanged to prove \eqref{E:IH-Lem1}.  

In our situation, we apply this for $U$ being an open neighborhood of $Z_R/\!\!/_LN$ in $X/\!\!/_LG$, so that $\hat U:=\hat \pi_G^{-1}(U)$ is its inverse image in $\hat X/\!\!/_{\hat L}G$, which is an open neighborhood of $E/\!\!/_{\hat L}G$
in $\hat X/\!\!/_{\hat L}G$, so that we obtain
\begin{equation}\label{E:KR1-L2.8}
\dim IH^i(X/\!\!/_LG)=\dim IH^i(\hat X/\!\!/_{\hat L}G)-\dim IH^i(\hat U)+\dim IH^i(U).
\end{equation}
Kirwan then shows that there is an open neighborhood $U$ as above that is homeomorphic to $\calN/\!\!/_LG$, and furthermore such that its preimage $\hat U$ is homeomorphic to $\hat{\calN}/\!\!/_{\hat L}G$~\cite[Lem.~2.9, and p.~487]{kirwanrational1}.  This establishes~\cite[Cor.~2.11]{kirwanrational1}:
\begin{equation}\label{E:KR1-C2.11}
\dim IH^i(X/\!\!/_LG)=\dim IH^i(\hat X/\!\!/_{\hat L}G)-\dim IH^i(\hat{\calN}/\!\!/_{\hat L}G)+\dim IH^i(\calN/\!\!/_LG).
\end{equation}

Now we use the fact that $\calN/\!\!/_LG\cong (\calN|_{Z^{ss}_R})/\!\!/_LN$ and $\hat {\calN}/\!\!/_{\hat L}G\cong (\hat{\calN}|_{\hat \pi^{-1}Z^{ss}_R})/\!\!/_{\hat L}N$
\cite[p.~493, and 1.7, p.~476]{kirwanrational1}, and the fact that the intersection cohomology of the quotient by a finite group is the subset of intersection cohomology that is invariant under the finite group~\cite[Lem.~2.12]{kirwanrational1}, to conclude that
\begin{align}
\nonumber  IH^\bullet(\calN/\!\!/_LG)&\cong [IH^\bullet((\calN|_{Z^{ss}_R})/\!\!/_LN_0)]^{\pi_0N}\\
\label{E:KR1-Ep493(2)}
IH^\bullet(\hat{\calN}/\!\!/_{\hat L}G)&\cong [IH^\bullet((\hat{\calN}|_{\hat \pi^{-1}Z^{ss}_R})/\!\!/_{\hat L}N_0)]^{\pi_0N}
\end{align}
where $N_0\subset N$ is the connected component of the identity, and $\pi_0N=N/N_0$ is the group of connected components of~$N$.
A Leray spectral sequence argument for the morphisms $(\calN|_{Z^{ss}_R})/\!\!/_LN_0\to Z_R/\!\!/_LN_0$ and $(\hat{\calN}|_{\hat \pi^{-1}Z^{ss}_R})/\!\!/_{\hat L}N_0\to Z_R/\!\!/_LN_0$  yields
\cite[p.~493, Lem.~2.15, Prop.~2.13]{kirwanrational1}:
\begin{align}
IH^\bullet((\calN|_{Z^{ss}_R})/\!\!/_LN_0)&\cong IH^\bullet(\calN_x/\!\!/R)\otimes H^\bullet(Z_R/\!\!/_LN_0)\\
\label{E:KR1-Lem2.15}
 IH^\bullet((\hat{\calN}|_{\hat \pi^{-1}Z^{ss}_R})/\!\!/_{\hat L}N_0)&\cong IH^\bullet( \hat {\calN_x} /\!\!/R)\otimes H^\bullet(Z_R/\!\!/_LN_0)
\end{align}
We emphasize here that we do not projectivize $\mathcal N_x$ or $\hat {\mathcal N}_x$.  
Here, as in~\eqref{E:xinZssR},  $x$ is a general point of $Z^{ss}_R$, and the fiber $\calN_x$ has an action $\rho:R\to \GL(\calN_x)$ as described in~\eqref{E:rhoDef}, which  is used as the linearization.  The quotient $\calN_x/\!\!/R$ is defined to be $\operatorname{Spec}(\operatorname{Sym}^\bullet \calN_x^\vee)^R$, the spectrum of the invariant ring (which is why we are not indicating a linearization in the notation), and similarly for $\hat {\calN_x}/\!\!/R$.
We are also using the fact established in the proof of~\cite[Prop.~2.13]{kirwanrational1} that $Z_R/\!\!/_LN_0$ has at worst finite quotient singularities, so that its cohomology and intersection cohomology are equal.

Combining this with~\eqref{E:KR1-Ep493(2)} yields~\cite[2.20, p.~494, Lem.~2.15]{kirwanrational1}:
\begin{align}
\label{E:KR1-E2.22(1)}
IH^\bullet(\calN/\!\!/_LG)&\cong [IH^\bullet(\calN_x/\!\!/R)\otimes H^\bullet(Z_R/\!\!/_LN_0)]^{\pi_0 N}\\
\label{E:KR1-E2.22(2)}
IH^\bullet(\hat {\calN}/\!\!/_{\hat L}G)&\cong [IH^\bullet( \hat {\calN_x} /\!\!/R)\otimes H^\bullet(Z_R/\!\!/_LN_0)]^{\pi_0 N}.
\end{align}

Finally one shows that $IH^\bullet(\hat{\calN}_x/\!\!/R)\cong IH^\bullet(\PP(\calN_x)/\!\!/R)$~\cite[Lem.~2.15]{kirwanrational1}, and that there is a natural surjection
$$
IH^i(\PP(\calN_x)/\!\!/R)\to IH^i(\calN_x/\!\!/R)
$$
whose kernel is isomorphic to $IH^{i-2}(\PP(\calN_x)/\!\!/R)$ if $i\le \dim \PP(\calN_x)/\!\!/R$ and to $IH^i(\PP(\calN_x)/\!\!/R)$ otherwise
\cite[Cor.~2.17]{kirwanrational1}.  Putting this all together we have~\cite[Prop.~2.1]{kirwanrational1}:
\begin{align}
\nonumber \dim IH^i(X/\!\!/_LG)&=\dim IH^i(\hat X/\!\!/_{\hat L}G)\\
\label{E:IH-CorTerm1}
&-\sum_{p+q=i}\dim\left[H^p(Z_R/\!\!/_LN_0)\otimes IH^{\hat q}(\PP(\calN_x)/\!\!/R)\right]^{\pi_0N}
\end{align}
where $\hat q=q-2$ for $q\le \dim \PP(\calN_x)/\!\!/R$ and $\hat q=q$ otherwise.
Reiterating from above,  $x$ is a general point of $Z^{ss}_R$, the fiber $\calN_x$ has an action $\rho:R\to \GL(\calN_x)$ as described in~\eqref{E:rhoDef}, which  is used as the linearization, $N_0\subset N$ is the connected component of the identity, and $\pi_0N:=N/N_0$.
The action of $\pi_0N$ on $H^\bullet(Z_{R}/\!\!/_LN_0)$ is induced from the given action of $N$ on $Z^{ss}_R$, and the action on the tensor product is induced via a Leray spectral sequence (see~\eqref{E:KR1-Ep493(2)} and~\eqref{E:KR1-E2.22(2)}).

\begin{rem}\label{R:pi0TrivB}
If the action of $\pi_0N$ on the tensor product $H^p(Z_R/\!\!/_LN_0)\otimes IH^{\hat q}(\PP(\calN_x)/\!\!/R)$ is trivial on the second factor, then we can conclude from~\eqref{E:IH-CorTerm1} that~\cite[Cor.~2.28]{kirwanrational1}:
\begin{align*}
\dim IH^i(X/\!\!/_LG)&=\dim IH^i(\hat X/\!\!/_{\hat L}G)\\
&-\sum_{p+q=i}\dim H^p(Z_R/\!\!/_LN)\cdot \dim IH^{\hat q}(\PP(\calN_x)/\!\!/R)
\end{align*}
so that
\begin{align}\label{E:IP-BR(t)}
IP_t(X/\!\!/_LG)&=IP_t(\hat X/\!\!/_{\hat L}G)-\underbrace{P_t(Z_R/\!\!/_L N)\,b_R(t)}_{\text{``}B_R(t)\text{''}}
\end{align}
where $B_R(t)$ is the product indicated above, and  the coefficients of the polynomial $b_R(t)$ are essentially the shifted intersection Betti numbers of the GIT quotient $\PP(\calN_x)/\!\!/R$ defined as follows.  Denoting $c=\dim \PP(\calN_x)/\!\!/R$, and denoting these intersection Betti numbers as
\begin{equation}\label{E:IHPNR}
IP_t(\PP(\calN_x)/\!\!/R)=1+b_1t+b_2t^2+b_3t^3+\dots+b_{c-1}t^{c-1}+b_{c}t^{c}+b_{c+1}t^{c+1}+b_{c+2}t^{c+2}+\dots+t^{2c},
\end{equation}
the polynomial $b_R(t)$ is defined as
\begin{equation}\label{E:bR(t)-def}
\qquad \quad \qquad \qquad b_R(t):= t^2+b_1t^3+\dots+b_{c-3}t^{c-1}+b_{c-2}t^{c}+b_{c+1}t^{c+1}+b_{c+2}t^{c+2}+\dots+t^{2c}\\
\end{equation}
We recall here that the intersection cohomology always satisfies Poincar\'e duality, and thus we could have written $b_j=b_{2c-j}$ instead.
\end{rem}

\begin{rem}\label{R:pi0TrivExpl}
We remark that the spectral sequence argument for the morphism $(\hat{\calN}|_{\pi^{-1}Z^{ss}_R})/\!\!/_{\hat L}N_0\to Z_R/\!\!/_LN_0$
that yielded~\eqref{E:KR1-Lem2.15} also shows that if $Z_R/\!\!/_LN_0$ is simply connected, then the action of $\pi_0N$ on the tensor product  splits, and we can conclude from~\eqref{E:IH-CorTerm1} that
\begin{align*}
\dim IH^i(X/\!\!/_LG)&=\dim IH^i(\hat X/\!\!/_{\hat L}G)\\
&-\sum_{p+q=i}\dim H^p(Z_R/\!\!/_LN)\cdot \dim [IH^{\hat q}(\PP(\calN_x)/\!\!/R)]^{\pi_0N}
\end{align*}
so that
\begin{align}\label{E:IP-BR(t)-1}
IP_t(X/\!\!/_LG)&=IP_t(\hat X/\!\!/_{\hat L}G)-\underbrace{P_t(Z_R/\!\!/_L N)\,b_R(t)}_{\text{``}B_R(t)\text{''}}
\end{align}
where $B_R(t)$ is the product indicated above, and  the coefficients $b_i$ of the polynomial $b_R(t)$ are the dimensions of the $\pi_0N$-invariant subspace of the intersection cohomology: $b_i:=\dim [IH^{\hat i}(\PP(\calN_x)/\!\!/R)]^{\pi_0N}$, shifted as before in~\eqref{E:bR(t)-def}.  In other words, we replace~\eqref{E:IHPNR} with $IP_t^{\pi_0N}(\PP(\calN_x)/\!\!/R)$, and then $b_R(t)$ is obtained from this as in~\eqref{E:bR(t)-def}.
\end{rem}

\subsection{The correction terms in general}
Having reviewed the case of a single blowup, we now give the formulas  for the cohomology  utilizing the full Kirwan blowup.
We use the notation from \S~\ref{SSS:kirBlUp} and especially Remark~\ref{R:pi-r-Def}; recall in particular that $\widetilde  X/\!\!/_{\tilde L}G$ denotes the Kirwan blowup, and $\widetilde X$ and $\tilde L$ are the iterated blowup and linearization, respectively, playing the roles of $\hat X$ and $\hat L$ in the discussion of the single blowup above.
From~\eqref{E:BR(t)-intro}, we have
\begin{equation}\label{E:BR(t)-gen}
IP_t(X/\!\!/_LG)=P_t(\widetilde  X/\!\!/_{\tilde L}G)-\sum_{R\in \calR}B_R(t),
\end{equation}
and our goal is to describe the polynomials $B_R(t)$ more precisely.

The relevant formula for computing the intersection cohomology of $X/\!\!/_LG$ from that of the full Kirwan blowup  $\widetilde X^{ss}/\!\!/_{\tilde L}G$,
generalizing~\eqref{E:IH-CorTerm1},  is~\cite[Thm.~3.1]{kirwanrational1}:
\begin{align}
\nonumber  \dim IH^i(X/\!\!/_LG)&=\dim H^i(\widetilde  X/\!\!/_{\tilde L}G)\\
\label{eq:IH-FullBU}
&-\sum_{R\in \calR} \sum_{p+q=i}\dim[H^p(Z_{R,\dim R+1}/\!\!/_LN_0^R)\otimes IH^{\hat q^R}(\PP(\calN_x^R)/\!\!/R)]^{\pi_0N^R}.
\end{align}
The   notation is explained after~\eqref{E:IH-CorTerm1}, where now the superscript $R$ indicates the corresponding object with respect to the given group $R$. For instance,
$\hat q^R=q-2$ for $q\le \dim \PP(\calN_x^R)/\!\!/R$ and $\hat q^R=q$ otherwise.
The notation $Z_{R,\dim R+1}$  indicates the strict transform of $Z_R$ in $X_{\dim R+1}$ under the appropriate sequence of blowups in the inductive process (see \S~\ref{SSS:kirBlUp} and especially Remark~\ref{R:pi-r-Def}).

\begin{rem}\label{R:pi0TrivBR}
If for some $R\in \calR$ the action of $\pi_0N^R$ on the tensor product $H^p(Z_{R,\dim R+1}/\!\!/_LN_0)\otimes IH^{\hat q}(\PP(\calN_x^R)/\!\!/R)$ is trivial on the second factor, one can simplify the corresponding term in~\eqref{eq:IH-FullBU} using Remark~\ref{R:pi0TrivB}.
In particular, if for all $R\in \calR$ the action of $\pi_0N^R$ on the tensor product $H^p(Z_{R,\dim R+1}/\!\!/_LN_0^R)\otimes IH^{\hat q}(\PP(\calN_x^R)/\!\!/R)$ is trivial on the second factor, then we  have
\begin{align*}
\dim IH^i(X/\!\!/_LG)&=\dim H^i(\widetilde X/\!\!/_{\tilde L}G)\\
&-\sum_{R\in \calR}\sum_{p+q=i}\dim H^p(Z_{R,\dim R+1}/\!\!/_LN^R) \dim IH^{\hat q^R}(\PP(\calN_x^R)/\!\!/R)
\end{align*}
so that
\begin{align}\label{E:IP-BR(t)gen}
IP_t(X/\!\!/_LG)&=P_t(\widetilde X/\!\!/_{\tilde L}G)-\sum_{R\in \calR}\underbrace{P_t(Z_{R,\dim R+1}/\!\!/_L N^R)b_R(t)}_{\text{``}B_R(t)\text{''}}
\end{align}
where $B_R(t)$  is the product indicated above, and  $b_R(t)$ is defined as in~\eqref{E:bR(t)-def}. If each of the $Z_{R,\dim R+1}/\!\!/_LN_0^R$ in
\eqref{eq:IH-FullBU} is simply connected, then the direct generalization of Remark~\ref{R:pi0TrivExpl} holds.
\end{rem}

\section[The GIT quotient for cubic threefolds]{The intersection cohomology of the GIT quotient for cubic threefolds}
We now apply this in the case of the GIT moduli space of cubic threefolds $\GIT$, for which~\eqref{E:BR(t)-gen} gives
\begin{equation}\label{E:GIT-MK-B}
IP_t(\GIT)=P_t(\MK)-\sum_{R\in \calR}B_R(t).
\end{equation}
In our case $\calR=\{R_{2A_5}\cong \CC^*, R_{3D_4}\cong (\CC^*)^2, R_c\cong \PGL(2,\CC)\}$, and to compute the terms $B_R(t)$, we utilize~\eqref{eq:IH-FullBU}, for most of which the computations have in fact already been done. Indeed, as we will see, we have already checked in the previous section that the quotients  $Z_{R,\dim R+1}/\!\!/_LN_0$ are simply connected, so that applying Remark~\ref{R:pi0TrivBR}, we can utilize~\eqref{E:IP-BR(t)gen}, and thus all that remains is to compute the intersection cohomology of the GIT quotients $\PP(\calN_x^R)/\!\!/R$. We will work out the terms $B_R(t)$ in the order of descending dimension of $R$, following the Kirwan blowup process.

\subsection{The  correction term $B_R(t)$ for $R_c\cong \PGL(2,\CC)$, the chordal cubic case}
\begin{pro}[The $B_R(t)$ term for the chordal cubic]\label{P:BRt-ChC}
For the group $R_c\cong \PGL(2,\CC)$, we have
\begin{enumerate}
\item $Z_{R_c}/\!\!/_{\calO(1)} N^{R_c}$ is a point.

\item $IP_t(\PP(\calN_x^{R_c})/\!\!/R_c)=1 +t^2 + 2t^4  +2t^6 +  3t^8 + 3t^{10} + 2t^{12}+ 2t^{14} + t^{16} + t^{18}$.

\item The action of ${\pi_0N^{R_c}}$ on $IH^\bullet(\PP(\calN_x^{R_c})/\!\!/R_c))$ is trivial.
\end{enumerate}
The term $B_{R_c}(t)$ is equal to
\begin{align*}
P_t(Z_{R_c}/\!\!/_{\calO(1)} N^{R_c}) b_{R_c}(t)&\equiv t^2+t^4+2t^6+2t^8+3t^{10} \mod t^{11}.
\end{align*}
\end{pro}
\begin{proof}
For brevity, let us write $R=R_c$,  $N=N^{R_c}$, and $\calN_x^{R_c}=\calN_x$.
By Proposition~\ref{P:ZRss}(2), $Z_R$ is a point, and thus~$Z_R/\!\!/_{\calO(1)} N$ is a point, proving (1).
Now the representation $\rho:R\to \calN_x $ was worked out in
Lemma~\ref{L:Rc-Nx-Rep} to be the representation of $\PGL(2,\CC)$ induced by the $\SL(2,\CC)$-representation $\Sym^{12}\CC^2$, where $\CC^2$ is the standard two-dimensional representation.
Thus $\PP(\calN_x )/\!\!/R$ is the GIT moduli space of  $12$ \emph{unordered} points in $\PP^1$, the intersection cohomology of which was worked out in~\cite[Table, p.~40]{kirwanhyp} (see also \cite{Bri} and \cite{LSa}), giving (2).

Now, since $Z_{R}/\!\!/_{\calO(1)} N$ is a point, we obtain from~\eqref{eq:IH-FullBU} that  the correction term is equal to  $B_R(t)=\sum_{i}t^i[\dim IH^{\hat i}(\PP(\calN_x )/\!\!/R)]^{\pi_0N}$, where the action is induced by an action of $\pi_0N$ on   $\PP(\calN_x )/\!\!/R$.
It follows from Lemma~\ref{L:RcNorm} that $\pi_0N$ consists of scalar matrices of the form $\zeta^i\cdot \operatorname{Id}$, where $\zeta$ is a primitive fifth root of unity.  Their action is evidently trivial, proving (3).
Finally, by Remark~\ref{R:pi0TrivBR} we conclude that  $B_{R}(t)=b_{R}(t)$, where $b_{R}(t)$ is worked out from (2) via~\eqref{E:bR(t)-def} to be
$$
b_{R}(t)=t^2+t^4+2t^6+2t^8+3t^{10} + 2t^{12}+ 2t^{14} + t^{16} + t^{18},
$$
completing the proof.
\end{proof}

\subsection{The  correction term $B_R(t)$ for $R_{3D_4}\cong (\CC^*)^2$, the $3D_4$ case}
\begin{pro}[The $B_R(t)$ term for the $3D_4$ cubic]\label{P:BRt-3D4}  For the group $R_{3D_4}\cong (\CC^*)^2$,
we have
\begin{enumerate}
\item $Z_{R_{3D_4}}/\!\!/_{\calO(1)} N^{R_{3D_4}}$ is a point.
\item $IP_t^{\pi_0N^{R_{3D_4}}}(\PP(\calN_x^{R_{3D_4}})/\!\!/R_{3D_4})=1 +t^2 + 2t^4  +3t^6 +  3t^8 + 3t^{10} + 3t^{12}+ 2t^{14} + t^{16} + t^{18}$.
\end{enumerate}
The
 term $B_{R_{3D_4}}(t)$ is given by
\begin{align*}
B_{R_{3D_4}}(t)&\equiv t^2+t^4+2t^6+3t^8+3t^{10} \mod t^{11}.
\end{align*}
\end{pro}
\begin{proof}
For brevity, let us write $R=R_{3D_4}$,  $N=N^{R_{3D_4}}$, and $\calN_x^{R_{3D_4}}=\calN_x$.
We have seen in Lemma~\ref{L:R3D4Norm1}(3) that $N^{ }$ acts transitively, so that  $Z_{R_{ }}/\!\!/_{\calO(1)} N^{ }$ is a point,
while the representation $\rho:R_{ }\to \calN_x^{ }$ was worked out in
Lemma~\ref{L:R3D4-Nx-Rep}.  The quotient $\PP(\calN_x^{ })/\!\!/R_{ }$ is a projective toric variety; the intersection cohomology can be worked out either torically, or via the general Kirwan process described in \S~\ref{S:PXss}.

The latter approach is quite elementary in this case, and so we sketch that here.  First, from the description of the weights of the action in Lemma~\ref{L:R3D4-Nx-Rep}, it is clear that there are no strictly semi-stable points in $\PP(\calN_x^{ })$ (this could also be deduced from the fact that we have locally already arrived at the full Kirwan blowup).  Thus the GIT quotient has at worst finite quotient singularities.  In any case, we have $IP_t(\PP(\calN_x^{ })/\!\!/R_{ })=P_t(\PP(\calN_x^{ })/\!\!/R_{ })=P_t^{R_{ }}(\PP(\calN_x^{ })^{ss})$.
But then using~\eqref{E:KD-(3.1)}, we have
\begin{align}
\nonumber IP_t(\PP(\calN_x^{ })/\!\!/R_{ })&=P_t^{R_{ }}(\PP(\calN_x^{ })^{ss})=P_t(\PP(\calN_x^{ }))P_t(BR_{ })-\sum_{0\neq \beta' \in \calB(\rho)} t^{2d(\beta')}P_t^{R_{ }}(S_{\beta'})\\
\nonumber&=P_t(\PP^{11})P_t(B(\CC^*)^2)-\sum_{0\neq \beta' \in \calB(\rho)} t^{2d(\beta')}P_t^{R_{ }}(S_{\beta'})\\
\label{E:IP3D4}&\equiv P_t(\PP^{11})P_t(B(\CC^*)^2) -3t^8 -3\dim H^1_{(\CC^*)^2}(S_{\beta '})t^9\mod t^{10}
\end{align}
where in \eqref{E:IP3D4}, in the notation of  Lemma~\ref{L:R(rho)3D4p1}, $\beta'$ is as in case (a), one of the  exactly three  $0\ne \beta'\in \calB(\rho)$ with $d(\beta)\le  4$, with isomorphic corresponding strata; since  $\PP(\calN_x^{ })/\!\!/R_{ }$ has dimension $9$, it suffices by Poincar\'e duality to compute up to $t^{9}$.
We can in fact conclude that the coefficient $-3\dim H^1_{(\CC^*)^2}(S_{\beta '})$ of $t^9$ is zero, since it must be non-negative (but it turns out this coefficient does not contribute to the final answer anyway).

Now, since $Z_{R}/\!\!/_{\calO(1)} N $ is a point, we obtain from~\eqref{eq:IH-FullBU} that  the correction term is equal to  $B_R(t)=\sum_{q}t^q[\dim IH^{\hat q}(\PP(\calN_x)/\!\!/R)]^{\pi_0N}$, where the action is induced by an action of $\pi_0N$ on   $\PP(\calN_x)/\!\!/R$.  It follows from Lemma~\ref{L:R3D4Norm1}(1) (see also the computations in Proposition~\ref{P:App-R=C*2}(2)) that $\pi_0N\cong S_3$ acts on $\PP(\calN_x )/\!\!/R$ via permutation of the coordinates $x_0,x_1,x_2$.  In the context of~\eqref{E:IP3D4}, the action on the cohomology of $\PP^{11}$ is trivial, the action permutes the $S_{\beta '}$, and acts on the torus $(\CC^*)^2$ in the following way.  The involution $\delta$ given by $x_0\leftrightarrow x_1$ acts by $(\lambda_0,\lambda_1)\mapsto (\lambda_1,\lambda_0)$, and the cyclic permutation $\sigma$ given by $x_0\mapsto x_1\mapsto x_2\mapsto x_0$ acts by $(\lambda_0,\lambda_1)\mapsto (\lambda_0^{-1}\lambda_1^{-1},\lambda_0)$.  Thus the action of $S_3$ on the cohomology of $B(\CC^*)^2$ is induced by the action of $S_3$ on the vector space $\QQ\langle c_1^{(1)},c_1^{(2)}\rangle$ by the matrices
$$
\delta = \left (
\begin{array}{cc}
0&1\\
1&0\\
\end{array}
\right),  \ \ \
\sigma = \left (
\begin{array}{cc}
-1&1\\
-1&0\\
\end{array}
\right).
$$
Looking at the characters, we see that this is the dihedral representation of $S_3$,
which has generating function $(1-t^4)^{-1}(1-t^6)^{-1}$ (the standard representation of $S_3$ has generating function $(1-t^2)^{-1}(1-t^4)^{-1}(1-t^6)^{-1}$, and is the direct sum of the  trivial representation and the dihedral representation).

Putting this back together with~\eqref{E:IP3D4},  we obtain
\begin{align*}
\scriptstyle
IP_t^{\pi_0N}(\PP(\calN_x )/\!\!/R_{})&\equiv (1-t^2)^{-1}\cdot (1-t^4)^{-1}(1-t^6)^{-1}-\frac{1}{3}\left(3t^8-3\dim H^1_{(\CC^*)^2}(S_{\beta '})  t^9\right) \\
&\equiv1+t^2+2t^4+3t^6+3t^8-\dim H^1_{(\CC^*)^2}(S_{\beta '}) t^9 \mod t^{10}.
\end{align*}
As noted above, $\dim H^1_{(\CC^*)^2}(S_{\beta '})=0$, but we have included it here again for clarity with respect to deducing the invariants from~\eqref{E:IP3D4}.   This concludes the proof of (2).

Now, computing the numerics as in~\eqref{E:bR(t)-def}, we find finally that
$$
B_R(t)\equiv t^2+t^4+2t^6+3t^8+3t^{10} \mod t^{11}$$
completing the proof (and explaining the claim above that the coefficient of $t^9$ in~\eqref{E:IP3D4} is irrelevant).
\end{proof}

\subsection{The  correction term $B_R(t)$ for $R_{2A_5}\cong \CC^*$, the $2A_5$ case} \label{S: BR-2A5}
In this section we denote by $\widehat Z_{R_{2A_5}}$ the strict transform of $Z_{R_{2A_5}}$ in the blowup along the chordal cubic locus.
\begin{pro}[The $B_R(t)$ term for $2A_5$ cubics]\label{P:BRt-2A5}  For the group $R_{2A_5}\cong \CC^*$,
we have
\begin{enumerate}
\item $\widehat Z_{R_{2A_5}}/\!\!/_{\calO(1)} N^{R_{2A_5}}\cong \PP^1$.

\item $IP_t^{\pi_0N^{R_{2A_5}}}(\PP(\calN_x^{R_{2A_5}})/\!\!/R_{2A_5})=1 +t^2 + 2t^4  +2t^6 +  3t^8  + 2t^{10}+ 2t^{12} + t^{14} + t^{16}$.
\end{enumerate}
The term $B_{R_{2A_5}}(t)$ is given by
\begin{align*}
B_{R_{2A_5}}(t)&=  t^2+2t^4+3t^6+4t^8+4t^{10} \mod t^{11}
\end{align*}
\end{pro}
\begin{proof}
For brevity, let us write $R=R_{2A_5}$,  $N=N^{R_{2A_5}}$, and $\calN_x^{R_{2A_5}}=\calN_x$.
The same argument as used in Proposition~\ref{P:App-R=C*p2} for the proof of Lemma~\ref{L:R2A5Norm}(3) shows that    $\widehat Z_{R}/\!\!/_{\calO(1)} N $ is a rational normal projective variety of dimension $1$; i.e., $\PP^1$.
Now the representation $\rho:R\to \calN_x$ was worked out in
Lemma~\ref{L:R2A5-Nx-Rep}.  The quotient $\PP(\calN_x^{R})/\!\!/R$ is a projective toric variety, and as in the previous case it is quite elementary to use our prior computations to compute its intersection cohomology via the general Kirwan process described in \S~\ref{S:PXss}.

Indeed, from the description of the weights of the action in Lemma~\ref{L:R2A5-Nx-Rep} (or from the fact that we have arrived at the Kirwan blowup) it follows that there are no strictly semi-stable points in $\PP(\calN_x)$, so that the GIT quotient has at worst finite quotient singularities.  Thus $IP_t(\PP(\calN_x)/\!\!/R)=P_t(\PP(\calN_x )/\!\!/R)=P_t^{R}(\PP(\calN_x)^{ss})$, and using~\eqref{E:KD-(3.1)} yields
\begin{align}
\nonumber IP_t(\PP(\calN_x)/\!\!/R) =P_t^{R}(\PP(\calN_x )^{ss})&=P_t(\PP(\calN_x ))P_t(BR)-\sum_{0\neq \beta' \in \calB(\rho)} t^{2d(\beta')}P_t^{R}(S_{\beta'})\\
\nonumber &=P_t(\PP^{9})P_t(B\CC^*)-\sum_{0\neq \beta' \in \calB(\rho)} t^{2d(\beta')}P_t^{R}(S_{\beta'})\\
\label{E:IP2A5}&\equiv P_t(\PP^{9})P_t(B\CC^*) \mod t^{9}
\end{align}
since from Lemma~\ref{L:R2A5-Nx-Rep} there are no $0\ne \beta'\in \calB(\rho)$ with $d(\beta)\le  4$; since  $\PP(\calN_x )/\!\!/R$ has complex dimension $8$, it suffices by Poincar\'e duality to compute modulo $t^9$.

Now, since $\widehat Z_{R}/\!\!/_{\calO(1)} N^{R}$ is simply connected, as described in Remark~\ref{R:pi0TrivBR} the action of $\pi_0N$ splits into an action on the base and an action on the fiber,
so that to compute $B_R(t)$, it suffices to compute $IP_t^{\pi_0N}(\PP(\calN_x)/\!\!/R)$.
We have seen in Lemma~\ref{L:R2A5Norm}  (see Proposition~\ref{P:App-R=C*p2} for more details) that $\pi_0N\cong \ZZ_2$, and acts on $\PP(\calN_x )/\!\!/R$ via permutation of the coordinates $x_0,x_1,x_2,x_3,x_4\leftrightarrow x_4,x_3,x_2,x_1,x_0$.  In the context of~\eqref{E:IP2A5}, the action of this involution on the cohomology of $\PP^{9}$ is trivial,  and acts on the torus $\CC^*$ by $\lambda\mapsto \lambda^{-1}$. Thus the action of $\ZZ_2$ on $H^\bullet(B\CC^*)=\mathbb Q[c_1]$ ($\deg c_1=2$) is induced by the action of $\ZZ_2$ on the vector space $\QQ\langle c_1\rangle$ by $c_1\mapsto -c_1$ (see Example \ref{Exa:SemDirExa}).  Thus $H^\bullet (B\mathbb C^*)^{\mathbb Z_2}=\mathbb Q[c_1^2]$, with generating function  $(1-t^4)^{-1}$.

Putting this together with~\eqref{E:IP2A5} we obtain
\begin{align}\label{IPquot2A5}
\nonumber IP_t^{\pi_0N}(\PP(\calN_x )/\!\!/R_{})&\equiv  (1-t^2)^{-1}\cdot (1-t^4)^{-1}\\
&\equiv  1+t^2+2t^4+2t^6+3t^8\mod t^{9}
\end{align}
completing the proof of (2).

From this we obtain the polynomial $b_R(t)=t^2+t^4+2t^6+2t^8+3t^{10}+2t^{12}+2t^{14}+t^{16}+t^{18}$, as in~\eqref{E:bR(t)-def} (with $c=9$).
We find finally that
\begin{align*}
B_R(t)&\equiv (1+t^2)\cdot( t^2+t^4+2t^6+2t^8+2t^{10})\\
&\equiv t^2+2t^4+3t^6+4t^8+4t^{10}\mod t^{11}.
\end{align*}
\end{proof}

\section{Putting the terms together to compute the cohomology of $\GIT$}
We now put together the results in the previous sections to complete  the proof of Theorem~\ref{teo:betti} for $\GIT$.  Recall that  the intersection cohomology of $\GIT$ satisfies Poincar\'e duality.  Consequently, as $\dim \GIT=10$, it suffices to compute $IP_t(\GIT)$ up to $t^{10}$.  We have:

\begin{align*}
&IP_t(\GIT) \equiv\\
&\equiv 1+4t^2+6t^4+10t^6+13t^8+15t^{10} &\text{(Kirwan blowup, Theorem~\ref{teo:betti})}\\
&\ \ \ \ -\ (t^2+t^4+2t^6+2t^8+3t^{10}) &\!\!\!\!\!\text{(Correction term, chordal cubic, Prop.~\ref{P:BRt-ChC})}\\
&\ \ \ \ -\ ( t^2+t^4+2t^6+3t^8+3t^{10}) &\text{(Correction term, $3D_4$ cubic, Prop.~\ref{P:BRt-3D4})}\\
&\ \ \ \ -\ (t^2+2t^4+3t^6+4t^8+4t^{10}) &\text{(Correction term, $2A_5$ cubic, Prop.~\ref{P:BRt-2A5})}\\
&\equiv 1+t^2+2t^4+3t^6+4t^8+5t^{10} \mod t^{11}.\!\!\!\!\!\!\!\!\!\!\!\!\!\!\!
\end{align*}
This completes the proof of Theorem~\ref{teo:betti} for $\GIT$.

\begin{rem}
Recall from~\eqref{eq:ARcontribution} that $P_t(\MK)=P^G_t(X^{ss})+\sum_{R\in \calR}A_R(t))$, where $A_R(t)$ is the difference of the 		``main term'' and the ``extra term'' (see~\eqref{E:finMainT} and~\eqref{E:finExtraT}).  From~\eqref{E:GIT-MK-B} we have $IP_t(\GIT)=P_t(\MK)-\sum_{R\in \calR}B_R(t)$, so that finally
$$
IP_t(\GIT)=P_t^G(X^{ss})+\sum_{R\in \calR}(A_R(t)-B_R(t)).
$$
From~\cite[Thm.~2.5]{kirwanrational1}, one knows that all the coefficients of the polynomial $\sum_{R\in \calR}(A_R(t)-B_R(t))$ are non-positive (note that in examples, the individual terms $A_R(t)-B_R(t)$ may have positive coefficients, e.g.,~\cite[6.5]{kirwanhyp}).
In our situation we have
\begin{align*}
A_{R_c}(t)-B_{R_c}(t)&\equiv \ \ \ 0 \mod t^{11}\\
A_{R_{3D_4}}(t)-B_{R_{3D_4}}(t)&\equiv -t^8 -t^{10}\mod t^{11}\\
A_{R_{2A_5}}(t)-B_{R_{2A_5}}(t)&\equiv \ \ \ 0 \mod t^{11}.
\end{align*}
In other words, up to $t^{10}$ (which is all that matters due to Poincar\'e duality),  the intersection cohomology of $\GIT$ agrees with the equivariant cohomology of the semi-stable locus, except for a correction in real codimensions~$8$ and $10$, coming from the blowup of the $3D_4$ locus.
\end{rem}

\section{The intersection cohomology of $\widehat{\calM}$}
Since the space $\widehat\calM$ is an intermediate step in the Kirwan blowup, obtained after one blows up only the chordal cubic point  $\Xi\in\GIT$, its intersection cohomology appears as an intermediate stage in our computations above, simply by taking only the blowup step corresponding to $R_c$, dealt with in Propositions~\ref{P:ET-ChC} and~\ref{P:BRt-ChC}. Thus we have
\begin{align*}
  IP_t(\widehat\calM)\equiv &1+t^2+2t^4+3t^6+4t^8+5t^{10}& \text{($IP_t(\GIT)$, Theorem~\ref{teo:betti})}\\
  &\ + t^2+t^4+2t^6+2t^8+3t^{10} &\text{(Correction term $B_{R_c}$, Proposition~\ref{P:BRt-ChC})}  \\
  \equiv &1+2t^2+3t^4+5t^6+6t^8+8t^{10} \mod t^{11}.\!\!\!\!\!\!\!\!\!\!\!\!\!\!\!\!\!\!\!\!\!\!\!\!\!\!\!\!\!\!
\end{align*}
This completes the proof of Theorem~\ref{teo:betti} for $\widehat {\calM}$.

%% file: sec6_IHball.tex
\chapter{The intersection cohomology of the ball quotient}\label{sec:IHball}
In this section we use the decomposition theorem in a different way to compute  the intersection cohomology of the Baily--Borel compactification $\BG$ of the ball quotient model $\calB/\Gamma$ of $\calM$.
For a paper addressing similar situations computing intersection cohomology of arithmetic quotients via the Kirwan blowup, see~\cite{KLW}.

\section{A special case of the decomposition theorem} \label{S:Decomp-Thm}
There is a special case of the decomposition theorem that will be quite useful for us in computing cohomology of the Baily--Borel and toroidal compactifications of the ball quotient.

We start by recalling that if  $f:\hat V\to V$ is a map from a variety $\hat V$ of dimension~$n$, smooth up to finite quotient singularities, to a possibly singular variety $V$ that is the blowup of a (not necessarily smooth) point $p\in V$ to an exceptional divisor $E\subset \hat V$, smooth up to finite quotient singularities (cf.~Diagram~\eqref{E:MV-setup}, with $C=p$),  then  the decomposition theorem gives the following. Writing $P_t(E)=\sum e_jt^j$, the decomposition theorem gives (e.g.,~\cite[Lem.~9.1]{GH-IHAg-17}):
\begin{equation}\label{eq:IHblowup}
P_t(\hat V)=IP_t(V)+e_{2n-2}t^2+e_{2n-3}t^3+\dots+e_{n+1}t^{n-1}+e_nt^n+e_{n+1}t^{n+1}+\dots +e_{2n-2}t^{2n-2}.
\end{equation}
In other words, the correction terms for $P_t(\hat V)$ in degree $\ge n=\dim \hat V$ are the Betti numbers of $E$, and in degree $\le n$ are set up so that Poincar\'e duality holds for $\hat V$.  We note that by Poincar\'e duality on $E$, we have $e_{2n-2-i}=e_i$.

\subsection{Comparing to Kirwan's computation}
We now observe that we have already seen another approach to computing $P_t(\hat V)$ in~\eqref{eq:IHblowup}.
Indeed, let $U\subseteq V$ be any open neighborhood of $p$  and let ${\hat U}:=f^{-1}(U)$  (cf.~Diagram~\eqref{E:MV-setup}, with $C=p$).   Then from~\eqref{E:IH-Lem1}, we have $IP_t(\hat V)=IP_t(V)+IP_t({\hat U})-IP_t(U)$.  Using the fact that $\hat V$ is smooth up to finite quotient singularities, we immediately get
$
P_t(\hat V)=IP_t(V)+P_t({\hat U})-IP_t(U)
$.
If we assume now that
${\hat U}$ retracts onto $E$, then we have
\begin{equation}\label{eq:IHblowup-1}
P_t(\hat V)=IP_t(V)+P_t(E)-IP_t(U).
\end{equation}

\begin{rem}
Combining~\eqref{eq:IHblowup} and~\eqref{eq:IHblowup-1} we find
\begin{equation}\label{E:BR-E-comp}
P_t(E)-IP_t(U)=e_{2n-2}t^2+e_{2n-3}t^3+\dots+e_{n+1}t^{n-1}+e_nt^n+\dots +e_{2n-2}t^{2n-2}.
\end{equation}
\end{rem}

We now compare this description of the decomposition theorem for the blowup of a point to the general setup of the Kirwan blowup machinery, which will provide us with an alternative viewpoint, and enable us to use some of the previous computations to deal with $\BG$.
To make this comparison,  let us return to Kirwan's general situation, for a single blowup, as in Chapter~\ref{SSS:IC1}.  We are further assuming that the morphism $\hat \pi_G:\hat X/\!\!/_{\hat L}G\to X/\!\!/_LG$ is such that $\hat X/\!\!/_{\hat L}G$ is smooth up to finite quotient singularities, that the center of the blowup $Z_R/\!\!/_{L}G$ is a point, and that the exceptional divisor $E/\!\!/_{\hat L}G$ is smooth up to finite quotient singularities.  In this situation,~\eqref{eq:IHblowup-1} translates to
\begin{equation}\label{E:KR1-C2.11-1}
P_t(\hat X/\!\!/_{\hat L}G)=IP_t(X/\!\!/_LG)+ \underbrace{P_t(E/\!\!/_{\hat L}G)-IP_t(\calN/\!\!/G)}_{B_R(t)}\,,
\end{equation}
where we are using the fact mentioned in deducing~\eqref{E:KR1-C2.11} from~\eqref{E:IH-Lem1}, that there is an appropriate open neighborhood $U\subseteq X/\!\!/_{L}G$ of $Z_R/\!\!/_{L}G$ with $U$ homeomorphic to $\calN/\!\!/G$.  Alternatively, this is~\eqref{E:KR1-C2.11} together with the fact
that
 $E/\!\!/_{\hat L}G\cong \hat{ \calN}/\!\!/G$~\cite[Lem.~2.15]{kirwanrational1}
 (see also~\cite[p.494]{kirwanrational1}).
We note here that in this special case, the formula~\eqref{eq:IHblowup} follows from~\eqref{E:KR1-C2.11-1} using~\eqref{E:KR1-E2.22(1)} and~\eqref{E:KR1-E2.22(2)}, and the shift in degrees mentioned after those equations.

The main  point for us, however, is the converse statement, that~\eqref{eq:IHblowup} gives an alternative approach to computing $B_R(t)$, assuming one knows $P_t(E/\!\!/_{\hat L}G)$.
Indeed, if $P_t(E/\!\!/_{\hat L}G)=\sum e_jt^j$, then combining \eqref{E:KR1-C2.11-1} and \eqref{E:BR-E-comp}, one sees that $B_R(t)= e_{2n-2}t^2+e_{2n-3}t^3+\dots+e_{n+1}t^{n-1}+e_nt^n+\dots +e_{2n-2}t^{2n-2}$, where here $n=\dim X/\!\!/_LG$.
Moreover,~\eqref{E:KR1-E2.22(2)} asserts that the cohomology of the exceptional divisor is given in our special situation as
\begin{equation}\label{E:DT-ExcD}
H^i(E/\!\!/_{\hat L}G)=\sum_{p+q=i}\left[H^p(Z_R/\!\!/_LN_0)\otimes IH^{q}(\PP(\hat{\calN}_x)/\!\!/R)\right]^{\pi_0N}.
\end{equation}
We also note that Remarks~\ref{R:pi0TrivB} and~\ref{R:pi0TrivExpl} give analogous statements for  $H^i(E/\!\!/_{\hat L}G)$ in this situation.

\section{The intersection cohomology of the ball quotient}

We start by recalling the birational maps relating the Kirwan blowup, GIT, and Baily--Borel compactifications, and introduce the notation for them (see also \S~\ref{sec:prelim}):
\begin{equation}\label{eq:diagramofmaps}
\xymatrix{
&\MK\ar[ldd]_\pi\ar[d]^f\ar[rdd]^g\\
&\widehat\calM\ar[ld]^{p}\ar[rd]_q\\
\GIT&&\BG
}
\end{equation}
Here $\pi$ is the Kirwan blowup, the geometry of which has been the focus of the paper up to this point. The space $\widehat\calM$ resolves the birational map from $\GIT$ to $\BG$, and has been described above in \S\ref{subsec:modulicubic}, following~\cite{act}. The map $p$ is the blowup of the point $\Xi\in\GIT$ that corresponds to the chordal cubic, to a divisor $D_h\subset\widehat\calM$.

Recall that~$\calT\subset \GIT$ is the rational curve of cubics with $2A_5$ singularities, which contains the point $\Xi$ corresponding to the chordal cubic. Let then $\widehat {\calT}=Z_{T,1}^{ss}\subset\widehat\calM$ be the strict  transform of the curve $\calT$ under the map $p$. The map $q$ then consists simply of blowing down the curve  $\widehat {\calT}$ to a point $c_{2A_5}$, this point being one of the two cusps of $\BG$. The other cusp $c_{3D_4}$ of $\BG$ corresponds to the $3D_4$ cubic. Thus the exceptional divisor of the map $g$ consists of two disjoint irreducible components $D_{3D_4}$ and $D_{2A_5}$, contracted to the two points that are the corresponding cusps of $\BG$.

To compute the intersection cohomology of~$\BG$, we apply the decomposition theorem to the map $g:\MK\to\BG$. The advantage of working with this map instead of with $q$ is that the domain $\MK$ is smooth up to finite quotient singularities, thus its intersection cohomology is equal to its cohomology, and the intersection complex is trivial. We will show that both exceptional divisors $D_{2A_5},D_{3D_4}\subset\MK$ of the map $g:\MK\to\BG$ are smooth up to finite quotient singularities, and thus we will be able to use~\eqref{eq:IHblowup}.

The divisor $D_{3D_4}$ is smooth up to finite quotient singularities, since it is obtained as an exceptional divisor in the Kirwan blowup process $\pi$, which is not modified after it is introduced (one can also check directly that in the divisor's description as a GIT quotient, it has no strictly semi-stable points at the stage it is first introduced, and is not modified by the subsequent blowups).
The divisor $D_{2A_5}$ is similarly seen to be smooth up to finite quotient singularities, since it is obtained as the last step of the Kirwan blowup process  (one can also check directly that in the divisor's description as a GIT quotient, it has no strictly semi-stable points at the stage it is first introduced).

We now compute the contribution to $P_t(\MK)$ due to the divisor $D_{3D_4}$.  In fact, in this case, rather than using~\eqref{eq:IHblowup}, and computing $P_t(D_{3D_4})$, we will use a slightly different approach.
Since $\pi(D_{3D_4})\in \calM^{GIT}$ and $f(D_{3D_4})\in \widehat {\calM}$  are both points and  $p$ is locally an isomorphism near those points,  and since $q$ is locally an isomorphism near the points $f(D_{3D_4})\in \widehat {\calM}$ and $c_{3D_4}\in\BG$, the term $B_{R_{3D_4}}(t)$ in Proposition~\ref{P:BRt-3D4} is precisely the contribution for the blowup $g$, over the cusp $c_{3D_4}$, namely
\begin{equation}\label{E:g3D4-contr}
 t^2+t^4   +2t^6  +3t^8  +3t^{10} \mod t^{11}
\end{equation}

For the contribution of the divisor $D_{2A_5}$ the situation is a little trickier, so that  we will have to use~\eqref{eq:IHblowup}, and compute $P_t(D_{2A_5})$.   The situation is more complicated because $f(D_{2A_5})=\widehat {\calT}$ is a curve, contracted by $q$ to the point $c_{2A_5}\in \BG$, and moreover, since the map $p$ is not an isomorphism in a neighborhood of $\calT$ and $\widehat {\calT}$ (one must take into account the blowup of the point $\Xi$).  Nevertheless, we have essentially already computed $P_t(D_{2A_5})$.  Indeed, from~\eqref{E:DT-ExcD}, Proposition~\ref{P:BRt-2A5}, and Remark~\ref{R:pi0TrivExpl} we see that
\begin{align}
\label{E:KirAlt2A5} P_t(D_{2A_5})&=P_t(\PP^1)\, IP_t^{\pi_0N}(\PP(\calN_x)/\!\!/R) & \\
\nonumber &\equiv (1+t^2)(1+t^2+2t^4+2t^6+3t^8)\  \mod t^{9} &\ \text{(From~\eqref{IPquot2A5})}\\
\label{E:g2A5-contr} &\equiv 1+2t^2+3t^4+4t^6+5t^8\ \mod t^9&.
\end{align}

Summarizing, we obtain
\begin{align*}
&IP_t(\BG) \equiv\\
&\equiv 1+4t^2+6t^4+10t^6+13t^8+15t^{10} &\text{(Kirwan blowup $\MK$, Theorem~\ref{teo:betti})}\\
&\ \ \ \                        -( t^2+t^4   +2t^6  +3t^8  +3t^{10})&\text{($D_{3D_4}$ contribution,~\eqref{E:g3D4-contr})}\\
&\ \ \ \                        -( t^2+2t^4 +3t^6  +4t^8  +5t^{10})&\text{($D_{2A_5}$ contribution,~\eqref{E:g2A5-contr},~\eqref{eq:IHblowup})}\\
&\equiv  1+2t^2+3t^4 +5t^6  +6t^8  +7t^{10}\mod t^{11}.\!\!\!\!\!\!\!\!\!\!\!\!\!\!\!
\end{align*}
\begin{rem}
We remark that thus $IP_t(\widehat\calM)-IP_t(\BG)=t^{10}$; in particular, this difference is not zero. Even though the map $q$ is small (being a contraction of a curve in a 10-fold to a point), it does not induce an isomorphism in intersection cohomology, which is possible for a small map whose domain is not smooth.
\end{rem}

\begin{rem}\label{R:Kirwan2A5}
Frances Kirwan suggested a small variation on how to establish~\eqref{E:KirAlt2A5}.
 The main claim is that the map $f|_{D_{2A_5}}:D_{2A_5}\to\widehat{\calT} =\PP^1$ is a fibration.  To see this,
one interprets this map as a GIT quotient of the locus in $\PP^{34}$ of all $2A_5$ cubics by the normalizer $N_{2A_5}$. The extra involution $\sigma$ that appears in the stabilizer of the special point (as explained in the appendix) is contained in $N_{2A_5}$ (recall that $\sigma$ acts diagonally, while the normalizer is a $\ZZ/2\ZZ$ extension of the maximal diagonal torus). Thus in thinking about the GIT quotient by $N_{2A_5}$, this extra involution plays no role.  Thus we have a fibration, with fibers isomorphic to $\PP(\calN_x)/\!\!/R)/\pi_0N$.
 In~\eqref{IPquot2A5} we have computed the invariant part of the  cohomology of these fibers. Notice, crucially, that this cohomology is zero in all odd degrees. Thus the spectral sequence computing the cohomology of $D_{2A_5}$, as a fibration over $\PP^1$, is completely degenerate, and we obtain~\eqref{E:KirAlt2A5}.
\end{rem}

%% file: sec7_IHtoroidal.tex
\chapter{The cohomology of the toroidal compactification}\label{sec:toroidal}
In this section we will compute the cohomology of the toroidal compactification of the ball quotient, completing the proof of Theorem~\ref{teo:betti}.

The starting point of our discussion is the~\cite{act} ball quotient model $\calB/\Gamma$ for the moduli of cubic threefolds (see Chapter~\ref{sec:prelim}). We recall that this locally symmetric variety $\calB/\Gamma$ is associated to an Eisenstein lattice $\Lambda$ (see \S\ref{subsec:Eisenstein}, esp.~\eqref{eq_def_lambda}) of signature $(1,10)$, that we will review below.  Similar to the better known case of $K3$ surfaces, the cusps of the Baily--Borel compactification $(\calB/\Gamma)^*$ correspond to $\Gamma$-conjugacy classes of primitive isotropic subspaces in $\Lambda$. By~\cite{act}, there are exactly two cusps, that we label $c_{2A_5}$ and $c_{3D_4}$ respectively, in accordance with Theorem~\ref{resgitball}. The toroidal compactification, which we  denote here by~$\oBG$, is a partial resolution $\oBG\to(\calB/\Gamma)^*$ of the two cusps. Unlike in the Siegel or the orthogonal case,
there are no choices involved, and one can thus speak about {\em the} toroidal compactification $\oBG$. The reason for the uniqueness is that all cones which are used in the toroidal compactification have
dimension $1$. For this reason it also holds that  $\oBG$ only has finite quotient singularities and, therefore, its singular cohomology and intersection cohomology coincide. From the general theory (see~\cite{AMRT}; see also~\cite{beh} for the ball quotient case), the two exceptional divisors of $\oBG\to(\calB/\Gamma)^*$ are quotients of $9$-dimensional abelian varieties by finite groups (Proposition~\ref{prop_structure_tor}). In each of the two cases occurring here, the relevant abelian variety  is in fact $(E_\omega)^9$, where $E_\omega$ is the elliptic curve with $j$-invariant equal to $0$, namely the quotient of the complex numbers by the Eisenstein integers. The content of the first subsection of this section is the identification of the two finite groups $\Gamma_{2A_5}$ and $\Gamma_{3D_4}$ acting on $(E_\omega)^9$ for the two cusps $c_{2A_5}$ and $c_{3D_4}$ respectively. In the second subsection, by adapting a well-known theorem of Looijenga (see~\cite{Lroot},~\cite[Thm. 2.7]{FMW}), we are able to compute the cohomology of these two exceptional divisors $(E_\omega)^9/\Gamma_{2A_5}$ and $(E_\omega)^9/\Gamma_{3D_4}$. Finally, combining this with an application of the decomposition theorem, we conclude the  proof of Theorem~\ref{teo:betti}.

\section{The arithmetic of the two cusps of $\calB/\Gamma$}\label{S:ArithmeticCusp}
 In this subsection, we discuss the structure of the toroidal compactification for the ball quotient model $\calB/\Gamma$ for cubic threefolds. To start, we briefly recall the notion of Eisenstein lattice, and the relationship to the even $\ZZ$-lattices endowed with an order $3$ isometry. We follow with the classification of the cusps of $(\calB/\Gamma)^*$, which is closely related to the classification of the Type II boundary components for the Baily--Borel compactification for cubic fourfolds (see~\cite[\S6.1]{laza}). Using some ideas from~\cite{beh}, we can describe the two exceptional divisors of $\oBG\to(\calB/\Gamma)^*$ (Proposition~\ref{prop_structure_tor}).

\subsection{Eisenstein Lattices}\label{subsec:Eisenstein} Let $\calE$ be the ring of Eisenstein integers
$$
\calE:=\ZZ[\omega], \ \ \ \omega= e^{\frac{2 \pi i}{3}}.
$$
By an {\it  Eisenstein lattice} $\calG$ we understand a free $\calE$ module, endowed with an hermitian form taking values in $\mathcal E$.
Throughout, we will make the \emph{additional convention that the Hermitian form takes values in $\theta \calE$}, where $\theta=\omega-\omega^2(=\sqrt{-3})$, i.e.,
$$\langle -,-\rangle:\mathcal G\times \mathcal G\to \theta \mathcal E;$$
this should be understood as an analogue of  even lattices over $\ZZ$. In particular, note that then $\|x\|^2=\langle x,x\rangle\in 3\ZZ$.

Associated to an Eisenstein lattice $\calG$, there is a usual $\ZZ$-lattice, that we denote $\calG_\ZZ$. Simply, $\calG_\ZZ$ is the underlying free $\ZZ$-module. On $\calG_\ZZ$, we define a bilinear symmetric form
$$
(-,-):= - \frac{2}{3} \operatorname{Re} \langle -,- \rangle: \calG_{\ZZ} \times \calG_{\ZZ} \to \ZZ.
$$
Under our convention on the hermitian form, $\calG_\ZZ$ is an even lattice. Note that $\calG_\ZZ$ comes endowed with an order $3$ isometry $\rho\in \operatorname{O}(\calG_\ZZ)$,
namely
\begin{equation}\label{eqrho}
\rho(x):=\omega\cdot x,
\end{equation}
where the multiplication is the multiplication by scalars in the Eisenstein module $\calG$. Clearly, $\rho$ acts on $\calG_\ZZ$ fixing only the origin (i.e.,  $\rho(x)\neq x$ for any $x\neq0$).
Conversely, given an even lattice $\calG_\ZZ$ together with an order $3$ isometry $\rho$ (fixing only the origin), we can define an Eisenstein lattice $\calG$ reversing the process above. More precisely, the Eisenstein structure on $\calG_{\ZZ}$ is determined by~\eqref{eqrho}. Then, the hermitian form is given by
\begin{equation}\label{arr:hermitianform}
\operatorname{Re} \langle x,y \rangle=-\frac{3}{2}\cdot(x,y),\quad
i\cdot\operatorname{Im} \langle x,y \rangle=\frac{\theta}{2}\cdot((\rho-\rho^2)x,y).
\end{equation}
Finally, we note that an isometry $\phi$ of $\calG$ induces an isometry of $\calG_\ZZ$, commuting with $\rho$, and conversely. In other words,
\begin{equation}\label{eq_rel_iso}
\Aut(\calG,\langle-,-\rangle)=\{\phi\in \operatorname{O}(\calG_\ZZ)\mid \phi\rho=\rho\phi\}.
\end{equation}
By abuse of notation, we will denote by
$$\operatorname{O}(\calG):=\Aut(\calG,\langle-,-\rangle)$$ the group of isometries (N.B. $\operatorname{O}(\calG)$ is a unitary group, and not an orthogonal group).

We now introduce the Eisenstein lattices relevant to our discussion. First, we consider the following definite lattices (defined in terms of Gram matrices):
$$\calE_1: (3),\ \ \calE_2: \begin{pmatrix}
3&\theta\\
\bar{\theta}&3\\
\end{pmatrix}, \ \ \calE_3: \begin{pmatrix}
3&\theta&0\\
\bar{\theta}&3&\theta\\
0&\bar\theta&3
\end{pmatrix}, \ \ \calE_4: \begin{pmatrix}
3&\theta&0&0\\
\bar{\theta}&3&\theta&0\\
0&\bar{\theta}&3&\theta\\
0&0&\bar\theta&3
\end{pmatrix}\,\,.$$
The underlying $\ZZ$ lattices $(\calE_i)_\ZZ$ are $A_2(-1)$, $D_4(-1)$, $E_6(-1)$, and $E_8(-1)$ respectively. Conversely, we note that the lattices $A_2(-1)$, $D_4(-1)$, $E_6(-1)$ and $E_8(-1)$ admit (up to conjugacy) a unique order $3$ isometry fixing only the origin, and thus they admit a unique Eisenstein structure (e.g.~\cite[Lem.~3]{HKN}).

We also consider the indefinite (signature $(1,1)$) lattice $\calH$ defined by
$$\calH: \begin{pmatrix}
0&\theta\\
\bar{\theta}&0\\
\end{pmatrix}\,,$$
whose underlying $\ZZ$ lattice is $2U$ (two copies of the hyperbolic plane).

The Eisenstein lattice used by Allcock--Carlson--Toledo~\cite{act} to define the ball quotient model $\calB/\Gamma$ for the moduli of cubic threefolds is
\begin{equation}\label{eq_def_lambda}
\Lambda:=  \calE_1 +  2\calE_4 +  \calH,
\end{equation}
with associated $\ZZ$ lattice
$$
 \Lambda_{\ZZ} \cong A_2(-1) + 2E_8(-1) + 2U,
$$
which is precisely (up to a sign) the lattice of the primitive middle cohomology of a smooth cubic fourfold. Returning to the construction of $\calB/\Gamma$, we recall
$$
\calB:=\calB_{10}:=\{[z ]: z^2>0\}^+\subset \PP(\Lambda\otimes_{\calE}\CC),
$$
and $\Gamma=\operatorname{O}(\Lambda)$ acts naturally (properly discontinuously) on $\calB$.

\smallskip

Finally, let us recall some basic terminology from Nikulin's theory for even $\ZZ$-lattices that will be needed later. Let $M$ be an even non-degenerate $\ZZ$-lattice. The {\it dual lattice}
is $M^\vee=\Hom_\ZZ(M,\ZZ)$. Using the quadratic form, the dual $M^\vee$ has the following
description
$$M^\vee=\{w\in M\otimes_\ZZ \QQ\mid (v,w)\in \ZZ \textrm{ for all } v\in M\}\,,$$
in particular $M\subset M^\vee\subset M\otimes_\ZZ \QQ$. The {\it discriminant group} is the finite group $A_M:=M^\vee/M$. A key insight of Nikulin is that the quadratic form on $M$ induces a finite quadratic form
$$q_M:A_M\to \QQ/2\ZZ\,\,.$$
For example, if $M=E_6(-1)$, then $A_{M}\cong \ZZ/3$ and $q_M(\xi)=-\frac{4}{3}\in \QQ/2\ZZ$ for $\xi$ a  generator of $A_M$.
We also recall that for  $v\in M$, the {\it divisibility} $\divisore v$ is the positive generator of the ideal $(v,M)\subset \ZZ$, i.e., the biggest natural number by which all integers $(v,m)$ for $m\in M$ are
divisible. Note that $\frac{v}{\divisore{v}}\in M^\vee$, and then via the projection $M^\vee\to A_M=M^\vee/M$ we obtain an element in $A_M$. In fact, every element of $A_M$ arises in this way. If $v\in M$ is primitive, then the order of (the class of) $\frac{v}{\divisore{v}}$ in $A_M$ is precisely $\divisore{v}$. Returning to the $M=E_6(-1)$ example,  we see that $\divisore{v}\in\{1,3\}$ for $v\in M$ primitive. Furthermore, if $v$ is primitive with $\divisore{v}=3$, then (the class of) $\frac{v}{3}$ is a generator of $A_M$. Using $q_M\left(\frac{v}{3}\right)=-\frac{4}{3}\in \QQ/2\ZZ$, one concludes $v^2=-12 \pmod{18}$. For $M=E_6(-1)$, there exists indeed an element $v$ of norm $-12$ and divisibility $3$.

\begin{rem}\label{R:DualEisen}
Most of the above discussion can be adapted to the case of Eisenstein lattices. Here, for $v$ in an Eisenstein lattice $\calG$,  we  define  $\divisore v$ as the generator of the ideal $\langle \calG,v\rangle\subset \theta \calE$.
Clearly  $\theta$ divides $\divisore v$, and $\divisore v$ divides $ \|v\|^2$. Similarly,  following the conventions in~\cite[p.285]{ALeech},~\cite[p.8645]{Ma} we define $\mathcal G^*=\operatorname{Hom}_{\mathcal E}(\mathcal G,\mathcal E)$, and under the  identification $\mathcal G^*=\{\nu\in \mathcal G\otimes_{\mathcal E}\mathbb Q(\omega): \langle \lambda ,\nu\rangle \in  \mathcal E\  \forall\  \lambda \in \mathcal G\}$, we naturally obtain a Hermitian form on $\mathcal G^*$ making it an Eisenstein lattice (with our conventions).
Under the $\theta$-value assumption on the Hermitian form,
it holds that
$$\frac{1}{\theta}\calG\subset \calG^*\subset \calG\otimes_\calE\QQ(\omega)\,.$$
Thus, the natural ``unimodularity'' condition in this setup is $\theta\calG^*\cong \calG$. For the lattices considered here, $\calE_4$ and $\calH$ satisfy this condition, while $\calE_1$, $\calE_2$, and $\calE_3$ do not.  Note that under the natural identification $\mathcal G\otimes_{\mathcal E}\mathbb Q(\omega)=\mathcal G\otimes_{\mathcal E}\mathcal E\otimes_{\mathbb Z}\mathbb Q=\mathcal G\otimes_{\mathbb Z}\mathbb Q$, we have $(\theta\mathcal G^*)_{\mathbb Z}=(\mathcal G_{\mathbb Z})^\vee$~\cite[p.8645]{Ma}, so that the notions of unimodularity in the Eisenstein, and underlying integral case, agree.  Finally, we observe that under our  $\theta$-value assumption on the Hermitian form, it is typically more convenient to work with $\mathcal G':= \operatorname{Hom}_{\mathcal E}(\mathcal G,\theta \mathcal E)= \{\nu\in \mathcal G\otimes_{\mathcal E}\mathbb Q(\omega): \langle \lambda ,\nu\rangle \in \theta \mathcal E\  \forall\  \lambda \in \mathcal G\}$.  Clearly $\mathcal G'=\theta\mathcal G^*$, the unimodularity condition becomes $\mathcal G'\cong \mathcal G$, and we have $(\mathcal G')_{\mathbb Z}=(\mathcal G_{\mathbb Z})^\vee$.
\end{rem}

\subsection{Identification of the two cusps of $\BG$} 
As mentioned above, a cusp of the Baily--Borel compactification corresponds to a primitive isotropic subspace (automatically of rank $1$) $\calF\subset \Lambda$, considered up to the action of $\Gamma$. As a consequence of~\cite{act} (see Theorem~\ref{resgitball}), we know that there are precisely two cusps, and thus two possible $\calF$. Our goal here is to describe these two cases explicitly. First, we note that a standard invariant that in many cases suffices to distinguish the Baily--Borel cusps is the definite lattice $\calF^\perp/\calF$ of rank $9$, where $\calF^\perp$ denotes the orthogonal complement of $\calF$ in $\Lambda$ (N.B. since $\calF$ is isotropic, $\calF\subset \calF^\perp$).
A weaker invariant is the associated $\ZZ$-lattice $(\calF^\perp/\calF)_\ZZ$ (negative definite of rank $18$). An even weaker invariant is $R\left((\calF^\perp/\calF)_\ZZ\right)$, i.e. the sublattice of $(\calF^\perp/\calF)_\ZZ$ spanned by roots (i.e., $-2$ classes). In our situation this weak invariant suffices to distinguish the cusps, as there are only two of them.
\begin{lem}\label{lem_inv_iso}
With notation as above (e.g., $\calF$ is the isotropic subspace associated to the corresponding cusp), the following hold:
\begin{itemize}
\item[i)] for the cusp $c_{2A_5}$ of $(\calB/\Gamma)^*$, $R\left((\calF^\perp/\calF)_\ZZ\right)\cong 2E_8(-1)+A_2(-1)$;
\item[ii)] for the cusp $c_{3D_4}$ of $(\calB/\Gamma)^*$, $R\left((\calF^\perp/\calF)_\ZZ\right)\cong 3E_6(-1)$.
\end{itemize}
\end{lem}
\begin{proof}
By Theorem~\ref{resgitball}, we know that $c_{2A_5}$ and $c_{3D_4}$ correspond to semi-stable cubic threefolds with $2A_5$ singularities and  $3D_4$ singularities respectively. The ball quotient model $\calB/\Gamma$ for cubic threefolds is obtained by considering the eigenperiods (see~\cite{DK}) of cubic fourfolds with a $\mu_3$ action (namely, to a cubic threefold $V(f_3(x_0,\dots,x_4))$ one associates the cubic fourfold   $V(f_3(x_0,\dots,x_4)+x_5^3)$). This construction is compatible with GIT and Baily--Borel compactifications. One immediately checks that the construction associates to a cubic threefold with $2A_5$ (resp. $3D_4$) singularities a semi-stable cubic fourfold with $2\widetilde E_8$ (resp. $3\widetilde E_6$) singularities. Now the claim follows from the classification of Type II boundary components for cubic fourfolds (and the discussion of their geometric meaning) in~\cite[\S6.1]{laza}).
\end{proof}

\begin{rem}\label{rem_comparebb}
For further reference, let us note the following. Let  $(\calD/\Gamma')^*$ be the Baily--Borel compactification for the moduli of cubic fourfolds (as discussed, this is associated to the lattice $\Lambda_\ZZ(-1)=2E_8+A_2+2U$). By construction, there exists a natural morphism $$A:(\calB/\Gamma)^*\to (\calD/\Gamma')^*\,,$$
which is generically an embedding (in fact, a normalization of the image). The two cusps of the Baily--Borel compactification $(\calB/\Gamma)^*$ map to points on the Type II components of $(\calD/\Gamma')^*$ (corresponding to the fact that $\calF_\ZZ$ is an isotropic rank $2$ subspace of $\Lambda_\ZZ$). The lemma above says that $A(c_{2A_5})\in II_{2E_8+A_2}$ and $A(c_{2A_5})\in II_{3E_6}$ respectively (where $II$ indexed by a root lattice denotes a Type II boundary component in $(\calD/\Gamma')^*$). It is well known (in full generality) that
the Type II boundary components of $(\calD/\Gamma')^*$ are modular curves, while here in fact $II_{2E_8+A_2}, II_{3E_6}\cong \mathfrak h/\SL(2,\ZZ)$. It is then clear
  (by construction) that the  $A(c_{2A_5})$ and $A(c_{3D_4})$ map to the special points on $II_{2E_8+A_2}$ and  $II_{3E_6}$ respectively corresponding to $j$-invariant equal to $0$.
\end{rem}

It remains now to identify two possibilities of primitive isotropic subspaces $\calF\subset \Lambda$ such that the associated invariant $R((\calF^\perp/\calF)_\ZZ)$ is as in Lemma~\ref{lem_inv_iso}.
The first case is immediate.

\begin{lem}[$2A_5$ cusp]
If $\calF$ is an isotropic subspace in the summand $\calH$ of  $2\calE_4+\calE_1+\calH= \Lambda$, then $\calF^\perp/\calF\cong  2\calE_4+\calE_1$. Hence $\calF$ defines the cusp $c_{2A_5}$.
\end{lem}

\begin{proof}
This is clear.
\end{proof}

For the second case (cusp $c_{3D_4}$), the argument is more lengthy, and we begin with some preliminary discussion.  We know that once we have found $\mathcal F$, then we will have $R((\calF^\perp/\calF)_\ZZ)\cong 3E_6$, and in fact $(\calF^\perp/\calF)_\ZZ$ is a lattice in the genus of $2E_8+A_2$ (see~\cite[Ch. 5]{scattone} and~\cite[\S6.1]{laza}). In other words, $(\calF^\perp/\calF)_\ZZ$ is an index $3$ overlattice of $3E_6$. As noted in Lemma~\ref{lem_inv_iso}, a semi-stable cubic threefold with $3D_4$ singularities leads (via the Allcock--Carlson--Toledo construction of adding a new monomial $x_5^3$ to the defining equation) to a semi-stable cubic fourfold with $3\widetilde E_6$ singularities, which in turns leads to the $3E_6$ sublattice in the vanishing cohomology. Thus, we see that the order $3$ isometry $\rho$ on $(\calF^\perp/\calF)_\ZZ$ defining the Eisenstein lattice $\calF^\perp/\calF$ is compatible with order $3$ isometries on each of the $E_6$ factors (giving $\calE_3$ lattices). In other words,  $\calF^\perp/\calF$ is an index $3$ overlattice of $3\calE_3$. We will now define an index $3$ overlattice $\widetilde{3\calE_3}$ of $3\calE_3$ (a posteriori, indeed $\calF^\perp/\calF\cong \widetilde{3\calE_3}$). We start with three copies of $\calE_3$, or equivalently with three copies of $E_6(-1)$ each endowed with an isometry $\rho$ of order $3$ (fixing only the origin). By Nikulin theory, see~\cite[\S1.4]{nikulin}
 there exists an index $3$ overlattice $\widetilde{3E_6(-1)}$ of $3E_6(-1)$.
Indeed, overlattices of $3E_6(-1)$ correspond to isotropic subgroups $H \subset A_{3E_6(-1)}$ of the discriminant group by taking the inverse image of $H$ under the projection $(3E_6(-1))^{\vee} \to A_{3E_6(-1)}$.
Here we take the subgroup $H$ generated by the diagonal embedding of $A_{E_6(-1)}$ into $A_{(3E_6(-1))}$. In fact, up to isometries of $3E_6(-1)$ this is the only isotropic subgroup.
Explicitly,  we can find elements $z_i\in E_6(-1)^{(i)}$ (the $i^{th}$ copy) with $z_i^2=-12$ and $\divisore z_i=3$.
Then $\widetilde{3E_6(-1)}$ is the lattice generated by $3E_6(-1)$ and $\frac{z_1+z_2+z_3}{3}$ (inside $3E_6(-1)\otimes \QQ$).  Note also that for each of the $E_6(-1)$ components, the isometry $\rho$ acts trivially on the discriminant (simply, the automorphism group of the discriminant $A_{E_6(-1)}\cong \ZZ/3$ has order $2$, while $\rho$ has order $3$).
This means  that $\rho$ (defined component-wise) on $3E_6(-1)$ extends to an isometry of $\widetilde{3E_6(-1)}$ whose only fixed point is the origin, thus giving the Eisenstein lattice $\widetilde{3\calE_3}$ (recall the hermitian form is determined as in~\eqref{arr:hermitianform})
and in fact this
is the only such overlattice.

\begin{rem}\label{rem_discr_3e6}
Note that $z_i$ as above are chosen such that $\frac{z_i}{3}$ generate the discriminant of the respective copy of $E_6(-1)$.
The condition $\divisore z_i=3$
 guarantees that $\frac{z_i}{3}\in E_6(-1)^\vee$, and then its projection into  $A_{E_6(-1)}\cong \ZZ/3\ZZ$ is a generator. Note also that $z_i$ is divisible by $\theta$
when we view $E_6(-1)(=\calE_3)$ as an Eisenstein lattice; indeed, we compute
$$\frac{1}{3} (\theta\cdot z_i) = \frac{1}{3}(\rho(z_i)-\rho^2(z_i))=\frac{\rho(z_i)}{3}-\frac{\rho^2(z_i)}{3}=0\in A_{E_6(-1)}\,,$$
or equivalently $\theta \cdot z_i=3v_i$ for some $v_i\in E_6(-1)^{(i)}$, and then $z_i=-\theta v_i$ (this relation makes sense even over $\ZZ$, by interpreting $\theta$ as the endomorphism $\rho-\rho^2$; over $\calE$, $\rho$ is the multiplication $\omega$, and thus $\rho-\rho^2$ is the multiplication by $\theta\in \calE$). Returning to $\widetilde{3E_6(-1)}$ and the companion Eisenstein lattice $\widetilde{3\calE_3}$, we note that the discriminant of $\widetilde{3E_6(-1)}$ is $\ZZ/3$ and it is generated by the class of $\frac{z_1-z_2}{3}$. Furthermore, the following hold
\begin{itemize}
\item $(z_1-z_2)^2=z_1^2+z_2^2=-24$.
\item $z_1-z_2\in \widetilde {3E_6(-1)}$ is primitive and $\divisore(z_1-z_2)=3$ (even in $\widetilde{3E_6(-1)}$).
\item $(z_1-z_2)=\theta\cdot (v_1-v_2)$.
\end{itemize}
In terms of the Hermitian norm, note that $\|z_1-z_2\|^2=36$, and then $\|v_1-v_2\|^2=12$.
\end{rem}

The lattice $\widetilde{3\calE_3}$ satisfies the following key property.

\begin{pro}\label{prop_3e6tilde}
There is an isomorphism of indefinite Eisenstein lattices
\begin{equation}
\Lambda(\cong \calE_1 + 2\calE_4 + \calH) \cong \widetilde{3\calE_3} + \calH.
\end{equation}
\end{pro}
\begin{proof}
Let us first note that the underlying $\ZZ$-lattices are indeed isomorphic, i.e., forgetting the Eisenstein structure, it holds that
$$
A_2(-1) + 2E_8(-1) +2U \cong \widetilde{3E_6(-1)} +2U\,.$$
Indeed, the two lattices have the same signature and isomorphic discriminant groups (together with the quadratic from on it), thus they are in the same genus (see~\cite[Cor.~1.9.4]{nikulin}). Since the signature is indefinite, this genus  contains only one element (see~\cite[Cor.~1.13.3]{nikulin}).

To lift this isometry to an isometry of Eisenstein lattices, we would need to know that (up to the action of the orthogonal group) there exists a unique Eisenstein structure on the $\ZZ$-lattice $A_2(-1) + 2E_8(-1) +2U$. We were not able to find such a result in the literature, but a related result is known: {\it an even indefinite unimodular lattice (e.g., $2E_8(-1) +2U$) admits at most one Eisenstein lattice structure} (see~\cite[Lem.~2.6]{basak}). Since $\calE_1 + 2\calE_4 + \calH$ is the direct sum of a unimodular Eisenstein lattice $2\calE_4 + \calH$ (with underlying  unimodular $\ZZ$-lattice $2E_8(-1)+2U$) and a rank $1$ lattice $\calE_1$ spanned by a norm $3$ vector $v$, it suffices to find a vector $\widetilde{3\calE_3} + \calH$ such that $\|w\|^2=3$ and $(w)^\perp_{\widetilde{3\calE_3} + \calH}$ is unimodular. This in turn is equivalent to $\|w\|^2=3$ and $\divisore w=3$.
By the discussion of Remark~\ref{rem_discr_3e6}, one sees that $w'=v_1-v_2$ (with the notations of the remark) satisfies the right divisibility condition, but not the norm condition. However, we can correct the norm by taking $w=w'+\theta u$ with $u\in \calH$ and $\|u\|^2=-3$. This completes the proof.
\end{proof}

As a consequence of the above proposition, we conclude:
\begin{cor}[$3D_4$ cusp]
If $\calF$ is an isotropic subspace of the summand $\calH$ of the sum
$\widetilde{3\calE_3}+\calH\cong \Lambda$, then $\calF^\perp/\calF\cong  \widetilde{3\calE_3}$. Hence $\calF$ defines the cusp $c_{3D_4}$.
\end{cor}
\begin{proof}
This is clear.
\end{proof}

\subsection{Structure of the two boundary divisors of $\oBG$}
We denote the two boundary divisors of $\oBG$ corresponding to the cusps $c_{2A_5}$ and $c_{3D_4}$ by $T_{2A_5}$ and $T_{3D_4}$, respectively. To be able to treat both cusps simultaneously we write
\begin{equation}\label{equ:decomp}
\Lambda = \calG+ \calH
\end{equation}
where $\calG =  \calE_1 +  2\calE_4$ or $\calG=  \widetilde{3\calE_3}$, respectively. As before, we can choose a rank $1$ primitive isotropic subspace $\calF\subset \calH$, and then $\calG \cong \calF^{\perp}/\calF$.

We denote by $E_{\omega}$ the elliptic curve with an order $3$ automorphism and note that
$$
E_{\omega}= \CC/\calE.
$$
We can write
$$
(E_{\omega})^9= E_{\omega} \otimes_{\calE} \calG=\CC^9/\calG.
$$

This description defines a natural action of $\O(\calG)$ on the $9$-dimensional abelian variety $(E_{\omega})^9$. The aim of this subsection is to prove the following
\begin{pro}\label{prop_structure_tor}
The following holds:
\begin{itemize}
\item[(1)] $T_{2A_5} \cong (E_{\omega}\otimes_{\mathcal E}(\mathcal E_1+2\mathcal E_4))/\O( \calE_1 +  2\calE_4)\ \  (\cong (E_{\omega})^9/\O( \calE_1 +  2\calE_4))$;
\item[(2)] $T_{3D_4}  \cong  (E_{\omega}\otimes_{\mathcal E}\widetilde{3\mathcal E_3})/\O(\widetilde{3\calE_3}) \ \ (\cong  (E_{\omega})^9/\O(\widetilde{3\calE_3}))$.
\end{itemize}
\end{pro}
\begin{proof}
We will give the proof for both cusps simultaneously. For this we pick an isomorphism as in (\ref{equ:decomp}) and  an isotropic vector $h$
 in $\calH$.
  As a matter of notation, by $F$ we will denote the cusp given by the isotropic line $\mathcal F=\mathcal Eh$.
 Now we choose $b_1:=h$ and extend this to a basis of $\Lambda$ such that the
hermitian form with respect to this basis has the Gram matrix
$$
Q=
\left(
\begin{array}{c|c|c}
0 & 0 & \theta \\ \hline
0 & B & 0\\ \hline
\bar \theta & 0 & 0\rule{0pt}{2.6ex}
\end{array}
\right)
$$
Here $b_2, \dots, b_{10}$ form a basis of $\calG$, and $B$ is the Gram matrix of $\calG$ with respect to this basis.  In order to understand the boundary we first have to determine
certain subgroups of $\O(\Lambda)$ related to the cusp $F$ (here we will only be dealing with the integral groups).
The first is the stabilizer subgroup $N(F)$ corresponding to $F$, i.e.~the subgroup of $\O(\Lambda)$ fixing the line spanned by $h$. A straightforward calculation, see~\cite[Sec.~4]{beh}, gives
\begin{equation}\label{pro:structureboundarycomp}
N(F)= \left\{ g \in \O(\Lambda): g= \left( \begin{array}{c|c|c}
u & v & w \\ \hline
0 & X & y \\ \hline
0 & 0 & s
\end{array}
\right)\right\}.
\end{equation}
Note that, in particular, this implies that $X\in \O(\calG)$. Its unipotent radical is given by
\begin{equation}
W(F)= \left\{g \in N(F): g= \left(
\begin{array}{c|c|c}
1 & v & w \\ \hline
0 & 1 & y \\ \hline
0 & 0 & 1
\end{array}
\right) \right\}
\end{equation}
and finally the center of the unipotent radical is
\begin{equation}
U(F)= \left\{g\in W(F): g= \left(
\begin{array}{c|c|c}
1 & 0 & w \\ \hline
0 & 1 & 0 \\ \hline
0 & 0 & 1
\end{array}
\right), w \in \ZZ \right\} \cong \ZZ.
\end{equation}
We have already introduced coordinates $(z_0:z_1: \dots : z_{10})$ on $\calB \subset \PP(\Lambda\otimes_{\calE}\CC)$ and we can assume that $z_{10}=1$.
The first step in the toroidal compactification is to take the partial quotient of $\calB$ by $U(F)$. This is given by
\begin{equation}
\begin{aligned} \calB &\to \CC^* \times \CC^9 \\  (z_0, \dots, z_{9}) &\mapsto (t_0=e^{2 \pi i z_0}, z_1, \dots,  z_9).
\end{aligned}
\end{equation}
Adding the toroidal boundary means adding the divisor $\{0\} \times \CC^9$, and we will use $z_1, \dots, z_9$ as coordinates on this boundary divisor. The quotient $N(F)/U(F)$ then acts on $\calB/U(F)$ and this quotient gives the
toroidal compactification of $\calB$ near the cusp $F$. Here we are only interested in the structure of the boundary divisor and hence in the action of $N(F)/U(F)$ on  $\{0\} \times \CC^9$. A straightforward calculation shows that
\begin{equation}\label{equ:action}
g=\left(
\begin{array}{c|c|c}
u & v & w \\ \hline
0 & X & y \\ \hline
0 & 0 & s
\end{array}
\right): \underline{z} \mapsto  \frac{1}{s}(X\underline{z} + y)
\end{equation}
where $\underline{z}=(z_1, \dots, z_9)$. We first look at  the normal subgroup $W(F)$, matrices whose elements act as follows
$$
g=\left(
\begin{array}{c|c|c}
1 & v & w \\ \hline
0 & 1 & y \\ \hline
0 & 0 & 1
\end{array}
\right):  \underline{z} \mapsto \underline{z} + y.
$$
Since $g\in \O(\Lambda)$ we have $y \in \calE^9$ and we claim that all vectors in  $\calE^9$ appear as entries in matrices $g\in W(F)$. Indeed a straightforward calculation, see~\cite[Sec.~4]{beh},  shows that  the condition that $g\in \O(\Lambda)$ is
$$
By+\bar{v}^t\theta=0, \quad \bar{y}^tBy+\bar{\theta}w + \theta \bar w=0.
$$
Given $y$ we define $v$ by $\bar{v}^t=-\frac{1}{\theta}By$. This is in $\calE^9$ since the coefficients of $By$ have values in $\theta \calE$.
Finally, we must check that we can  find a suitable element $w \in \calE$.
We know that $\bar{y}^tBy \in 3\ZZ$
and hence we can write $\bar{y}^tBy=3n$ for some integer $n$.
Hence we can take $w= -\frac12 + \frac{i}{2} \sqrt{3}n \in \calE$ if $n$ is odd and $w= \frac{i}{2} \sqrt{3}n \in \calE$ if $n$ is even, respectively . This shows that
$$
\CC^9/W(F)\cong (E_{\omega})^9.
$$

Next we consider the action of the subgroup
$$
\left\{ g \in \O(\Lambda): g= \left( \begin{array}{c|c|c}
1 & 0 & 0 \\ \hline
0 & X & 0 \\ \hline
0 & 0 & 1
\end{array}
\right)\right\}.
$$
which acts on $(E_{\omega})^9$ as claimed in the proposition.

It remains to consider elements of the form
$$
g=\left(
\begin{array}{c|c|c}
u & 0 & 0 \\ \hline
0 & 1 & 0 \\ \hline
0 & 0 & s
\end{array}
\right) \in N(F).
$$
The condition that such a matrix lies in $\O(\Lambda)$ is that $s\bar{u}=1$ with $s\in \calE$. Hence $s$ is a power of $\omega$ and these elements act on $(E_{\omega})^9$ by multiplication with powers of $\omega$. But by
(\ref{equ:action})  these elements are already in
$\O(\calG)$ and hence we do not get a further quotient, and the claim follows.
\end{proof}

\subsection{Isometry groups associated to the two cusps} We now discuss the isometry groups $\operatorname{O}(2\calE_4+\calE_1)$ and $\O(\widetilde{3\calE_3})$ associated to the two cusps $c_{2A_5}$ and $c_{3D_4}$ respectively (see Proposition~\ref{prop_structure_tor}). As we will discuss below, these groups are easily determined once the isometry groups of the basic lattices $\calE_3$ and $\calE_4$ are understood.
It turns out that the lattices $\calE_3$ and $\calE_4$ are special lattices, they are ``root lattices'' in the sense of Eisenstein lattices. Consequently, the associated isometry groups $\O(\calE_i)$ are essentially the complex reflections $W(\calE_i)$ generated by the roots ($W(\calE_i)$ is the analogue for Eisenstein lattices of the usual Weyl group).

To start our discussion of the isometry groups, let us recall that the role of reflections is taken by {\em triflections}. First, an {\it Eisenstein root} is an element $r\in \calG$ with $\|r\|^2=\langle r, r \rangle =3$. Note that $r$, when viewed as an element of
$\calG_\ZZ$, is then a root in the usual sense, i.e. $r^2=(r,r)=-2$.
The role of the $(-2)$-reflections is taken by {\em triflections}. To explain these, let $r$ be an Eisenstein root in the above sense and
define
$$
R_r: x \mapsto x - (1-\omega) \frac{\langle r,x \rangle }{\langle r,r \rangle }x.
$$
This defines an isometry of order $3$ (called a triflection). For an Eisenstein lattice $\calG$, we define
$$W(\calG)\subset \O(\calG)$$
to be the subgroup of isometries generated by triflections $R_r$ in all roots $r$ defined above.   Note that it follows, for instance,  that   $W(\calE_3)$ is the subgroup of $W(E_6)$ consisting of the elements that commute with $\rho$ (similar statements hold for all $\calE_n$ for $n=1,\dots,4$).

We will now discuss the isometry group of the relevant Eisenstein lattices.
Clearly
\begin{equation}\label{E:OOE1}
\O(\calE_1) = W(\calE_1) \times \ZZ/2\ZZ  \cong  \ZZ/3\ZZ \times \ZZ/2\ZZ
\end{equation}
where the factors are generated by multiplication by $\omega$ and $-1$, respectively. The (complex) Weyl groups $W(\calE_3)$ and $W(\calE_4)$ are well known complex reflection groups, typically denoted by $L_3$ and $L_4$ (see~\cite[Table 2]{Dolgachev-ref}), and described as follows.
\begin{pro}\label{propwe3}
The following holds:
\begin{itemize}
\item[(1)]
The Weyl group $W(\calE_3)$ is a group of order $2^3 \cdot 3^4=648$. It is isomorphic to $\operatorname{U}(3,\FF_4)$, respectively a semidirect product of an extra special group of order $27$ and exponent $3$ with
$\operatorname{SL}(2,\FF_3)$.
\item[(2)] $\O(\calE_3) \cong W(\calE_3) \rtimes \ZZ/2\ZZ$.
\end{itemize}
\end{pro}
\begin{proof}
For the first item see for example~\cite[Thm.~8.42]{LT}. For the second item, we recall that the underlying $\ZZ$-lattice is $E_6(-1)$. It is well known that  $\O(E_6(-1))=W(E_6(-1))\rtimes \ZZ/2\ZZ$, with the $\ZZ/2\ZZ$ factor corresponding to the outer automorphism $\tau$ given by the symmetry of the Dynkin diagram. To recover the groups $W(\calE_3)$ and $\O(\calE_3)$, we note that the Eisenstein structure on $E_6(-1)(=\calE_3)$ determines an order $3$ element $\rho\in W(E_6(-1))$. Then,  $W(\calE_3)$ can be obtained as the centralizer of $\rho$ in $W(E_6(-1))$ (e.g.~\cite[Ex. 9.5]{Dolgachev-ref}). Similarly, $\O(\calE_3)$ is the centralizer of $\rho$ in $\O(E_6(-1))$ (see~\eqref{eq_rel_iso}). Finally, a direct calculation shows that
$\tau$ commutes with the order $3$ automorphism $\rho$ which defines the Eisenstein structure on $E_6(-1)$. Item (2) follows.
\end{proof}

\begin{pro}\label{propwe4}
The following holds:
\begin{itemize}
\item[(1)] The Weyl group $W(\calE_4)$ has order $2^7 \cdot 3^5 \cdot 5=  155,520$. It is isomorphic to $\ZZ/3\ZZ \times \operatorname{Sp}(4,\FF_3)$.
\item[(2)] $W(\calE_4)=\O(\calE_4)$.
\end{itemize}
\end{pro}
\begin{proof} This follows from~\cite[Thm. 5.2]{ALeech} (see also~\cite[Thm.~8.43]{LT}).
\end{proof}

With these preliminaries, we can now conclude the computation of the groups of isometries occurring for the two cusps of the ball quotient model.
\begin{pro} \label{lem:autoEisenstein}
The following holds:
\begin{itemize}
\item[(1)] $\O(\calE_1 + 2 \calE_4) \cong W(\calE_1) \times (W(\calE_4)^{\times 2} \rtimes S_2)$
\item[(2)] $\O(\widetilde{3\calE_3}) \cong (W(\mathcal E_3)^{\times 3} \rtimes S_3) \rtimes \ZZ/2\ZZ$.
\end{itemize}
Here $S_n$ denotes the symmetric group in $n$ elements.
\end{pro}
\begin{proof}

Let $\calG$ be a definite Eisenstein lattice. We start with two basic remarks. Firstly, the set of Eisenstein roots $\calR(\calG)$ is finite, and in fact (under our scaling assumptions) coincides (set-theoretically) with the set of $-2$ roots for the
$\ZZ$-lattice $\calG_\ZZ$.
Indeed, this follows from~\eqref{arr:hermitianform}.
Secondly, any isometry $\phi\in \O(\calG)$ preserves the set of roots ($\phi(\calR(\calG))=\calR(\calG)$). Furthermore, if the set of roots $\calR(\calG)$ generates  $\calG$ (over $\QQ(\omega)$), then $\phi$ is determined uniquely by the action of $\phi$ on the finite set $\calR(\calG)$. In our situation, $\calG=2\calE_4+\calE_1$ or $\calG=\widetilde{3\calE_3}$, and we know $R(\calG_\ZZ)=2E_8(-1)+ A_2(-1)$ and $R(\calG_\ZZ)=3E_6(-1)$ respectively (see Lemma~\ref{lem_inv_iso}), where we denote by $R(\calG_\ZZ)$ the sublattice of $\calG_\ZZ$ spanned by the roots $\calR(\calG_\ZZ)$. Thus, given $\phi\in \O(2\calE_4+\calE_1)$ (or $\phi\in \O(\widetilde{3\calE_3})$ respectively), after a permutation (giving the $S_2$ and $S_3$ factors above), we can assume that $\phi$ preserves the irreducible root summands $2\calE_4+\calE_1$ (and respectively each of the three $\calE_3$ in $\widetilde{3\calE_3}$).  Then the isometry $\phi$ is determined by the action of $\phi$ on the summands $\calE_1$, $\calE_4$ and $\calE_3$ respectively (which were described in Propositions~\ref{propwe4} and~\ref{propwe3}); see also~\cite[Lem.~1]{HKN} for a related argument.
This shows immediately that the isometries of $\calE_1 + 2 \calE_4$ are as claimed. We have also seen that any isometry of $\widetilde{3\calE_3}$ is the extension of an isometry of $3\calE_3$. Now, an isometry of  $3\calE_3$ lifts to
 $\widetilde{3\calE_3}$ if and only if the induced isometry on the discriminant preserves the defining subgroup $H\subset A_{3E_6(-1)}\cong (\ZZ/3)^3$. This immediately shows that $W(\mathcal E_3)^{\times 3} \rtimes S_3 \subset  \O(\widetilde{3\calE_3})$.
To complete the proof we recall from Proposition~\ref{propwe3} that $\O(\calE_3) \cong W(\calE_3) \rtimes \ZZ/2\ZZ$ where the $\ZZ/2\ZZ$-factor is generated by the involution $\tau$ given by the symmetry of the Dynkin diagram. Since $\tau$ acts by $-1$ on
$D(E_6(-1))\cong \ZZ/3\ZZ$ it follows that only the identity and the diagonal element $(\tau,\tau,\tau)$ extend to  $\widetilde{3\calE_3}$. This gives the extra factor of $\ZZ/2\ZZ$ and completes the proof.
\end{proof}

\section{The cohomology of the toroidal boundary divisors} We will compute the cohomology  of the toroidal compactification $\oBG$ by applying the decomposition theorem to the map $\oBG \to (\calB/\Gamma)^*$ and using our knowledge of the cohomology of $(\calB/\Gamma)^*$ (see \S~\ref{sec:IHball}).
For this we require the knowledge of the cohomology of the two toroidal boundary divisors $T_{2A_5}$ and $T_{3D_4}$.
Since the lattices involved in the definition of the two exceptional divisors (see Proposition~\ref{prop_structure_tor}) are essentially direct sums of $\calE_3$ and $\calE_4$ lattices, the main ingredient needed to compute the cohomology of $T_{2A_5}$ and $T_{3D_4}$ respectively is the cohomology of the spaces $(E_\omega\otimes_{\calE} \calE_k)/W(\calE_k)$ for $k=3,4$. These are quotients of abelian varieties by finite groups, and it turns out that these spaces are equivariantly isogenous to weighted projective spaces of dimension $k$ (see Proposition~\ref{thm_L_Eis} and~\eqref{E:Eisen-Looij},  below, for a precise statement), and thus they have simple cohomology, making the remaining computations routine.

\subsection{An Eisenstein analogue of Chevalley's theorem} \label{S:Eies-Chev}
The fact that $(E_\omega\otimes_{\calE} \calE_k^*)/W(\calE_k)$ for $k=3,4$  have the same cohomology as weighted projective spaces of dimension $k$ is an analogue over $\calE$ of a Chevalley type Theorem due to Looijenga~\cite{Lroot}.  Specifically, we recall that if $R$ is an irreducible ADE root lattice (we make this assumption for simplicity) and $E$ is an elliptic curve, then $(E\otimes_\ZZ R^\vee)/W(R)\cong W\PP^r$, where $W\PP^r$ denotes a weighted projective space of dimension $r$ equal to the rank of $R$ (see~\cite[Thm. 2.7]{FMW}).  In the Eisenstein case, we get that the quotients are equivariantly isogenous to weighted projective spaces; in particular we obtain:

\begin{pro}\label{thm_L_Eis}
Let $E_\omega$ be the elliptic curve with $j$-invariant $0$. Then:
\begin{itemize}
\item[(1)] $H^\bullet((E_\omega\otimes_\calE \calE_3)/W(\calE_3))\cong H^\bullet(W\PP(1,2,2,3))$
\item[(2)] $H^\bullet((E_\omega\otimes_\calE \calE_4)/W(\calE_4))\cong H^\bullet(W\PP(2,3,4,5,6))$.
\end{itemize}
\end{pro}
\begin{proof}
Chevalley type theorems for complex reflection groups acting on projective varieties were obtained by Bernstein and Schwarzman~\cite{BS}. In particular, the fact that the cohomology of $(E_\omega\otimes \calE_k)/W(\calE_k)$ agrees with that of weighted projective space (in much more generality) follow from~\cite[\S2.3]{BS}.

For the reader's convenience, we sketch a geometric proof of the two cases that are needed in our paper, following the outline of~\cite{FMW}. For  simplicity, we will discuss only the $\calE_3$ case, the other case being obtained by minor changes. Let us discuss first the situation over $\ZZ$ (i.e., the classical setup of Looijenga), namely the statement that $(E\otimes_\ZZ E_6^\vee)/W(E_6)\cong W\PP(1,1,1,2,2,2,3)$ (N.B.~for the moment, $E$ is any (fixed) elliptic curve).

We consider the moduli of anticanonical pairs\footnote{In general, {\it an anticanonical pair} $(S,D)$ is a rational surface $S$ together with a reduced anticanonical cycle $D\in |-K_S|$. Clearly, $D$ is either a smooth elliptic curve (as in our case) or a cycle of rational curves. The latter case is sometimes known also as {\it Looijenga pair}. The moduli of anticanonical pairs is well understood: the case when $D$ is smooth is essentially reviewed here, while the harder case when $D$ is singular is treated in~\cite{GHK},~\cite{F13}.}, \emph{with $E$ fixed},
$$\calP_E=\{(S,E)\mid S \textrm{ is a degree $3$ del Pezzo surface, and }E\in|-K_S|\}\,.$$
This moduli space has two different descriptions. On the one hand, as an instance of Pinkham's general theory of deformations of singularities with $\CC^*$-action (see~\cite{pinkham}), we obtain
 a GIT description as $\calP_E=B_{-}\gquot \CC^*$, where $B_{-}\cong \CC^7$ is the negative weight deformation space for the versal deformation
 of the singularity of type $\widetilde E_6$ (the affine cone over the elliptic curve $E\subset \PP^2$). (In general, the versal deformation space of a singularity is a germ. However, for singularities with $\CC^*$-action, there is an induced $\CC^*$-action, which allows one to globalize the negative weight subspace. Thus, in the case of quasi-homogeneous hypersurface singularities, one can take $B_{-}$ to be an affine space.)
 The versal deformations of $\widetilde E_6$ are easily described explicitly, and as a consequence one gets
 \begin{equation}\label{wp_iso}
 \calP_E\cong W\PP(1,1,1,2,2,2,3).
 \end{equation}
 (We refer to~\cite{rthesis} for further related discussion.) On the other hand, one gets a period map
  \begin{eqnarray}\label{pe_quot}
  \Phi_E:\calP_E&\to&\Hom_\ZZ(E_6(-1),E)/W(E_6)\\
  (S,E)&\to& H^2(S,E)\notag
  \end{eqnarray}
  which can be explained as follows: Firstly, $W(E_6)$ is the monodromy group acting on the primitive cohomology $E_6(-1)\cong H^2(S,\ZZ)_{0}$ of the del Pezzo surface $S$. Secondly, from the exact sequence of a pair one gets
$$
0\to H^1(E)\to H^2(S,E)\to H^2(S)_0\to 0,
$$
where we are using that $H^2(S)_0=\ker (H^2(S)\to H^2(E))$.  This shows that the mixed Hodge structure on $H^2(S,E)$ is an extension of  a trivial weight $2$ Hodge structure (type $(1,1)$) by the elliptic curve $E$. Carlson~\cite{carlson} showed that these type of extensions are classified by $\Hom_\ZZ(H^2(S)_0,E)$.  Since the elliptic curve $E$ is fixed, the monodromy $W(E_6)$ acts naturally on $H^2(S)_0\cong E_6(-1)$ and on $\Hom_\ZZ(H^2(S)_0,E)$. Thus, the period space (i.e., the period domain modulo monodromy) for $\Phi_E$ is $\Hom_\ZZ(E_6(-1),E)/W(E_6)$,
showing that the period map above is well defined. Let us then note that in fact the period map $\Phi_E$ has an easy geometric description. Namely, by identifying $H^2(S)_0$ with $\Pic(S)_0$, i.e., degree $0$ line bundles on $S$ with respect to the polarization $-K_S$, and $E$ with $\Pic^0(E)$, one sees that the ``period point''
$$\Psi:=\Phi_E(S,E)\in \Hom_\ZZ(H^2(S)_0,E)$$ is just the natural restriction morphism
\begin{eqnarray}
\Psi:\Pic(S)_0&\to& \Pic^0(E)\notag\\
\mathcal L&\to&\mathcal L_{\mid E}.\label{eq_psi}
\end{eqnarray}
Finally, an easy Torelli type theorem (essentially, the surface $S$ is the blow-up of $6$ points in $\PP^2$ which lie on a smooth cubic curve $C\cong E$) establishes that  $\Phi_E$ is an isomorphism. (We refer to~\cite{carlson} and~\cite{Fannals} for details of the period map construction and the Torelli theorem. In particular, we note that our case is one of the main examples in~\cite{carlson}.)  Comparing the GIT~\eqref{wp_iso} and Hodge theoretic~\eqref{pe_quot} descriptions of $\calP_E$, one obtains the claimed result $(E\otimes_\ZZ E_6^\vee)/W(E_6)\cong W\PP(1,1,1,2,2,2,3)$.

  Returning to our situation, i.e., the Eisenstein analogue of the above argument, we proceed as follows. We specialize to the case $E=E_\omega$, and we define $\calP_{E_\omega}^\omega$  to be the subspace of $\calP_{E_\omega}$ corresponding to pairs $(S,E_\omega)$ for which the $\mu_3$-action on $E_{\omega}$ extends (linearly) to $S$. It is easy to see that we can choose a normal form for $S$ as follows
  $$S=V((x^3+y^3+z^3)+a_1txy+a_2 t^2x+a_3t^2y+a_4t^3)\subset \PP^3\,,$$
 with the elliptic curve $E_\omega$ being the hyperplane at infinity $(t=0)$.
 Since $E_\omega$ is fixed, the only transformation allowed is the rescaling of $t$. This shows that $\calP_{E_\omega}^\omega\cong W\PP(1,2,2,3)$ (and this is a natural subspace of $\calP_{E_\omega}\cong W\PP(1,1,1,2,2,2,3)$).
 Let us now discuss the restriction of the period map $\Phi_{E_\omega}$ to the subspace $\calP_{E_\omega}^\omega\subset \calP_{E_\omega}$. By definition $\calP_{E_\omega}^\omega$ is the locus of pairs $(S,E_\omega)$ that admit an order $3$ automorphism $f$. Since $f$ preserves $K_S$, we see that $f^*$ acts as an order $3$ isometry, call it  $\rho$, on $H^2(S)_0\cong E_6(-1)$. On the other hand, the restriction of $f$ to $E_\omega$  acts as multiplication by $\omega$. Since $f$ acts compatibly on the pair $(S,E)$, we get $(f^*\mathcal L)_{\mid E}=f^*(\mathcal L_{\mid E})$, which in turn is equivalent to saying (compare~\eqref{eq_psi})
 $$\Psi(\rho(\mathcal L))=\omega\cdot \Psi(\mathcal L)\,.$$
 We conclude that the period domain for the restricted period map $\Phi_{E_\omega \mid \calP_{E_\omega}^\omega}$ is the $\omega$-eigenspace in $(E_\omega\otimes_\ZZ E_6^\vee)$ (w.r.t.~the action induced by $\rho$ on the second factor).
 In short we have
 \begin{equation}\label{E:Eisen-Looij}
 W\PP(1,2,2,3)\cong \calP_{E_\omega}^\omega\cong (E_\omega\otimes_\ZZ E_6^\vee)_\omega/W(\calE_3),
 \end{equation}
 where the subscript indicates the $\omega$-eigenspace.

We  now recall the identification $E_6^\vee=((\mathcal E_3)_{\mathbb Z})^\vee = (\mathcal E_3')_{\mathbb Z}$ from Remark~\ref{R:DualEisen}.  By considering the eigenspaces for the $\omega$ action on the right factor of $\mathcal E \otimes_{\mathbb Z}(\mathcal E_3')_{\mathbb Z}$, we obtain  inclusions $\mathcal E_3 \hookrightarrow \mathcal E_3'\hookrightarrow (\mathcal E\otimes_\ZZ (\mathcal E_3')_{\mathbb Z})_\omega= (\mathcal E\otimes_\ZZ E_6^\vee)_\omega$, with torsion co-kernels.  Tensoring with $E_\omega \otimes_{\mathcal E}-$, we obtain isogenies
$E_\omega\otimes_{\mathcal E}\mathcal E_3 \twoheadrightarrow E_\omega \otimes_{\mathcal E}\mathcal E_3'\twoheadrightarrow (E_\omega\otimes_{\mathbb Z} E_6^\vee)_\omega$, which are $W(\mathcal E_3)$-equivariant.  Since isogenies give isomorphisms on the cohomology of abelian varieties with rational coefficients, we have the identifications $H^\bullet (E_\omega\otimes_{\mathcal E}\mathcal E_3) /W(\mathcal E_3))=H^\bullet (E_\omega\otimes_{\mathcal E}\mathcal E_3))^{W(\mathcal E_3)}=H^\bullet ((E_\omega\otimes _{\mathbb Z}E_6^\vee )_\omega)^{W(\mathcal E_3)}=H^\bullet ((E_\omega\otimes _{\mathbb Z}E_6^\vee )_\omega/{W(\mathcal E_3)})=H^\bullet (W\PP(1,2,2,3))$.
 \end{proof}

\subsection{The cohomology of the divisors $T_{2A_5}$ and $T_{3D_4}$}
We are now ready to compute the topology of the toroidal boundary components.
The result is the following

\begin{pro}\label{pro:cohomoloytorbound}
The cohomology of the toroidal boundary divisors $T_{2A_5}$ and $T_{3D_4}$ is given by the following table:
\begin{equation}
\renewcommand*{\arraystretch}{1.3}
\begin{array}{r|ccccccccccc}
j&0&2&4&6&8&10&12&14&16&18\\\hline
\dim H^j(T_{2A_5})&1&2&3&4&5&5&4&3&2&1\\
\dim H^j(T_{3D_4})&1&1&2&3&3&3&3&2&1&1\\
\end{array}
\end{equation}
All odd cohomology vanishes.
\end{pro}
\begin{proof}
As discussed in Proposition~\ref{thm_L_Eis}, the quotients $(E_{\omega}\otimes_\calE \calE_i)/W(\calE_i)\cong(E_\omega)^i/W(\calE_i)$ (for $i=3,4$) have the cohomology of weighted projective spaces. Hence, as graded vector spaces, we have
\begin{equation}
H^\bullet((E_{\omega})^i/W(\calE_i),\QQ) \cong \QQ[x]/(x^{i+1}).
\end{equation}
We shall first treat the case $T_{2A_5}$. It follows from  Proposition~\ref{pro:structureboundarycomp} and Lemma~\ref{lem:autoEisenstein} that
\begin{equation}
T_{2A_5} \cong E_{\omega}/W(\calE_1) \times ((E_{\omega})^4/W(\calE_4))^{\times 2}/S_2.
\end{equation}
We shall first compute the cohomology of the second factor. For this we have to consider the $S_2$ invariant parts of a tensor product
$\QQ[x]/(x^{i+1}) \otimes \QQ[y]/(y^{i+1})$. The invariants in each degree are given by $1$ in degree $0$, $x+y$ in degree $1$, $x^2 + y^2,xy$ in degree $2$, $x^3+y^3, x^2y + xy^2$ in degree $3$, and
$x^4+y^4,x^3y + xy^3, x^2y^2$ in degree $4$. Hence, using Poincar\'e duality we see that all the odd cohomology vanishes, and that the entire cohomology is equal to
\begin{equation}
P_t\left(\left((E_{\omega})^4/W(\calE_4)\right)^{\times 2}/S_2\right)=1+t^2+2t^4+2t^6+3t^8+2t^{10}+2t^{12}+t^{14}+t^{16}.
\end{equation}
The cohomology of the first factor is that of $\PP^1$, equal to $1+t^2$, and an application of the K\"unneth formula therefore gives
\begin{equation}
\begin{aligned}
P_t(T_{2A_5})&=(1+t^2)\cdot(1+t^2+2t^4+2t^6+3t^8+2t^{10}+2t^{12}+t^{14}+t^{16})\\
&=1+2t^2+3t^4+4t^6+5t^8+5t^{10}+4t^{12}+3t^{14}+2t^{16}+t^{18}.
\end{aligned}
\end{equation}
We shall now treat the second boundary component $T_{3D_4}$. We first note that the inclusion $3\calE_3 \subset \widetilde{3\calE_3}$ gives us an \'etale $3:1$ map
\begin{equation}\label{equ:covering}
\CC^9/3\calE_3 \cong (E_{\omega})^9 \to \CC^9/ \widetilde{3\calE_3}.
\end{equation}
To compute the cohomology of $T_{3D_4}$ is equivalent to computing the invariant cohomology of  $\CC^9/ \widetilde{3\calE_3}$  under the group
$\O(\widetilde{3\calE_3}) \cong (W(\mathcal E_3)^{\times 3} \rtimes S_3) \rtimes \ZZ/2\ZZ$. Since the covering group of the \'etale $3:1$ map (\ref{equ:covering})
acts by translation on the product of elliptic curves, and hence trivially on cohomology,
this is equivalent
to computing the invariant cohomology of $\CC^9/3\calE_3 \cong (E_{\omega})^9$. We will first restrict to the subgroup $W(\mathcal E_3)^{\times 3} \rtimes S_3$.
Again using the fact that each factor $(E_{\omega})^3/W(\calE_3)$ has the cohomology of a weighted projective space, and counting invariants under the symmetry  group $S_3$ as above, we obtain for the invariant cohomology
\begin{equation}
P_t\left((E_\omega)^9\right)^{W(\mathcal E_3)^{\times 3} \rtimes S_3}=1+t^2+2t^4+3t^6+3t^8+3t^{10}+3t^{12}+2t^{14}+t^{16}+t^{18}.
\end{equation}
It remains to consider the action of the outer automorphism $\tau$
(see the proof of Proposition~\ref{lem:autoEisenstein}),
which acts diagonally on the triple product. Note, however, that  $(E_{\omega})^3/W(\calE_3)$ has $1$-dimensional cohomology
in even degree and no odd cohomology. Hence $\tau$ acts trivially on the cohomology of $(E_{\omega})^3/W(\calE_3)$, and this finishes the proof.
\end{proof}

\begin{rem}
Note that the Betti numbers of $T_{3D_4}$ and  $T_{2A_5}$  agree with those of  $D_{3D_4}$ and $D_{2A_5}$, given by formulas~\eqref{E:g3D4-contr}  and~\eqref{E:g2A5-contr}, respectively.
\end{rem}

\section{The cohomology of the toroidal compactification}
At this point, we can conclude the computation of the cohomology of the toroidal compactification $\oBG$.  This completes the proof of our main Theorem~\ref{teo:betti}:
\begin{teo}\label{teo:coh_obg}
The cohomology of the toroidal compactification  $\oBG$ of the ball quotient is given by
\begin{equation}
\renewcommand*{\arraystretch}{1.5}
\begin{array}{r|ccccccccccc}
j&0&2&4&6&8&10&12&14&16&18 &20\\\hline
\dim H^j(\oBG)&1&4&6&10&13&15&13&10&6&4&1
\end{array}
\end{equation}
All odd cohomology vanishes.
\end{teo}
\begin{proof}
We recall that the toroidal compactification $\oBG$ is smooth up to finite quotient singularities, and that the morphism $\oBG\to\BG$ is the blowup of two points. Hence we can
apply the decomposition theorem in the form of~\cite[Lem.~9.1]{GH-IHAg-17}, i.e.~we are in the special case of \S\ref{S:Decomp-Thm}, to the morphism $\oBG \to \BG$. We thus compute

\begin{align*}
&P_t(\oBG) \equiv\\
&\equiv 1+2t^2+3t^4 +5t^6    +6t^8  +7t^{10}&\text{($IP_t(\BG)$, from the previous section)}\\
&\ \ \ \                + t^2  +t^4   +2t^6    +3t^8  +3t^{10}&\text{($T_{3D_4}$ contribution, from Proposition~\ref{pro:cohomoloytorbound})}\\
&\  \ \ \               + t^2  +2t^4 +3t^6    +4t^8  +5t^{10}&\text{($T_{2A_5}$ contribution, from Proposition~\ref{pro:cohomoloytorbound})}\\
&\equiv 1+4t^2+6t^4  +10t^6  +13t^8 +15t^{10}\mod t^{11}\!\!\!\!\!\!\!\!\!\!\!\!\!\!\!\!\!\!\!\!\!\!\!
\end{align*}
by applying equation~\eqref{eq:IHblowup} to determine the contribution to the cohomology of $\oBG$ from each of the two exceptional divisors.
\end{proof}

%% file: sec8_equivcoh.tex
\chapter{Equivariant cohomology}\label{sec:equivcoh}
In this appendix we review a few basic facts from the theory of equivariant cohomology.  The first subsection, \S\ref{S:A-AB83}, is a review of~\cite[\S 13]{atiyahbott83}.
In \S\ref{S:A-G/K}, we review some results concerning  the equivariant cohomology of    Lie groups.
In \S\ref{S:A-KEC} we recall~\cite[Prop.~5.8]{kirwan84} concerning equivariant cohomology for quotients of symplectic manifolds by compact Lie groups.
These results are all standard by now, but unfortunately, we are not aware of a reference where the results are all stated. As is the case throughout the paper, for a topological space $X$,  we use the convention $H^\bullet(X)=H^\bullet (X,\mathbb Q)$.  

\section{Review of  Atiyah--Bott}\label{S:A-AB83}
For any topological group $G$, a classifying space $BG$ is defined as the base of a left principal $G$-bundle $EG\to BG$ whose total space $EG$ is contractible. A classifying space is unique up to homotopy, so that in particular $H^\bullet (BG)$ depends only on $G$.   For every topological group $G$ a classifying space exists~\cite{milnor56}.
\begin{exa}\label{E:A-BGL}
The principal $\GL(n,\CC)$-bundle induced by the universal vector bundle $E_n $ over the Grassmannian $\operatorname{Gr}(n,\CC^\infty)$ makes  $\operatorname{Gr}(n,\CC^\infty)$ into a classifying space for $\GL(n,\CC)$.  The cohomology can be described as $H^\bullet(B\GL(n,\CC),\ZZ)\cong \ZZ[c_1,\dots,c_n]$ with $c_i$ taken to have degree $2i$.  From this one can deduce that $P_t(B\GL(n,\CC)) = (1-t^2)^{-1}(1-t^4)^{-1}\dots (1-t^{2n})^{-1}$.
\end{exa}
More generally if  $G$ acts on a topological space $X$ on the right, and a choice of classifying space $BG$ has been made, then
we define $X_G:=X\times_G EG:=(X\times EG)/G$, which is
a locally trivial fibration over $BG$ with fiber $X$ and structure group $G$.  The $G$-equivariant cohomology of $X$ is defined to be the ordinary cohomology of $X_G$:
\begin{equation}\label{E:AB-EC-1}
 H^\bullet_G(X):=H^\bullet(X_G).
\end{equation}
In particular we have $H^\bullet(BG)=H^\bullet_G(\ast)$, where $\ast$ is a topological space with one point.   Moreover, it is well known (e.g.,~\cite[Thm.~6.10.5]{weibel94}) that $H^\bullet(BG)\cong H^\bullet_{\mathsf {gp}}(G,\QQ)$; i.e., that the cohomology of the classifying space is given by the group cohomology.

If  the quotient map   $X\to X/G$ is a right principal $G$-bundle, for instance if $G$ is a compact Lie group acting freely on a manifold $X$, then
\begin{equation}\label{E:AB-EC-2}
H^\bullet_G(X)\cong H^\bullet (X/G).
\end{equation}
Indeed,  since $X\to X/G$ is a right principal $G$-bundle,  applying $X\times_G -$ to  the canonical morphism $EG\to \ast$,  we obtain that the morphism $X_G=X\times_G EG \to  X\times_G \ast  \cong  X/G$  is a locally trivial fibration with contractible fiber $EG$.
Thus there is a homotopy equivalence $X_G\simeq X/G$.

Note that if $G$ is a compact Lie group acting properly on a manifold $X$, with finite stabilizers, then~\eqref{E:AB-EC-2} still holds.  In this case, the fiber of $X_G\to X/G$ over a point $[x]\in X/G$ is isomorphic to $E_G/G_x$, which satisfies $H^i(E_G/G_x)=0$, $i\ge 1$; thus one may conclude via the Leray spectral sequence.

Another useful observation is the following.  If $G$ is a  subgroup of a topological group $G'$, and $G'\to G'/G$ is a principal $G$-bundle, for instance if $G$ is a closed subgroup of a Lie group $G'$,
  then
  \begin{equation}\label{E:AB-EC-4}
  H^\bullet_G(X)\cong H^\bullet_{G'}(X\times_G G').
\end{equation}
The proof is as follows:  Since $G'\to G'/G$ is a principal $G$-bundle, we have that $EG'\to EG'/G$ is a principal $G$-bundle as well, so that we may take $EG=EG'$.  Thus we have
$
(X\times_G G')_{G'}:= (X\times_G G' ) \times_{G'} EG' \cong X\times _G EG '=X_G$.

As an immediate application, if a Lie group $G$ acts transitively on   $X$ with  $X=x\cdot G$ for some $x\in X$, and $G_x$ is the stabilizer of $x$, then $X\cong G_x\backslash G \cong x \times_{G_x}G$, so that~\eqref{E:AB-EC-4} gives
  \begin{equation}\label{E:AB-EC-5}
H_{G_x}(x)\cong H_G(X).
\end{equation}

\section{Compact  and complex Lie groups} \label{S:A-G/K}
Here we focus on the situation where $K$ is a subgroup of a topological group $G$ such that the quotient map $G\to G/K$ is a principal $K$-bundle; for instance $K$ is a closed subgroup of a Lie group $G$.    In this situation $EG\to EG/K$ is also a principal $K$-bundle, so we may take $EK=EG$.  Since  $X\times EG\to X\times BG$ is a principal $G$-bundle, we have that
$$
X_K=(X\times EG)/K \longrightarrow (X\times EG)/G=X_G
$$
is a locally trivial fibration with fiber $G/K$.

Under various assumptions on $G$ and $K$ we can  deduce some further consequences.
For instance, if $G$ is a  connected Lie group and  $K$ is a maximal compact subgroup, then
\begin{equation}\label{E:AS-EC-1}
H^\bullet_G(X)\cong H^\bullet_K(X).
\end{equation}
Indeed in this case $X_K\to X_G$ is a homotopy equivalence,  since $G/K$ is homeomorphic   to $\RR^n $ for some $n$.

\begin{exa}\label{E:BUn}
Since $\operatorname{U}(n)$ is a maximal compact subgroup of $\GL(n,\CC)$, we have from~\eqref{E:AS-EC-1}  and Example~\ref{E:A-BGL}  that $H^\bullet(B\operatorname{U}(n),\ZZ)\cong \ZZ[c_1,\dots,c_n]$ with $c_i$ taken to have degree $2i$.  From this one can deduce that $P_t(B\operatorname{U}(n)) = (1-t^2)^{-1}(1-t^4)^{-1}\dots (1-t^{2n})^{-1}$.
\end{exa}

\begin{exa}\label{E:BSUn} Identifying $S^1=\operatorname{U}(1)$ with the  group of $n\times n$ diagonal matrices with all entries equal, the surjective multiplication  homomorphism $\operatorname{SU}(n)\times S^1\to \operatorname{U}(n )$ has kernel isomorphic to the group $\mu_n$ of $n$-th roots of unity.   The Lyndon/Hochschild--Serre spectral sequence (e.g.,~\cite[Thm.~6.8.2]{weibel94}) for the normal subgroup $\mu_n$ of $\operatorname{SU}(n)\times S^1$ then degenerates, since the higher group cohomology for $\mu_n$, being torsion, vanishes with $\QQ$-coefficients (e.g.,~\cite[Cor.~6.3.5]{weibel94}), giving an isomorphism $H^\bullet_{\mathsf {gp}}(\operatorname{U}(n))\cong H_{\mathsf {gp}}^\bullet(\operatorname{SU}(n)\times S^1)$.  Finally, as $H_{\mathsf {gp}}^\bullet(\operatorname{SU}(n)\times S^1)\cong H_{\mathsf {gp}}^\bullet(\operatorname{SU}(n) )\otimes H_{\mathsf {gp}}^\bullet(  S^1)$ (e.g.,~\cite[Exe.~6.1.10]{weibel94}),
we obtain
\begin{equation}\label{E:SU-EC-1}
H^\bullet (B{\operatorname{U}(n)}) =H^\bullet (B{\operatorname{SU}(n)})\otimes H^\bullet(BS^1).
\end{equation}
Since $H^\bullet(BS^1)\cong H^\bullet(B\operatorname{U}(1))\cong \QQ[c_1]$,
we have $P_t(B\operatorname{SU}(n))=(1-t^4)^{-1}\dots (1-t^{2n})^{-1}$.   As $\operatorname{SU}(n)$ is a maximal compact subgroup of $\SL(n)$, one has $P_t(B\SL(n))=P_t(B\operatorname{SU}(n))$ and $H^\bullet (B{\GL(n,\CC)}) =H^\bullet (B{\SL(n,\CC)})\otimes H^\bullet(B\CC^*)$.   Using the short exact sequence $1\to \mu_n\to \SL(n,\mathbb C)\to \PGL(n,\mathbb C)\to 1$, similar arguments show that $H^\bullet_{\mathsf {gp}}(\mathbb \PGL(n,\mathbb C))=H^\bullet_{\mathsf {gp}}(\SL(n,\mathbb C))$, so that $H^\bullet(B\PGL(n,\mathbb C))=H^\bullet(B\SL(n,\mathbb C))$.
\end{exa}

If $K$ is a closed normal subgroup of a Lie group  $G$,  and  $G/K$ is a finite group, then
\begin{equation}\label{E:AS-EC-2}
H^\bullet_G(X)=(H^\bullet_K(X))^{(G/K)}.
\end{equation}
Indeed in this case $X_K\to X_G$ is a principal bundle for the finite group $G/K$.
Note that the action of $G/K$ on $H_K^\bullet(X)$ is induced by an action of $G/K$ on $X_K$.   As a particular example, if  $G$ is finite, one obtains as a special case
\begin{equation}\label{E:AS-EC-FinG}
H_G^\bullet(X)=H^\bullet(X)^G=H^\bullet(X/G).
\end{equation}

\begin{exa}\label{Exa:SemDirExa}  Suppose we have $G=K\rtimes F$, where $K$ is a compact Lie group and $F$ is a finite group.    Let $\phi:F\to \operatorname{Aut}(K)$ be the homomorphism associated to the semidirect product.  This induces a homomorphism $\Phi:F\to \operatorname{Aut}(BK)$, giving the action of $F$ on $H^\bullet(BK)$ such that $H^\bullet(BG)=H^\bullet(BK)^F$.
When $K=T=(S^1)^r$ is a compact torus, this can be made more explicit.  We have $\operatorname{Aut}(T)=\GL(r,\ZZ)$, and $BT=\prod^r BS^1=\prod^r\PP^\infty_{\CC}$.
The canonical action of $\operatorname{Aut}(T)$ on $BT$ is given, for each $\phi\in \operatorname{Aut}(T)$, by sending a right principal $H$-bundle $P\to B$ to $P\times_{H,\phi}H$. More concretely, $H^\bullet(BT)=\operatorname{Sym}^\bullet H^2(BT)=\operatorname{Sym}^\bullet  \QQ\langle c_1^{(1)},\dots,c^{(r)}_1\rangle = \QQ[c_1^{(1)},\dots,c^{(r)}_1]$, $\deg c_1^{(i)}=2$, $i=1,\dots,r$.   Viewing $\phi\in \operatorname{Aut}(T)=\GL(r,\ZZ)\subseteq \GL(r,\QQ)$ as a matrix,  we obtain an action of $\phi$ on  $\mathbb Q^r=  \QQ\langle c_1^{(1)},\dots,c^{(r)}_1\rangle=H^2(BT)$ by matrix multiplication. 
This induces an action of 
 $\phi$ on $H^\bullet(BT)=\operatorname{Sym}^\bullet H^2(BT)$, 
 which one can check agrees with the canonical action under these identifications.
  Similarly, if $K=T\times \Gamma$ for a finite abelian group $\Gamma$, and $T$ a compact torus as above, then $H^\bullet(BK)^F=H^\bullet(BT)^F$, where the action of $F$ on $BT$ is induced by the action of $F$ on $T$, viewing $T$ as the connected component of the identity.
\end{exa}

If $K$ is a compact connected Lie group and $T$ is a maximal torus in $K$,
\begin{equation}\label{E:AS-EC-3}
H^\bullet_T(X)=H^\bullet_K(X)\otimes H^\bullet(K/T).
\end{equation}
Indeed, in this case the fibration   $X_T\to X_K$ has fiber given by  the  flag variety
$K/T$.  A direct computation
(see e.g.,~\cite[p.35]{kirwan84})
 shows that the associated Leray spectral sequence degenerates, giving~\eqref{E:AS-EC-3}.

We focus again on the situation where $K$ is a subgroup of a topological group $G$ such that the quotient map $G\to G/K$ is a principal $K$-bundle; for instance $K$ is a closed subgroup of a Lie group $G$.
If $K$ is central and contained in the kernel of the map $G\to \operatorname{Aut}(X)$, then
\begin{equation}\label{E:K-AS-EC-cent}
H^\bullet_G(X)= H^\bullet(BK) \otimes H^\bullet_{G/K}(X).
\end{equation}
Indeed, we start with the observation that, with $G$ acting on $E(G/K)$ via the quotient map to $G/K$, we have that $EG\times E(G/K)$ is contractible with a free $G$-action. Thus we have
$$X\times_G EG\simeq X\times_G (EG\times E(G/K))=((X\times EG\times E(G/K))/K)/(G/K)$$
$$
=((X\times_K EG) \times E(G/K))/(G/K)=(X\times_KEG)\times_{G/K}E(G/K)
$$
$$
\simeq (X\times_KEK)\times_{G/K}E(G/K),
$$ 
where in the last step we are using that $EG\simeq EK$ (\S~\ref{S:A-G/K}).  Now using the fact that $K$ acts trivially on $X$, we obtain that this is equal to
$(X\times BK)\times_{G/K}E(G/K)$.  Considering $BK$ as the universal base for principal $K$-bundles, and that the action of $G/K$ on a $K$-principal bundle is given via conjugation, then the fact that $K$ is central implies that the action of $G/K$ on $BK$ is homotopic to the  trivial action.  Thus we finally arrive at
$BK\times (X\times_{G/K}E(G/K))$, completing the proof.
\footnote{
Alternatively,  as suggested to us by Frances Kirwan, one can consider the Leray spectral sequence for the fibration $X \times_G( EG \times E(G/K)) \to EG/K = BK$  with fiber $X \times_{G/K} E(G/K)$, and use Deligne's argument as in~\cite[p.35]{kirwan84} to show the spectral sequence degenerates.
}
\begin{exa}\label{E:H-PGL}
There is a central extension $1\to \mu_n\to \SL(n,\CC)\to \PGL(n,\CC)\to 1$.  Consequently, since $H^\bullet (B\mu_n)=\mathbb Q$, we have that if $\PGL(n,\CC)$ acts on $X$, then $H^\bullet _{\PGL(n,\CC)}(X)=H^\bullet_{\SL(n,\CC)}(X)$, for the induced $\SL(n,\CC)$ action.  Note, in particular, that applying this to the case where $X$ is a point gives $H^\bullet (B\operatorname{PGL}(n,\mathbb C))=H^\bullet(B\operatorname{SL}(n,\mathbb C))$ (Example \ref{E:BSUn}).  
\end{exa}

\section{Kirwan's result for compact groups acting on symplectic manifolds}\label{S:A-KEC}
For a compact symplectic manifold $X$ acted on by a compact connected Lie group $K$ such that the moment map exists,
it is shown in~\cite[Prop.~5.8]{kirwan84} that the Leray spectral sequence for the   fibration $X_K\to BK$ degenerates, giving
\begin{equation}\label{E:K-EC-1}
H^\bullet_K(X)\cong H^\bullet(BK)\otimes H^\bullet(X).
\end{equation}
If $K$ is disconnected, setting $K_0$ to be its identity component, it follows from~\eqref{E:K-EC-1} and~\eqref{E:AS-EC-2}   that  $H_K^\bullet (X)$ is the invariant part of
$  H^\bullet (BK_0)\otimes H^\bullet (X)$
under the action of the finite group $K/K_0$.

\section{Fibrations}
Suppose we have a right $G$-equivariant fibration
$$
\xymatrix@R=1em{
F\ar[r] \ar[d]& X \ar[d]^\pi\\
y\ar[r]& Y
}
$$
with $G$ acting transitively on $Y$, and with stabilizer $G_y$ of  a point $y\in Y$.   Then we have the following equality of equivariant Poincar\'e polynomials:
\begin{equation}\label{E:PGfib}
P^G(X)=P^{G_y}(F).
\end{equation}
This is straightforward from the definitions.

%% file: sec9.tex
\chapter{Stabilizers, normalizers, and fixed loci for cubic threefolds}\label{S:Elem}

In this section we compute some stabilizers, normalizers, and fixed loci for cubic threefolds, which have appeared in the main body of the paper.  While the computations are fairly elementary, they are nevertheless somewhat lengthy, and we have included the details here for the convenience of the reader.
\section{Connected component $\CC^*$}
We recall that $2A_5$ cubics of the form $V(F_{A,B})$ are given by equations~\eqref{eq:2A5}, and that for $4A/B^2\ne 1$ such a cubic has either exactly two $A_5$ singularities, or two $A_5$ singularities and an $A_1$ singularity. We continue to denote by $\operatorname{Aut}(V(F_{A,B}))\subseteq \PGL(5,\CC)$ and by $\operatorname{Stab}(V(F_{A,B}))\subset\SL(5,\CC)$ the stabilizers of such a cubic, and recall that by definition $R_{2A5}:=\operatorname{Stab}^0(V(F_{A,B}))$ is the connected component of the stabilizer. All these are computed by the following proposition, which enhances the statement of Lemma~\ref{L:R2A5Norm}(1) with more computations.
\begin{pro}\label{P:App-R=C*p1}
For a cubic of the form $V(F_{A,B})$ with $4A/B^2\ne 1$:
\begin{enumerate}
\item
The connected component of the  stabilizer is the $1$-PS  \begin{equation}\label{E:App-R2A5}
R_{2A5}=\operatorname{diag}(\lambda^2,\lambda,1,\lambda^{-1},\lambda^{-2})\cong \CC^*
\end{equation}
(i.e. the $1$-PS with weights $(2,1,0,-1,-2)$).
For a polystable cubic $V$, we have $\operatorname{Stab}^0(V)=R_{2A5}$ (up to conjugation) if and only if $V$ is in the orbit of $V(F_{A,B})$ with $4A/B^2\ne 1$.    These are the cubics corresponding to points on the curve $(\mathcal T-\lbrace\Xi\rbrace) \subseteq \GIT$.

\item If $4A/B^2\ne 0,1,\infty$,  then the stabilizer $\operatorname{Aut}(V(F_{A,B}))\subseteq \PGL(5,\CC)$  is
$$
 \operatorname{Aut}(V(F_{A,B}))\cong R_{2A_5}\rtimes \ZZ/2\ZZ\cong \CC^*\rtimes \ZZ/2\ZZ,
$$
where the involution is  $\tau:x_i \mapsto x_{4-i}$, and the semi-direct product is given  by the homomorphism $\ZZ/2\ZZ \to \operatorname{Aut}(\CC^*)$ defined by $ \tau \mapsto (\lambda \mapsto  \lambda^{-1})$.
Furthermore, we have
$$
1\to \mu_5\to \operatorname{Stab}(V(F_{A,B}))\to  \operatorname{Aut}(V(F_{A,B}))\to 1.
$$

\item  If $4A/B^2=\infty$,  then
$$
 \Aut(V_{F_{1,0}}) \cong (\CC^*\times \ZZ/2\ZZ)\rtimes\ZZ/2\ZZ,
$$
where the second $\ZZ/2\ZZ$ factor corresponds to the automorphism $\tau$ that exists for a generic $C$ (and thus also for $C=\infty$), while the first $\ZZ/2\ZZ$ factor is given by the involution $\sigma:(x_0:x_1:x_2:x_3:x_4)\mapsto (x_0:-x_1:x_2:x_3:x_4)$, which commutes with the diagonal action of $\CC^*$.

\item  The normalizer $N(R_{2A_5})$ is equal to
$$N(R_{2A_5})\cong \TT^4\rtimes \ZZ/2\ZZ\,,$$
where $\TT^4$ is the maximal torus, and the $\ZZ/2\ZZ$ factor corresponds to the involution $\tau:x_i \mapsto x_{4-i}$.
\end{enumerate}
\end{pro}
\begin{rem}
Before proceeding with the proof, we note that the above does not cover the case of automorphisms for $C=4A/B^2=1$, i.e. the case of the chordal cubic $F_{1,-2}$. Indeed, in Kirwan's machinery this is a separate blowup, and will be treated separately in the next proposition, Proposition~\ref{P:App-R=SL2} --- the proof of which uses this proposition.
\end{rem}

\begin{proof}
Allcock~\cite[Thm.~5.4]{allcock} states that the automorphism group of a general $F_{A,B}$ is as stated in (2), and that the automorphism group of $F_{1,0}$ is as stated in (3) --- we have just included an explicit description.

For completeness, we give a determination of the automorphism group for the case $C=0$, i.e.,~for the $2A_5+A_1$ case:
$$
 F_{0,1}=x_0x_3^2+x_1^2x_4-x_0x_2x_4+x_1x_2x_3.
$$
We make the following geometric observation, as mentioned in~\cite{allcock}. Any automorphism must map singularities of the cubic to singularities, and so must its inverse. One easily checks that $F_{0,1}$ has $A_5$ singularities at the points $(1:0:0:0:0)$ and $(0:0:0:0:1)$, and an $A_1$ singularity at $(0:0:1:0:0)$ (which is not there for cubics $F_{A,B}$ with $A\ne 0$). Thus for any automorphism $\gamma\in\Aut(V_{F_{0,1}})$, either $\gamma$ or $\tau\circ\gamma$ must fix each of these three points. Thus, after possibly composing with $\tau$, such an automorphism must have the form
$$
 g:=\left(\begin{smallmatrix} *&*&0&*&0\\ 0&*&0&*&0\\ 0&*&*&*&0\\ 0&*&0&*&0\\ 0&*&0&*&*\end{smallmatrix}\right).
$$
Denoting coefficients of this matrix by $a_{ij}$ for $0\le i,j\le 4$, we see for example that the coefficient of the monomial $x_0x_1^2$ in $F_{0,1}(gx)$
would be equal to $a_{00}a_{13}^2$. Since this coefficient must be zero, while $a_{00}$ cannot be zero in such an invertible matrix, it implies that $a_{13}=0$. Similarly from the coefficient of $x_4x_3^2$ in $ F_{0,1} (gx)$ being zero we deduce that $a_{31}=0$. Continuing in this way, one sees finally that the matrix $a$ must be diagonal.
Denoting this diagonal matrix then by $\diag(\lambda_0,\lambda_1,\lambda_2,\lambda_3,\lambda_4)$, and requiring the matrix to act on $F_{0,1}$ by scaling it by some $a$, we get the equations
$$
\lambda_0\lambda_3^2=a;\ \lambda_4\lambda_1^2=a;\ \lambda_0\lambda_2\lambda_4=a;\ \lambda_1\lambda_2\lambda_3=a.
$$
As we are interested in the automorphisms in $\PGL(5,\CC)$, all $\lambda_i$ are non-zero, and we can always rescale to make $\lambda_2=1$. We then express everything in terms of $\lambda_1$. From the last equation one gets $\lambda_3=a\lambda_1^{-1}$, from the second equation one gets $\lambda_4=a\lambda_1^{-2}$, substituting $\lambda_3$ in the first equation yields $\lambda_0=a\lambda_3^{-2}=a^{-1}\lambda_1^2$, and thus the third equation finally yields $a=\lambda_0\lambda_4=a^{-1}\lambda_1^2a\lambda_1^{-2}=1$, so that the matrix is diagonal of the form $\diag(\lambda^2,\lambda,1,\lambda^{-1},\lambda^{-2})$, i.e.,~lies in the generic $\CC^*$ stabilizer.

We finally prove (4), that is determine the normalizer $N=N(R_{2A_5})$.   For this we do a direct computation. Indeed, a matrix $n=(n_{ij})_{0\le i\le j\le 4}$ lies in $N$ if and only if for any $s\in \TT$ there exists an $s'\in \TT$ such that $nsn^{-1}=s'$, where we think of $s\in \TT$ as the diagonal matrix $\diag(s^2,s,1,s^{-1},s^{-2})$. If this is the case, the map $s\mapsto s'$ gives a homomorphism of the torus. Since conjugating by $n^{-1}$ gives an inverse, this homomorphism must be an isomorphism, and thus we must have either $s'=s$ or $s'=s^{-1}$. Furthermore, note that the involution $\J$ that gives the permutation of coordinates $x_i\mapsto x_{4-i}$ satisfies $\J s\J=s^{-1}$, and thus for any $n$ such that $nsn^{-1}=s^{-1}$ we have $(n\J)s(n\J)^{-1}=s$. This implies that the normalizer $N$ is a semidirect product of its subgroup $N_0$ consisting of $n$ such that $nsn^{-1}=s$ for all $s\in \TT$, and the $\ZZ/2\ZZ$ generated by $\J$. Finally, the matrix equality $ns=sn$ for any $n\in N_0$ translates into the equalities $n_{ij}s^{2-i}=n_{ij}s^{2-j}$ for all $0\le i\le j\le 4$ for the entries of the matrix, which must be valid for arbitrary $s$. Thus for $i=j$ there is no restriction on $n_{ij}$, while for $i\ne j$ we must have $n_{ij}=0$. This implies that $N_0\subset G$ consists of diagonal matrices, and is thus the maximal torus $\TT^4$, and $N=\TT^4\rtimes \ZZ/2\ZZ$, as claimed.
\end{proof}
We now describe the fixed locus, and the action of the normalizer on it, supplementing the statement of Lemma~\ref{L:R2A5Norm}(2), (3) with more details.
\begin{pro}\label{P:App-R=C*p2}
\begin{enumerate}
\item
The fixed locus $Z^{ss}_{R_{2A_5}}$~\eqref{E:ZRss} is the set of cubics defined by equations of the form
\begin{equation}\label{E:App-ZR2A5}
F=a_0x_2^3+a_1x_0x_3^2+a_2x_1^2x_4+a_3x_0x_2x_4+a_4x_1x_2x_3,
\end{equation}
with $a_1,a_2,a_3\ne 0$, $(a_0,a_4)\ne (0,0)$. For $(A,B)\ne (0,0)$ we have $V(F_{A,B})\in  Z^{ss}_{R_{2A_5}}$, and conversely every cubic in $Z^{ss}_{R_{2A_5}}$ is  projectively equivalent to a cubic of the form $V(F_{A,B})$ with $(A,B)\ne (0,0)$.

\item The orbit of the chordal cubic  meets $Z^{ss}_{R_{2A_5}}$
in the divisor defined by the equation
$$
4a_0a_1a_2+a_3a_4^2=0.
$$
\item $Z_{R_{2A_5}}^{ss}/N(R_{2A_5}) \cong \PP^1$.  We also have $Z_{R_{2A_5}}^{ss}/\TT^4\cong \PP^1$.
\end{enumerate}
\end{pro}

\begin{proof}
We prove (1) by describing all semi-stable cubics that are stabilized by $R=R_{2A_5}=\operatorname{diag}(\lambda^{2},\lambda,1,\lambda^{-1},\lambda^{-2})$.
 To be stabilized by this torus, the monomials must all be of the same weight with respect to that torus.  If they all have the same non-zero weight, then they would be unstable with respect to the $1$-PS $R$, and therefore unstable.  So we are reduced to looking for the monomials of weight $0$ with respect to that torus.  We obtain the projective space of weight $0$ monomials for $R$:
$$
Z_R=\PP^4=\{a_0x_2^3+a_1x_0x_3^2+a_2x_1^2x_4+a_3x_0x_2x_4+a_4x_1x_2x_3=0\}.
$$

We next use Allcock's description of the unstable locus.  For this, we note that conveniently, the monomials in question are indicated in black squares in~\cite[Fig.~3.2(c)]{allcock}.  Now, returning to~\cite[Fig.~3.1]{allcock}, describing unstable cubics, we have that~\cite[Fig.~3.1(a)]{allcock} implies the cubic is unstable if $(a_1,a_3)= (0,0)$,~\cite[Fig.~3.1(b)]{allcock} implies the same if  $a_3= 0$,~\cite[Fig.~3.1(c)]{allcock} implies the same if  $a_1=0$,~\cite[Fig.~3.1(d)]{allcock} implies the same if  $(a_0,a_4)=(0,0)$,~\cite[Fig.~3.1(e)]{allcock} implies the same if  $a_2=0$, and~\cite[Fig.~3.1(f)]{allcock} implies the same if $(a_0,a_4)= (0,0)$.  Thus in summary, the cubic is unstable if at least one of $a_1,a_2,a_3$ is $0$, or if
$(a_0,a_4)= (0,0)$. Therefore, conversely, let us assume that $a_1,a_2,a_3\ne 0$, and $(a_0,a_4)\ne (0,0)$. Then using the maximal torus $\TT^4$, we can easily put the equation for the cubic in the form $F_{A,B}$, with $(A,B)\ne (0,0)$.  As Allcock has shown these are all polystable, we see that every point in $Z^{ss}_R$ is semi-stable (in fact, polystable).  Moreover, we see that every cubic in $Z^{ss}_R$  can be taken to a cubic of the form $V(F_{A,B})$ by the action of the maximal torus $\TT^4$.

(2)  We now want to identify the orbit of the chordal cubic inside of $Z_R^{ss}$. The claim is that
$$
G\cdot \{V(F_{-1,2})\} \cap Z_R^{ss}=\{4a_0a_1a_2+a_3a_4^2=0\}\subseteq Z^{ss}_R.
$$
Given a cubic in $Z_R^{ss}$, we saw in the proof of (1) that we could take it into a cubic of the form $V(F_{A,B})$ using just the maximal torus.
So to determine if $(a_0:\dots :a_4)$ defines a cubic in the orbit of the chordal cubic, it suffices to consider the maximal torus orbit, and see whether one can take the cubic into one defined by $F_{A,B}$ with $4A/B^2=1$; in other words, to see when the torus takes $(a_0:\dots :a_4)$ into $(A:1:1:-1:B)$ with $4A/B^2=1$.
The torus $\diag(s_0,\dots,s_4)$ acts on  $(a_0:\dots:a_4)$  by
\begin{align*}
a_0&\mapsto a_0s_2^3\\
a_1&\mapsto a_1s_0s_3^2\\
a_2&\mapsto a_2s_1^2s_4\\
a_3&\mapsto a_3s_0s_2s_4\\
a_4&\mapsto a_4s_1s_2s_3.
\end{align*}
It is immediate to check that if $(a_0:\dots:a_4)=(A:1:1:-1:B)$  with $4A/B^2=1$, then the full orbit satisfies the given equation (a $2A_5$ cubic is chordal if and only if $4A/B^2=1$). Conversely, let us show that if $(a_0:\dots:a_4)$ satisfies the given equations, then we can find $\diag(s_0,\dots,s_4)$ taking $(a_0:\dots:a_4)$ into the form $(A:1:1:-1:B)$.  The first thing to note is that if $a_4$ or $a_0$ is zero, then the equation $4a_0a_1a_2+a_3a_4^2=0$ implies both are zero  (since the other $a_i$ are assumed non-zero), so we can assume none of the $a_i$ are zero.
We want $s_0,\dots,s_4$ such that:
\begin{align*}
a_1s_0s_3^2&=1\\
a_2s_1^2s_4&=1\\
a_3s_0s_2s_4&=-1\\
4a_0s_2^3-a_4^2s_1^2s_2^2s_3^2&=0.
\end{align*}
Canceling $s_2^2$, we can take the last equation as $4a_0s_2-a_4^2s_1^2s_3^2=0$.
In other words, we have
\begin{align*}
s_0&=\frac{1}{a_1s_3^2}\\
s_4&=\frac{1}{a_2s_1^2}\\
s_2&=-\frac{1}{a_3s_0s_4}\\
s_2&=\frac{a_4^2s_1^2s_3^2}{4a_0}.
\end{align*}
Taking $s_1$ and $s_3$ arbitrary defines $s_0, s_4,s_2$ via the first three equations.  Then one can check that the last equation holds, since by assumption $4a_0a_1a_2+a_3a_4^2=0$.

(3)
Since $N$ is 4-dimensional, and the stabilizer of a generic point (which is contained in $N$) is $1$-dimensional, it follows that the quotient $Z^{ss}_R/N$ is $1$-dimensional. As this quotient is clearly unirational and normal (it is the quotient of a normal space by a reductive group action),
it must be an open subset of  $\PP^1$.  Since the copy of $\PP^1\subseteq Z^{ss}_{R}$ given by $V(F_{A,B})$ for $(A,B)\ne (0,0)$ surjects onto the quotient, the quotient is also compact, and is therefore isomorphic to $\PP^1$.  The identical proof works for the quotient $Z_{R_{2A_5}}^{ss}/\TT^4\cong \PP^1$.
\end{proof}
\section{Connected component $\PGL(2,\CC)$}
\begin{pro}\label{P:App-R=SL2}
For the cubic of the form $V(F_{1,-2})$~\eqref{eq:2A5}; i.e.,  the chordal cubic, the connected component of the  stabilizer is
\begin{equation}\label{E:App-Rc}
R_{c}:=\operatorname{Stab}^0(V(F_{1,-2}))\cong\PP \GL(2,\CC)
\end{equation}
given as the copy of $\PGL(2,\CC)$ embedded into $\SL(5,\CC)$ via the representation  $\Sym^4(\CC^2)$ ($\cong \CC^5$).
 For a polystable cubic $V$, we have $\operatorname{Stab}^0(V)=R_{c}$ (up to conjugation) if and only if $V$ is in the orbit of $V(F_{A,B})$ with $4A/B^2= 1$; i.e., if and only if the cubic is projectively equivalent to the chordal cubic.   These are the cubics corresponding to  the point  $\Xi \in \GIT$.
Moreover, we have
\begin{enumerate}
\item The full stabilizer group of $V(F_{1,-2})$ in $\PGL(5,\CC)$ is $\PGL(2,\CC)$, and thus there is a split central extension
\begin{equation}\label{E:App-StabCh1}
1\to \mu_5\to \operatorname{Stab}(V(F_{1,-2}))\to \PGL(2,\CC)\to 1.
\end{equation}
\item  The normalizer $N(R_c)$ is equal to the stabilizer $\operatorname{Stab}(V(F_{1,-2}))$.
\item The fixed locus is $Z^{ss}_{R_c}=\{V(F_{1,-2})\}$; i.e., it is the point corresponding to the chordal cubic.
\end{enumerate}
\end{pro}

\begin{proof}
The fact
\eqref{E:App-Rc} follows from~\cite{allcock}.  Indeed, the stabilizer in $\PGL(5,\CC)$ of the cubic  $V(F_{1,-2})$  is computed in~\cite[Thm.~5.4]{allcock} to be $\PGL(2,\CC)$ embedded via the $\operatorname{Sym}^4$-representation.  This immediately gives~\eqref{E:App-StabCh1}: to show that the connected component of the identity is $\PGL(2,\CC)$, it suffices to construct a section of~\eqref{E:App-StabCh1}. For this, observe that the standard representation of $\SL(2,\CC)$ on $\CC^2$ induces a homomorphism $\SL(2, \CC)\to \SL(\operatorname{Sym}^4\CC^2)$, with kernel equal to $\mu_2$; in other words, the image is $\PGL(2,\CC)$, providing the section.

For (2) it is convenient to recall the $\operatorname{Sym}^4\CC^2$ representation of $\SL(2,\CC)$ explicitly. The matrix
$$
\left(
\begin{array}{cc}
a&b\\
c&d
\end{array}
\right)\in \SL(2,\CC)
$$
acts on $\CC^2$ by sending homogeneous coordinates $(t_0:t_1)$ to $(at_0+bt_1: ct_0+dt_1)$.  Then, in terms of the standard basis for $\operatorname{Sym}^4\CC^2$:
$$
(t_0^4:t_0^3t_1:t_0^2t_1^2:t_0t_1^3:t_1^4),
$$
the action of $
\left(
\begin{array}{cc}
a&b\\
c&d
\end{array}
\right)$ is given by the rule:
\begin{align*}
t_0^4&\mapsto (at_0+bt_1)^4=a^4t_0^4+4a^3bt_0^3t_1+6a^2b^2t_0^2t_1^2+4ab^3t_0t_1^3+b^4t_1^4\\
t_0^3t_1&\mapsto (at_0+bt_1)^3(ct_0+dt_1)=\dots\\
\vdots &
\end{align*}
Thus the induced homomorphism  $\SL(2,\CC) \to \SL(5,\CC)$ is given explicitly by:
\begin{equation}\label{E:SL2RepCh}
\scriptstyle
\left(
\begin{array}{cc}
a&b\\
c&d
\end{array}
\right)\mapsto
\left(
\begin{array}{ccccc}
a^4&a^3c &a^2c^2 &ac^3&c^4\\
4a^3b&3a^2bc+a^3d&2abc^2+2a^2cd&bc^3+3ac^2d&4c^4d\\
6a^2b^2&3ab^2c+3a^2bd&b^2c^2+4abcd+a^2d^2&3bc^2d+3acd^2&6c^2d^2\\
4ab^3&b^3c+3ab^2d&2b^2cd+2abd^2&3bcd^2+ad^3&4cd^3\\
b^4&b^3d&b^2d^2&bd^3&d^4\\
\end{array}
\right).
\end{equation}
The kernel is given by the matrix $\left(
\begin{array}{cc}
-1&0\\
0&-1
\end{array}
\right)$ confirming that the image of $\SL(2,\CC)$ in $\SL(5,\CC)$ is $\PGL(2,\CC)$.

We now move on to the proof of (2). We first introduce the matrix
$\tau$ which is the matrix associated to the involution sending $x_i$ to $x_{4-i}$; note that $\tau$ is the image of the matrix $\left(
\begin{array}{cc}
0&i\\
i&0
\end{array}
\right)$ in $\SL(2,\CC)$. This will be needed in the proof of the next claim.

Since $R_c$ being the connected component is clearly normal in $\operatorname{Stab}(V(F_{1,-2}))$, we have $\operatorname{Stab}(V(F_{1,-2}))\subseteq N(R_c)=N$.  For the converse, we argue with two claims:

\vskip .2 cm \emph{Claim 1: For any $n\in N$, there is a $g\in R_c$ with $ng\in \TT^4\cap N$, where $\TT^4$ is the maximal torus.}
\vskip .2cm \noindent
Indeed, any element $n\in N$ must conjugate the standard maximal torus $\TT\subset\PGL(2,\CC)$, which is embedded into $\SL(5,\CC)$, into some torus $\TT'\subset\SL(5,\CC)$. Since all such tori are conjugate under the action of $\PGL(2,\CC)$, this means there must exist some $g'\in\PGL(2,\CC)$ such that $n':=ng'$ fixes the maximal torus $\TT$ as a set, which is simply to say that $n'$ lies in the normalizer of $R_{2A_5}$, computed in Proposition~\ref{P:App-R=C*p1} to be the subgroup generated by $\TT^4$ and $\tau$. Thus for $i\in \{0,1\}$, we have $n'':=ng'\tau^i\in \TT^4$. We may as well replace $g'$ with $g=g'\tau\in \PGL(2,\CC)$.

\vskip .2 cm \emph{Claim 2: $\TT^4\cap N\subseteq \langle \mu_5,\PGL(2,\CC)\rangle=\operatorname{Stab}(V(F_{1,-2}))$.}
\vskip .2cm \noindent
This will suffice to prove (2), since then for any $n\in N$, there is a $g\in R_c$ such that $ng=s\in \operatorname{Stab}(V(F_{1,-2}))$.  Since $R_c\subseteq \operatorname{Stab}(V(F_{1,-2}))$, we have $n\in \operatorname{Stab}(V(F_{1,-2}))$.

Thus we just need to show the claim.
For this we consider the special case of upper triangular matrices $\left(\begin{smallmatrix} 1&t\\0&1\end{smallmatrix}\right)\in\SL(2,\CC)$ for arbitrary $t\in\CC$.
The fourth symmetric power of such a matrix gives its action as an element $M_t$ of $\SL(5,\CC)$:
\begin{equation}\label{eq:sym4t}
\scriptstyle
  M_t\circ \begin{pmatrix}x_0\\ x_1\\ x_2\\ x_3\\ x_4\end{pmatrix}=\begin{pmatrix} 1&4t&6t^2&4t^3&t^4\\ 0&1&3t&3t^2&t^3\\ 0&0&1&2t&t^2\\ 0&0&0&1&t\\ 0&0&0&0&1\end{pmatrix}\circ \begin{pmatrix}x_0\\ x_1\\ x_2\\ x_3\\ x_4\end{pmatrix}=\left(\begin{array}{r} x_0+4t\ x_1+6t^2x_2+4t^3x_3+t^4x_4\\ x_1+3t\ x_2+3t^2x_3+t^3x_4\\ x_2+2t
  \ x_3+t^2x_4\\ x_3+t\ x_4\\ x_4\end{array}\right).
\end{equation}
and we need to check whether a diagonal matrix $d\in \TT^4$ (where $ \TT^4$ is the maximal torus of $\SL(5,\CC)$), can conjugate $M_t$ to the action of some element of $\SL(2,\CC)$. Since conjugating an upper triangular matrix with $1$'s on the diagonal by a diagonal matrix leaves it upper-triangular with $1$'s on the diagonal, we need to check when for any $t\in\CC$ there exists a $t'\in\CC$ such that $dM_t d^{-1}=M_{t'}$. Again, $t\mapsto t'$ is then an isomorphism of the additive group, so that it is either the identity or $t\mapsto -t$.
If the map on $t$ is the identity, i.e.,~if for any $t$ the identity $dM_td^{-1}=M_t$, holds, then the equality of the last columns of these matrices yields that each $d_i/d_4$ must be equal to one, so that all $d_i$ are equal, and thus $d$ is scalar multiplication by an arbitrary $5$th root of unity. We note that such a scalar multiplication is not an element of $\SL(2,\CC)$ because it is easily checked not to be an element of the diagonal maximal torus $\TT$ as above. On the other hand, for the case when for any $t$ the identity $dM_{t}d^{-1}=M_{-t}$ holds, looking again at the last column of these matrices shows that $d_0=d_2=d_4=-d_3=-d_1$, so that $d$ is the product of $\diag(1,-1,1,-1,1)$ and an arbitrary scalar fifth root of unity. However, the diagonal matrix $\diag(i,-i)\in\SL(2,\CC)$ gives rise precisely to the matrix $\diag(1,-1,1,-1,1)$ under the fourth symmetric power map, and thus this diagonal matrix is already accounted for by the $\SL(2,\CC)$.

(3) We now determine the set of cubics fixed by the action of $\PGL(2,\CC)$.  Let $V$ be such a cubic.  Since $R_{2A_5}\subseteq \PGL(2,\CC)$, we must have that $V\in Z^{ss}_{R_{2A_5}}$.  If $V$ were not in the orbit of the chordal cubic, then we have seen in Proposition~\ref{P:App-R=C*p1} that the stabilizer would have dimension $1$, which would be a contradiction. Thus $V$ is in the orbit of the chordal cubic, say $V=g\cdot V(F_{1,-2})$.  But then the connected component $R_V$ of the stabilizer of $V$ is equal to  $gR_cg^{-1}$.  If $V$ is fixed by $R_c$, then for dimension reasons, we must have $R_V=R_c$, so that $g$ is in the normalizer of $R_c$.  But we saw in (1) that $N(R_c)=\operatorname{Stab}(V(F_{1,-2}))$, so that $V=V(F_{1,-2})$.
\end{proof}

\begin{rem}
Recall that in the construction of the Kirwan blowup $\MK$,  one first blows up the point $\Xi \in \calM^{GIT}$ corresponding to the chordal cubic, followed by a blowup of the strict transform of the rational curve $\mathcal T$ parameterizing $2A_5$ cubics (the point $\Delta$, corresponding to the $3D_4$ cubic, can be dealt with separately).  To fix notation, let $\widehat D_c$ be the exceptional divisor of the blowup of $\Xi$, and let $\widehat {\mathcal T}$ be the strict transform of $\mathcal T$ in this blowup.   We explain here   that $\widehat {\mathcal T}$ meets $\widehat D_c$ in a single point.

On the one hand, by investigating the proof of Proposition~\ref{P:App-R=C*p2}(3), describing $\mathcal T$ as the quotient $Z^{ss}_{R_{2A_5}}/N(R_{2A_5})$, one can show that $\mathcal T$ is locally unibranched near $\Xi$, and thus
that $\widehat {\mathcal T}$ meets $\widehat D_c$ in a single point.
On the other hand, this can alternatively be seen via the identification of $\widehat D_c$ with the GIT of 12 points on $\mathbb P^1$.

More precisely,  this one point of intersection of $\widehat{\mathcal T}$ and $\widehat{D_c}$ can be identified as follows.  By construction, and smoothness of the Kirwan blowup up to finite quotient singularities, every point of intersection of  $\widehat{\mathcal T}$ with the exceptional divisor $\widehat D_c$ must have a stabilizer containing $\CC^*$.  On the other hand, since the exceptional divisor does not intersect the locus of $3D_4$ cubics, and there are no further blowups in constructing $\MK$, any point on $\widehat D_c$ with a $\CC^*$ contained in its stabilizer must be contained in  $\widehat {\mathcal T}$. Now, since $\widehat D_c$ is isomorphic to the GIT quotient of 12 points in $\PP^1$, the only strictly semi-stable points are where precisely 6 of the 12 points have come together; moreover, one can see immediately that for such a point to have an infinite stabilizer requires the remaining 6 points to also have come together. Thus the only strictly semi-stable points are where the 12 points were separated in two groups of 6.  In other words, $\widehat{\mathcal T}\cap \widehat{D_c}$ is the strictly semi-stable point of $\widehat D_c$ corresponding to the case where the 12 points were separated in two groups of 6.

Having explained that $\widehat{\mathcal T}\cap \widehat{D_c}$ consists of a single point, we now point out further that with the identification of $\widehat D_c$ as the GIT of 12 points in $\mathbb P^1$, the stabilizers of all of the points of $\widehat{\mathcal T}$ can be described uniformly (as extension of $\mathbb C^*$ as in Proposition~\ref{P:App-R=C*p1}).   Indeed, consider this point $\widehat{\mathcal T}\cap \widehat{D_c}$.
By acting by $\PGL(2,\CC)$, we can  move the two underlying points (of the pairs of 6 points)  to $0$ and $\infty\in\PP^1$, respectively, so that $\CC^*$ acts by rescaling the coordinate~$z$, and there is an extra involution $z\mapsto 1/z$, so that the stabilizer of these two 6-tuples of points (recall that the points are unlabeled) is $\CC^*\rtimes\ZZ/2\ZZ$.

Recalling that the automorphisms of the chordal cubic were identified with the automorphisms of the rational normal curve, we can describe this group as follows.
One can identify the rational normal curve explicitly in coordinates as $(t_0^4:t_0^3t_1:t_0^2t_1^2:t_0t_1^3:t_1^4)$ as done in the proof of Proposition~\ref{P:App-R=SL2}. Then the action of the involution $\tau$ is induced by the involution $(t_0:t_1)\mapsto (t_1:t_0)$ on $\PP^1$, and thus the stabilizer of the point $\widehat{\mathcal T}\cap \widehat{D_c}$ can be identified concretely as a subgroup of $\PGL(5,\mathbb C)$, and is the same as for points of $\mathcal T$ with  $C\ne 0,1,\infty$ (as described in Proposition~\ref{P:App-R=C*p1}(1)).
\end{rem}

\section{Connected component $(\CC^*)^2$}
We now give the computations for the $3D_4$ case, proving Lemmas~\ref{L:R3D4Norm1} and~\ref{L:R3D4Norm2} and providing some more information. We use the notation for the groups involved in the statements of these lemmas.
\begin{pro}\label{P:App-R=C*2}
\begin{enumerate}
\item For the cubic of the form $V(F_{3D_4})$~\eqref{eq:3D4}, i.e.,  with $3D_4$ singularities, the connected component $R_{3D4}$ of the  stabilizer in $\SL(5,\CC)$  is given by equation~\eqref{E:R3D4main}. For a polystable cubic $V$, we have $\operatorname{Stab}^0(V)=R_{3D_4}$ (up to conjugation) if and only if $V$ is in the orbit of $V(F_{3D_4})$; i.e., if and only if the cubic has exactly $3D_4$ singularities.  These are the cubics corresponding to  the point  $\Delta \in \GIT$.
\item The normalizers and stabilizers in $\SL(5,\CC)$, $\PGL(5,\CC)$, $\GL(5,\CC)$ are as given in Lemma~\ref{L:R3D4Norm2}, described as certain central extensions in terms of the group $D$ defined there.
\item The fixed locus $Z^{ss}_{R_{3D_4}}$ is the set of cubics defined by equations of the form
$$
x_0x_1x_2+P_3(x_3,x_4)
$$
where $P_3(x_3,x_4)$ is an arbitrary homogeneous cubic with three distinct roots, and the normalizer $N(R_{3D_4})$ acts on it transitively, as stated in Lemma~\ref{L:R3D4Norm1} (3).
\end{enumerate}
\end{pro}

\begin{proof} (1) and (2): We compute explicitly all the groups involved.
We first derive the stabilizer group $\GL_{V(F_{3D_4})}$ of $V(F_{3D_4})$ in   $\GL(5,\CC)$.
To begin, it is clear that the group
\begin{equation}\label{E:App-GF3D4-1}
\left\{\left(
\begin{array}{c|c}
\SSS_3&\\ \hline
&\SSS_2\\
\end{array}
\right): \lambda_0\lambda_1\lambda_2=\lambda_3^3=\lambda_4^3\right\}\subseteq \GL(5,\CC)
\end{equation}
stabilizes $V(F_{3D_4})$. We wish to show that this is all of the matrices in the stabilizer.  For this, we observe that any symmetry must permute  the $3$ singularities of the cubic, and thus permute the points $(1:0:0:0:0)$, $(0:1:0:0:0)$ and $(0:0:1:0:0)$.  This forces a matrix stabilizing $V(F_{3D_4})$ to be of the form:
$$
\left(
\begin{array}{c|c}
\SSS_3&*\\ \hline
0&\GL_2\\
\end{array}
\right).
$$
Such a transformation sends the monomial $x_0x_1x_2$ to $(\lambda_0 x_0+*x_3+*x_4)\cdot (\lambda_1 x_1+*x_3+*x_4)\cdot (\lambda_2 x_2+*x_3+*x_4)$, where all the $\lambda$'s are non-zero, and $*$ are the entries of the unknown $2\times 3$ block of the matrix. Furthermore, $x_3$ and $x_4$ are sent to linear combinations of only $x_3$ and $x_4$. Thus all entries $*$ must be equal to zero, or otherwise applying this transformation to $F_{3D_4}$ would give a cubic with non-zero coefficient of some monomial $x_ax_b x_c$ with $0\le a<b\le 2$ and $3\le c\le 4$. Thus we have deduced that the matrix stabilizing $V(F_{3D_4})$ must actually be of the form
$$
\left(
\begin{array}{c|c}
\SSS_3&0\\ \hline
0&\GL_2\\
\end{array}
\right).
$$
However, for a matrix in $\GL_2$ acting on the span of $x_3$ and $x_4$ to stabilize $x_3^3+x_4^3$, it must lie in $\SSS_2$, or some cross terms would appear, and thus the stabilizer can only contain matrices of the form
$$
\left(
\begin{array}{c|c}
\SSS_3&0\\ \hline
0&\SSS_2\\
\end{array}
\right).
$$
Finally the conditions $\lambda_0\lambda_1\lambda_2=\lambda_3^3=\lambda_4^3$ are obvious.  This completes the proof that the stabilizer group is as claimed.

We now want to describe the structure of the stabilizer group   $\GL_{V(F_{3D_4})}$ in $\GL(5,\CC)$ more precisely.
There is clearly a left exact sequence
$$
1\to D\to \GL_{V(F_{3D_4})} \to S_3\times S_2\to 1
$$
where $D$ is the subgroup of diagonal matrices in $\GL_{V(F_{3D_4})}$, and the map to $S_3\times S_2$ is the one taking a generalized permutation matrix to the associated permutation matrix.  There is an obvious section $S_3\times S_2\to \GL_{V(F_{3D_4})}$, viewing $S_3\times S_2$ as block diagonal permutation matrices. This means
$$
\GL_{V(F_{3D4})}\cong D\rtimes (S_3\times S_2)
$$
where the action of $S_3\times S_2$ on $D$ is to permute the entries.

We now wish to describe $D$.
Concretely,  $D=\{\diag(\lambda _0,\lambda _1,\lambda_2,\lambda_3,\lambda_4): \lambda_0\lambda_1\lambda_2=\lambda_3^3=\lambda_4^3\}$.
Fixing the torus $\TT^3=\diag(\lambda_0,\lambda_1,\lambda_0^{-1}\lambda_1^{-1}\lambda_3^3,\lambda_3,\lambda_3)\cong (\CC^*)^3$, we have $\TT^3\subseteq D$, and we now describe the quotient.  Given an element of $D$, then up to elements of $\TT^3$, we may assume it is of the form
$\diag(1,1,\lambda_2,1,\lambda_4)$.  But then we must have $1\cdot 1\cdot \lambda_2= 1^3=\lambda_4^3$, so that $\lambda_2=1$ and $\lambda_4$ is a $3$-rd root of unity.
Fixing the group $\mu_3=\diag(1,1,1,1,\zeta^i)\cong \ZZ/3\ZZ$ where $\zeta$ is a primitive $3$-rd root of unity, we have
$$
D=\mathbb T^3\times \mu_3.
$$

We determine the normalizer $N(R_{3D_4})$  by an explicit computation. Indeed, if a matrix $n=(n_{ij})_{0\le i\le j\le 4}$ lies in $N$, then for any $(s_1,s_2)\in \mathbb T^2$ we have
\begin{equation}\label{eq:normT2}
\diag(s_1,s_2,s_1^{-1}s_2^{-1},1,1)\cdot n=n\cdot \diag(t_1,t_2,t_1^{-1}t_2^{-1},1,1)
\end{equation}
for some $(t_1,t_2)\in \mathbb T^2$. We first observe that~\eqref{eq:normT2} immediately implies that for any $0\le i\le 2$ and $3\le j\le 4$ we must have $n_{ij}=0$. Furthermore, we note that this equality implies no restrictions whatsoever on the entries $n_{33},n_{34},n_{43},n_{44}$, which can thus be arbitrary. The map $f:(s_1,s_2)\mapsto (t_1,t_2)$ is an automorphism of $\mathbb T^2$, which is to say that $t_1=s_1^as_2^b$ and $t_2=s_1^cs_2^d$ for some matrix $\left(\begin{smallmatrix}a & b \\ c & d \end{smallmatrix}\right)\in\SL(2,\ZZ)$. By writing down the conditions for the entries $n_{ij}$ with $0\le i\le j\le 2$ of the matrix, we see that these elements can be non-zero only if the map $f$ permutes the three diagonal entries $s_1,s_2,s_1^{-1}s_2^{-1}$. Conversely, any such permutation lies in the normalizer with respect to $\GL(5,\CC)$. If this permutation, as an element of $S_3$, is even, we compose $n$ with this permutation of coordinates $x_0,x_1,x_2$; if such a permutation is odd, we compose $n$ with this permutation of $x_0,x_1,x_2$, together with changing the signs of $x_0,x_1,x_2$ (so that the resulting transformation is still in $\SL(5,\CC)$). Thus $N$ is a semidirect product of $S_3$ and of the normal subgroup $N_0\subset N$ for which $f$ is the identity map. Finally, if $f$ is the identity map, so that $t_1=s_1$ and $t_2=s_2$, then clearly~\eqref{eq:normT2} implies that the submatrix $(n_{ij})_{0\le i\le j\le 2}$ is diagonal. Thus finally $N_0$ is the intersection of $\mathbb T^3\times \GL(2,\CC)$ with $\SL(5,\CC)$, and thus $N$ is as claimed.

(3) We now describe the fixed locus $Z^{ss}_{R_{3D_4}}$.  As usual, to be semi-stable, and fixed by $R_{3D_4}\cong (\CC^*)^2$, the cubic must be defined by monomials of weight $0$ with respect to any $1$-PS in $R_{3D_4}$.  It is easy to see that the only such monomials are
$$
x_0x_1x_2+P_3(x_3,x_4)
$$
where $P_3(x_3,x_4)$ is a homogeneous cubic.
 Allcock has shown that these are semi-stable if and only if $P_3(x_3,x_4)$ has $3$ distinct roots; i.e., the cubic has exactly $3D_4$ singularities.  More precisely, as mentioned earlier,~\cite[Thm.~4.1]{allcock} shows that the orbit of $V(F_{3D_4})$ is closed in the semi-stable locus.  But any cubic as above with $P_3(x_3,x_4)$ having multiple roots is in the closure of the orbit of $V(F_{3D_4})$, but does not have $3D_4$ singularities, which is a contradiction.

Finally, the  matrices of the form
$$
\left(
\begin{array}{c|c}
\operatorname{Id}_3&0\\ \hline
0&\SL_2\\
\end{array}
\right)
$$
that lie in the normalizer clearly act transitively on $Z^{ss}_{R_{3D_4}}$.
\end{proof}

\begin{cor}\label{C:App-ZssR-rel}
We have the following relationships among the fixed loci:
\begin{equation}\label{E:App-ZssR-Rel}
Z^{ss}_{R_c}\subset Z^{ss}_{R_{2A5}}, \ \ \ \  Z^{ss}_{R_{2A_5}} \cap Z^{ss}_{R_{3D4}}=\emptyset.
\end{equation}
\end{cor}

\begin{proof}
The inclusion on the left follows immediately from the first inclusion in~\eqref{E:Rcont}.  For the equation on the right in~\eqref{E:App-ZssR-Rel}, suppose that $x\in Z^{ss}_{R_{2A_5}}\cap Z^{ss}_{R_{3D_4}}$, and let $V$ be the corresponding cubic.  Then $\operatorname{Aut}^0(V)\supseteq R_{2A_5}\cup R_{3D_4}$, and one can see this implies it contains a $3$-torus isomorphic to $(\CC^*)^3$.  On the other hand, $V$ degenerates to a polystable cubic, and  consequently we have that $\operatorname{Aut}^0(V)$ is contained in a conjugate of $ R$ for some $R\in \calR$. For dimension reasons, it would have to be contained in a conjugate of $R_c=\SL(2,\CC)$, but this does not contain a $3$-torus.
\end{proof}

We recall from Lemma~\ref{L:R3D4Norm1} the normalizer
$$
N=N(R_{3D_4})=
\left\{\left(
\begin{array}{c|c}
\SSS_3&\\ \hline
&\GL_2\\
\end{array}
\right)\in \SL(5,\CC)\right\}\,,
$$
and define a subgroup $N_0$:
$$
N_0:=
\left\{\left(
\begin{array}{c|c}
\TT^3&\\ \hline
&\GL_2\\
\end{array}
\right)\in \SL(5,\CC)\right\}.
$$

Recall also that the stabilizer $G_x=G_{3D_4}$ of $x=V(F_{3D_4})$ in $G=\SL(5,\CC)$ is:
\begin{equation*}
G_{F_{3D4}}=
\left\{\left(
\begin{array}{c|c}
\SSS_3&\\ \hline
&\SSS_2\\
\end{array}
\right)\in \SL(5,\CC): \lambda_1\lambda_2\lambda_3=\lambda_4^3=\lambda_5^3\right\}.
\end{equation*}
Here $\lambda_i$ is the unique non-zero entry in column $i$. We now compute the relevant stabilizers and their action, proving Lemma~\ref{L:R(rho)3D4p2} and providing more details. We record two propositions, separately for the cases when $\beta'$ correspond to the codimension 4 and codimension 5 strata, as in the cases (a) and (b) of Lemma~\ref{L:R(rho)3D4p1}, respectively.
\begin{lem}\label{L:App-R(rho)3D4a}
For $\beta'=\frac{1}{2}(-\frac{2}{3},\frac{1}{3},\frac{1}{3})$ (case (a) of Lemma~\ref{L:R(rho)3D4p1}),
we have
\begin{enumerate}
\item
$$
\operatorname{Stab}_G\beta'=
\left\{\left(
\begin{array}{c|c|c}
\CC^*&&\\ \hline
&\GL_2&\\ \hline
&&\GL_2
\end{array}
\right)\in \SL(5,\CC)\right\},
$$
$$
N\cap \operatorname{Stab}_G\beta' =
\left\{\left(
\begin{array}{c|c|c}
\CC^*&&\\ \hline
&\SSS_2&\\ \hline
&&\GL_2
\end{array}
\right)\in \SL(5,\CC)\right\}\,.
$$

\item The group $N\cap \operatorname{Stab}_G\beta'$ acts transitively on $Z^{ss}_R$.

\item The stabilizer of a point is
$$\scriptstyle{
(N\cap \operatorname{Stab}_G\beta')_x=G_{F_{3D_4}}\cap N\cap \operatorname{Stab}_G\beta' =
\left\{\left(
\begin{array}{c|c|c}
\CC^*&&\\ \hline
&\SSS_2&\\ \hline
&&\SSS_2
\end{array}
\right)\in \SL(5,\CC):\lambda_1\lambda_2\lambda_3=\lambda_4^3=\lambda_5^3\right\}.}
$$

\item The locus $Z^{ss}_{\beta'}$ is
$$
Z^{ss}_{\beta'}=\{[a:b]\in \PP\CC\langle  x_1x_3x_4,x_2x_3x_4\rangle: a\ne 0, b\ne 0 \}\cong \CC^*.
$$
\item The action of $(N\cap \operatorname{Stab}_G\beta')_x$ on $Z^{ss}_{\beta'}$ is induced by change of coordinates, via the inclusion $(N\cap \operatorname{Stab}_G\beta')_x\subseteq \SL(5,\CC)$, and the description of the loci above in terms of cubic forms. In fact
that $(N\cap \operatorname{Stab}_G\beta')_x$ acts transitively on $Z^{ss}_{\beta '}$, and the stabilizer of the point $(1:1)\in Z^{ss}_{\beta'}$ is given by

$$\scriptstyle{
((N\cap \operatorname{Stab}_G\beta')_x)_{(1:1)}
=\left\{\left(
\begin{array}{c|c|c}
\CC^*&&\\ \hline
&\SSS_2&\\ \hline
&&\SSS_2
\end{array}
\right)\in \SL(5,\CC):\lambda_0\lambda_1\lambda_2=\lambda_3^3=\lambda_4^3,\ \lambda_1=\lambda_2\right\}}.
$$
Here $\lambda_i$ is the unique non-zero entry in column $i$.
In fact we have $$((N\cap \operatorname{Stab}_G\beta')_x)_{(1:1)}\cong (\CC^*\times \mu_{15})\times (S_2\times S_2)\,,$$ where $\CC^*=\diag(\lambda^{-2},\lambda,\lambda,1,1)$, $\mu_{15}=\diag(\zeta^{3i},1,1,\zeta^i,\zeta^{-4i})$ for $\zeta$ a primitive $15$-th root of unity, and the first (resp.~second) copy of $S_2$ is the subgroup of $((N\cap \operatorname{Stab}_G\beta')_x)_{(1:1)}$ generated by the
matrix
$$
\left(
\begin{array}{ccccc}
-1&&&&\\
&0&1&&\\
&1&0&&\\
&&&-1&\\
&&&&-1\\
\end{array}
\right)
,\ \
\left(resp.~
\left(
\begin{array}{ccccc}
-1&&&&\\
&1&&&\\
&&1&&\\
&&&0&-1\\
&&&-1&0\\
\end{array}
\right)
\right).
$$
\end{enumerate}
\end{lem}
The results for the codimension 5 orbits are as follows.
\begin{lem}\label{L:App-R(rho)3D4b}
For $\beta'= \frac{1}{7}(2,1,-3)$ (case (b) of Lemma~\ref{L:R(rho)3D4p1}), we have
\begin{enumerate}
\item
$$
\operatorname{Stab}_G\beta'
= N_0=
\left\{\left(
\begin{array}{c|c}
\TT^3&\\ \hline
&\GL_2
\end{array}
\right)\in \SL(5,\CC)\right\},
$$
$$
N\cap \operatorname{Stab}_G\beta' = N_0=
\left\{\left(
\begin{array}{c|c}
\TT^3&\\ \hline
&\GL_2
\end{array}
\right)\in \SL(5,\CC)\right\}.
$$

\item The group $N\cap \operatorname{Stab}_G\beta'$ acts transitively on $Z^{ss}_R$.

\item The stabilizer of a point is
$$\scriptstyle{
(N\cap \operatorname{Stab}_G\beta')_x=G_{F_{3D_4}}\cap N\cap \operatorname{Stab}_G\beta' =
\left\{\left(
\begin{array}{c|c}
\TT^3&\\ \hline
&\SSS_2\\
\end{array}
\right)\in \SL(5,\CC):\lambda_1\lambda_2\lambda_3=\lambda_4^3=\lambda_5^3\right\}}.
$$

\item The locus $Z^{ss}_{\beta'}$ is
$$
Z^{ss}_{\beta'}=\{[a:b:c]\in \PP\CC\langle  x_0x_3x_4,x_1^2x_3, x_1^2x_4\rangle: a\ne 0,\ \text{and}\  (b,c)\ne (0,0) \}\cong \AAA^2-\{0\}.
$$

\end{enumerate}
\end{lem}
We prove both lemmas in parallel.
\begin{proof}
(1) Given $\beta'$, we are first looking at computing the stabilizer $\operatorname{Stab}_G\beta'$ for the group $G=\SL(5,\CC)$ acting by the adjoint representation; i.e.,  conjugation.  Since all of our $\beta'$ are given explicitly as diagonal matrices, this is quite easy.  Indeed, given any diagonal matrix $D=\operatorname{diag}(d_1,\dots,d_n)$ and any $n\times n$ matrix $A$, the condition that $AD=DA$ is given by $d_ia_{ij}=d_ja_{ij}$.    In other words: if $d_i=d_j$, then $a_{ij}$ may be arbitrary; if $d_i\ne d_j$, then $a_{ij}=0$.  The rest is an elementary computation.

(2) It is immediate that $N_0$ acts transitively on $Z^{ss}_R$.  Thus, since for each $\beta'$ in either case (a) or (b) we have $N_0\subseteq N\cap \operatorname{Stab}_G\beta'$, and we are done.

(3) Recall that we computed
\begin{equation*}
G_{F_{3D4}}=
\left\{\left(
\begin{array}{c|c}
\SSS_3&\\ \hline
&\SSS_2\\
\end{array}
\right)\in \SL(5,\CC): \lambda_1\lambda_2\lambda_3=\lambda_4^3=\lambda_5^3\right\}.
\end{equation*}
The rest follows immediately from the previous parts.

(4) This follows immediately from the definitions, by inspection of the previous computations.

(5) (for case (a) only) We have $\TT^2=\diag (\lambda_0,\lambda_1,\lambda_0^{-1}\lambda_1^{-1},1,1)\subseteq (N\cap \operatorname{Stab}_G\beta')_x$.
The action of $\TT^2$ on $Z^{ss}_{\beta'}$ is given by
$\diag (\lambda_0,\lambda_1,\lambda_0^{-1}\lambda_1^{-1},1,1) \cdot (a:b)=(\lambda_1a:\lambda_0^{-1}\lambda_1^{-1}b)$, thus the action of $\TT^2$ on $Z^{ss}_{\beta'}$ is transitive, and therefore the same is true of $(N\cap \operatorname{Stab}_G\beta')_x$.  The stabilizer $((N\cap \operatorname{Stab}_G\beta')_x)_{(1:1)}$ is easily worked out to be as claimed, from the previous description of $(N\cap \operatorname{Stab}_G\beta')_x$.  The direct product decomposition can be deduced as follows.  First, let $D'$ be the diagonal matrices in $((N\cap \operatorname{Stab}_G\beta')_x)_{(1:1)}$.  There is a short exact sequence $$1\to D'\to ((N\cap \operatorname{Stab}_G\beta')_x)_{(1:1)}\to S_2\times S_2\to 1\,,$$ and the matrices given above clearly define a section.  Those matrices commute, and commute with the diagonal matrices, and so we obtain a direct product $D'\times (S_2\times S_2)$.
Now we have essentially already analyzed the diagonal matrices $D'$; indeed we described a group $D\subseteq \GL(5,\CC)$ of diagonal matrices in Proposition~\ref{P:App-R=C*2}(1), with $D'\subseteq D\cong (\CC^*)^3\times \mu_3$.  One can easily deduce the structure of $D'$ from this.
For clarity, we reproduce the argument in this special case.
Assume we have a diagonal matrix $\operatorname{diag}(\lambda_1,\lambda_2,\lambda_3,\lambda_4,\lambda_5)\in D'$.  Since we are only interested up to the torus $\TT=\diag(\lambda^{-2},\lambda,\lambda,1,1)$, we may scale so that $\lambda_1=\lambda_2=1$.  We now have that $\lambda_0=\lambda_3^3=\lambda_4^3$, and $\lambda_0\lambda_3\lambda_4=1$.  Together these imply that $\lambda_3^4\lambda_4=1$.
This implies that $\lambda_4=\lambda_3^{-4}$.
This implies
$\lambda_3^3=\lambda_4^3=(\lambda_3^{-4})^3=\lambda_3^{-12}$, so that $\lambda_3^{15}=1$;
i.e., $\lambda_3$ is a $15$-th root of unity.  In other words, up to scaling by the torus, any diagonal matrix in $D'$ is of the form $\diag(\lambda_4^3,1,1,\lambda_4,\lambda_4^{-4})$
where $\lambda_4$ is a $15$-th root of unity.    We may as well write:
$$
D'=
 \left\{
\left(
\begin{array}{ccccc}
\lambda^{-2}\zeta^{3j}&&&&\\
&\lambda&&&\\
&&\lambda &&\\
&&&\zeta^{j}&\\
&&&&\zeta^{-4j}\\
\end{array}
\right): \lambda \in \CC^*,\ \zeta =e^{2\pi i/15},\  j=0,\dots,14
\right\}.
$$
This completes the proof.
\end{proof}

%% file: sec10_cub_surface.tex
\chapter{The moduli space of cubic surfaces}
\label{sec:surfaces}
As a demonstration of the techniques developed in the paper, we briefly outline how one obtains analogous results for the moduli space of cubic curves and surfaces.   The new results in this appendix are the computations of the Betti numbers of the toroidal and Naruki compactifications of the moduli space of cubic surfaces (Theorem~\ref{T:CubSurfH}).

\section{The moduli space of cubic curves}
The case of cubic curves is trivial, but nevertheless we review this situation, for completeness.  The GIT moduli space $\GIT_\curv$ has stable points corresponding to smooth cubic curves, and strictly semi-stable points corresponding to cubic curves with nodes.  There is a unique strictly polystable orbit, corresponding to the cubic curve $V(x_0x_1x_2)$, the so-called $3A_1$ cubic curve.
Being a normal rational projective variety of dimension~$1$, we have $\GIT_\curv\cong \PP^1$.
The natural period map is $\calM_\curv\to \fH/\Gamma_1$, taking a cubic curve to its Jacobian, with $\Gamma_1=\SL(2,\ZZ)$; here $\calM_\curv$ is the locus of smooth cubic curves.  As the Baily--Borel compactification $(\fH/\Gamma_1)^*$ is also a normal rational projective variety of dimension~$1$, it is also  isomorphic to $\PP^1$, and the period map extends to an isomorphism $\GIT_\curv\cong (\fH/\Gamma_1)^*$.  Note also that
the boundary of the Baily--Borel compactification is already a divisor (it is simply a point on a curve), and since $(\fH/\Gamma_1)^*$ is smooth, it is its own canonical toroidal compactification $\overline{\fH/\Gamma_1}=(\fH/\Gamma_1)^*$.  Finally, since $\GIT_\curv$ has a strictly polystable point, the Kirwan blowup is not just the identity map; however, since the Kirwan blowup $\MK_\curv$ is smooth, projective and of dimension~$1$, it is also isomorphic to $\PP^1$, so that $\MK_\curv\to \GIT_\curv$ is an isomorphism.   In other words, all of the compactifications in question are isomorphic to $\PP^1$, and the cohomology is obvious.

\section{The moduli space of cubic surfaces}
The moduli of cubic surfaces has a number of compactifications constructed in a  similar way to those of cubic threefolds.  To begin with, the GIT compactification $\GIT_\surf$
 can be described as follows (see e.g.,~\cite[\S 7.2(b)]{mukai}).  A  cubic surface $V$ is:

\begin{itemize}
\item stable if and only if it has at worst $A_1$ singularities,
\item semi-stable if and only if it is stable, or has at worst $A_2$ singularities, and does not contain the axes of the $A_2$ singularities,
\item strictly polystable if  and only if it is projectively equivalent to $V(x_0x_1x_2+x_3^3)$ (the so-called $3A_2$ cubic).
\end{itemize}
Note that it is a classical result that $\GIT_\surf\cong W\PP(1,2,3,4,5)$ (see~\cite[(2.4)]{DvGK}).

By considering the  triple cover of $\PP^3$ branched along a cubic surface, one obtains a cubic threefold, and via the period map for cubic threefolds, one obtains a period map to a $4$-dimensional ball quotient $\calM_\surf\to \calB_4/\Gamma_4$ (see~\cite{ACTsurf}); here $\calM_\surf$ is the locus of smooth cubic surfaces.  This is an open embedding, and the complement of the image is the Heegner divisor $D_n=\calD_n/\Gamma_4$.   The rational period map $\GIT_\surf\dashrightarrow (\calB_4/\Gamma_4)^*$ to the Baily--Borel compactification extends to an isomorphism, taking the discriminant $D_{A_1}\subset \GIT_\surf$  to the divisor $D_n$.  Under this isomorphism, the unique strictly polystable point  $\Delta\in \GIT_\surf$ corresponding to the $3A_2$ cubic is identified with the sole cusp of $(\calB_4/\Gamma_4)^*$, which we thus denote $c_{3A_2}$.  The Kirwan blowup $\MK_\surf\to \GIT_\surf$ is a blowup with center supported at $\Delta$.

In a different direction, Naruki~\cite{naruki} has constructed a modular compactification $\widetilde \calN$ of the moduli space of marked cubic surfaces (this was subsequently reworked by~\cite{HKT09} from a different perspective). There is a natural action by $W(E_6)$ on $\widetilde\calN$, and denoting by $\overline\calN=\widetilde\calN/W(E_6)$, we get another smooth (as always, up to finite quotient singularities) compactification for the moduli space of cubic surfaces. As discussed in~\cite{DvGK}, $\overline\calN$ maps to $\GIT_\surf\cong (\calB_4/\Gamma_4)^*$, and this map contracts a divisor to the boundary point $\Delta$ (resp.~$c_{3A_2}$);
denoting this divisor $D_{\overline \calN_{3A_2}}$, this contraction induces an isomorphism  $\overline\calN-D_{\overline \calN_{3A_2}}\cong \GIT_\surf-\Delta$~\cite[\S2.10]{DvGK}.
 In summary, we have a diagram (compare~\eqref{E:BirDiagMod})
\begin{equation}\label{eq_diag_surf}
\xymatrix{
&\MK_\surf \ar[ld] \ar[rd]\ar@{<-->}[r]&\overline{\calB_4/\Gamma_4} \ar[d]\ar@{<-->}[r]&\overline{\calN}\ar[ld]\\ \GIT_\surf \ar[rr]^{\sim}&&(\calB_4/\Gamma_4)^*
}
\end{equation}
where $\overline{\calB_4/\Gamma_4} $ is the (again, unique) toroidal compactification.  The purpose of this section is to establish  that these three compactifications ($\MK_\surf$, $\overline{\calB_4/\Gamma_4}$, and  $\overline{\calN}$) have the same cohomology.
Note that all three spaces are blowups of the point $c_{3A_2}\in (\calB_4/\Gamma_4)^*$; we expect that they are all isomorphic, but this is not yet known (compare Remark~\ref{rem_possible_iso}).

In~\cite{kirwanhyp} and~\cite{ZhangCubic}
 the (intersection) Betti numbers of the spaces $\GIT_\surf\cong (\calB_4/\Gamma_4)^*$ and $\MK_\surf$ were computed\footnote{
Note there is an error in~\cite[Thm.~1.6, p.50, and 5.2]{kirwanhyp} regarding the Betti numbers of $\MK_\surf$, corrected in~\cite{ZhangCubic}. Specifically, the set $\calR$ of connected components of stabilizers consists only of $\TT^2$, and does not also include $\operatorname{SO}(3,\CC)$, as claimed in~\cite[p.59]{kirwanhyp}: the only strictly polystable orbit is the orbit of the $3A_2$ cubic surface, with connected component of the stabilizer given by $\TT^2$.  The rest of the computations in~\cite{kirwanhyp} go through unchanged, and yield $P_t(\MK_\surf)=P^G_t(X_\surf^{ss})+A_{\TT^2}(t)\equiv (1+t^2+2t^4)+t^2\equiv 1+2t^2+2t^4\mod t^5$; i.e., one simply does not add the $A_{\operatorname{SO}(3,\CC)}(t)\equiv t^2+t^4 \mod t^5$ contribution from the erroneous group $R=\operatorname{SO}(3,\CC)$.   The computation of $IP_t(\GIT_\surf)$ is then also corrected by omitting the terms corresponding to $R=\operatorname{SO}(3,\CC)$, so that one obtains $IP_t(\GIT_\surf)=P_t(\MK_\surf)-B_{\TT^2}(t)\equiv (1+2t^2+2t^4)-(t^2+t^4)\equiv 1+t^2+t^4\mod t^5$; i.e.,  the formula for $IP_t(\GIT_\surf)$ in \cite[Thm.~1.6]{kirwanhyp} is correct.
}
:
\begin{equation}\label{E:KirThmCubSurf}
\begin{array}{r|ccccc}
j&0&2&4&6&8\\\hline
\dim H^j(\MK_\surf)&1&2&2&2&1 \rule{0pt}{2.6ex}\\
\dim IH^j(\GIT_\surf)=\dim IH^j((\calB_4/\Gamma_4)^*)&1&1&1&1&1 \rule{0pt}{2.6ex}
\end{array}
\end{equation}
with all odd degree (intersection) cohomology vanishing.
Note that the bottom row is immediate, since $\GIT_\surf$ is a weighted projective space, as recalled above.

For the cohomology of the toroidal compactification of the ball quotient model $\calB_4/\Gamma_4$ of the moduli of cubic surfaces, we apply the same approach (but, of course, with easier computational details) as for cubic threefolds (see Chapter~\ref{sec:toroidal}).  As announced, we obtain that the cohomology of the toroidal compactification coincides with the cohomology of the Kirwan blowup $\MK_\surf$.

\begin{teo}\label{T:CubSurfH}
The Betti numbers of the toroidal compactification of the ball quotient model $\overline{\calB_4/\Gamma_4}$ of the moduli space of cubic surfaces are as follows:
\begin{equation}\label{E:CubSurfH}
\begin{array}{r|ccccc}
\hskip2cm j&0&2&4&6&8\\\hline
\dim H^j(\overline{\calB_4/\Gamma_4})&1&2&2&2&1 \rule{0pt}{2.8ex}
\end{array}
\end{equation}
while all the odd degree cohomology vanishes.
\end{teo}

\section{The proof of Theorem~C.1}
In this section, following the setup of \S\ref{S:ArithmeticCusp}, we discuss the structure of the toroidal compactification $\overline{\calB_4/\Gamma_4}$  of the ball quotient model for surfaces, and prove Theorem~\ref{T:CubSurfH}.

\subsection{The Eisenstein lattice for cubic surfaces}
The Eisenstein lattice used by Allcock--Carlson--Toledo~\cite[(2.7.1)]{ACTsurf} to define the ball quotient model $\calB_4/\Gamma_4$ for the moduli of cubic surfaces is
\begin{equation}
 \Lambda = \calE_1(-1)+4 \calE_1
\end{equation}
with the associated $\ZZ$-lattice
$$
\Lambda_\ZZ=A_2+4A_2(-1)
$$
(see~\cite[\S5, \S6]{DvGK} for a discussion of the lattice $\Lambda_\ZZ$ and its relevance to the ball quotient construction).
Returning to the construction of $\calB_4/\Gamma_4$, we recall
$$
\calB_4:=\{[z]: z^2>0\}^+\subseteq \PP(\Lambda \otimes_{\calE}\CC),
$$
and $\Gamma_4:=\operatorname{O}(\Lambda)$ acts naturally (properly discontinuously) on $\calB_4$. Let us note that one can construct a natural $W(E_6)$-cover
\begin{equation}\label{markedball}
\calB_4/\Gamma_4^m\to \calB_4/\Gamma_4
\end{equation} of the ball quotient model parameterizing marked cubic surfaces (i.e., cubic surfaces with the $27$ lines labeled). This corresponds to an arithmetically defined normal subgroup $\Gamma_4^m\subset \Gamma_4$ with $\Gamma_4/\Gamma_4^m\cong \pm1\times W(E_6)$ (with $\pm 1$ acting trivially on $\calB_4$); we refer to~\cite[\S6.10]{DvGK} and~\cite[\S3]{ACTsurf} for details.

\subsection{Identifying the cusp of $(\calB_4/\Gamma_4)^*$}

From the description above, it is elementary to find a representative isotropic line $\calF\subseteq \Lambda$ defining the cusp $c_{3A_2}$, namely the one generated by
$$
h=(1,1,0,0,0).
$$
One then sees immediately that
$$
h^\perp/h=3\calE_1,
$$
and we recall then that $(3\calE_1)_\ZZ=3A_2(-1)$.

\subsection{The isometry group of the cusp}
Clearly,  $\operatorname{O}(\calE_1)=\ZZ_3\times \ZZ_2$ (compare~\eqref{E:OOE1}), with $\ZZ_3$ acting by $\omega$ and $\ZZ_2$ acting by $-1$.
It is easy to see that
\begin{equation}\label{E:OO3EE1}
\operatorname{O}(3\calE_1)=\operatorname{O}(\calE_1)^{\times 3}\rtimes S_3=(\ZZ_3\times \ZZ_2)^{\times 3}\rtimes S_3
\end{equation}
where the semi-direct product is given by the action of $S_3$ on the three copies of  $\operatorname{O}(\calE_1)$.

\subsection{The structure of the toroidal boundary divisor}
We denote the boundary divisor of $\overline{\calB_4/\Gamma_4}$
corresponding to the cusp $c_{3A_2}$ by $T_{3A_2}$.

\begin{lem}\label{lem_struc_tor_surf}
The following holds:
$$
T_{3A_2}\cong (E_\omega\otimes_{\calE}3\calE_1)/\operatorname{O}(3\calE_1) \ \ (\cong (E_\omega^3)/\operatorname{O}(3\calE_1)).
$$
\end{lem}

\begin{proof}
The proof is analogous to that  of  Proposition~\ref{prop_structure_tor}, with a minor difference.
To make this appendix accessible to readers who are primarily interested in cubic surfaces we will give a self-contained proof here, but also comment on the differences to the previous case.
We start with $\Lambda= \calE_1(-1) + 4\calE_1$ and
the isotropic vector $b_1:=h=(1,1,0,0,0)$.
We will denote the corresponding cusp given by the isotropic line $\calF=\calE h$ by $F$.
We then add $b_2,b_3,b_4$ where each $b_i$ is a generator of a copy of $3\calE_1=h^{\perp}/h$, and complement this by
$b_5=(1,-1,0,0,0)$. The difference to Proposition~\ref{prop_structure_tor} is that this is a $\QQ(\sqrt{-3})$-basis of $\calE_1(-1) + 4\calE_1$, and not an $\calE$-basis. With respect to this
basis the hermitian form is given by
$$
Q=
\left(
\begin{array}{c|c|c}
0 & 0 & 6 \\ \hline
0 & B & 0\\ \hline
 6 & 0 & 0
\end{array}
\right)
$$
where
$$
B=
\left(
\begin{array}{c|c|c}
3 & 0 & 0 \\ \hline
0 & 3 & 0\\ \hline
0 & 0 & 3
\end{array}
\right).
$$

In order to determine the structure of the boundary one first has to understand the structure of the
stabilizer subgroup $N(F)$ corresponding to $F$, i.e.~the subgroup of $\O(\Lambda)$ fixing the line spanned by $h$. A straightforward calculation, see~\cite[Sec.~4]{beh}, gives
\begin{equation}\label{pro:structureboundarycompA}
N(F)= \left\{ g \in \O(\Lambda): g= \left( \begin{array}{c|c|c}
u & v & w \\ \hline
0 & X & y \\ \hline
0 & 0 & s
\end{array}
\right)\right\}.
\end{equation}
Note that, in particular, this implies that $X\in \O(3\calE)$. Its unipotent radical is given by
\begin{equation}
W(F)= \left\{g \in N(F): g= \left(
\begin{array}{c|c|c}
1 & v & w \\ \hline
0 & 1 & y \\ \hline
0 & 0 & 1
\end{array}
\right) \right\}
\end{equation}
and finally the center of the unipotent radical is
\begin{equation}
U(F)= \left\{g\in W(F): g= \left(
\begin{array}{c|c|c}
1 & 0 & w \\ \hline
0 & 1 & 0 \\ \hline
0 & 0 & 1
\end{array}
\right), w \in \ZZ \right\} \cong \ZZ.
\end{equation}

We have natural coordinates  coordinates $(z_0:z_1: z_2:z_3:z_{4})$ on $\calB \subset \PP(\Lambda\otimes_{\calE}\CC)$ and we can assume that $z_4=1$. Then we  obtain a map
\begin{equation}
\begin{aligned} \calB &\to \CC^* \times \CC^3 \\  (z_0, z_1, z_2,  z_{3}) &\mapsto (t_0=e^{2 \pi i z_0}, z_1, z_2,  z_3)
\end{aligned}
\end{equation}
and adding the toroidal boundary amounts to adding $\{0\} \times \CC^3$.

The quotient $N(F)/U(F)$ then acts on $\calB/U(F)$ and this quotient gives the
toroidal compactification of $\calB$ near the cusp $F$. Here we are only interested in the structure of the boundary divisor and hence in the action of $N(F)/U(F)$ on  $\{0\} \times \CC^3$.
By a straightforward calculation
\begin{equation}\label{equ:action2}
g=\left(
\begin{array}{c|c|c}
u & v & w \\ \hline
0 & X & y \\ \hline
0 & 0 & s
\end{array}
\right): \underline{z} \mapsto  \frac{1}{s}(X\underline{z} + y)
\end{equation}
where $\underline{z}=(z_1, z_2, z_3)$. We first look at  the normal subgroup $W(F)$, matrices whose elements act as follows
$$
g=\left(
\begin{array}{c|c|c}
1 & v & w \\ \hline
0 & 1 & y \\ \hline
0 & 0 & 1
\end{array}
\right):  \underline{z} \mapsto \underline{z} + y.
$$
Since $g\in \O(\Lambda)$, we must necessarily have $y \in \calE^3$ (where we now use the notation $\calE^3$ rather than $3\calE$ since we want to
emphasize the vector space structure rather than the lattice). This is where there is a difference to the case of cubic fourfolds: it is no longer true that all vectors in  $\calE^3$ appear as entries $y$ in matrices $g\in W(F)$.
Indeed by a straightforward calculation, see~\cite[Sec.~4]{beh},  the condition that $g\in \O(\Lambda)$ is
$$
By+6\bar{v}^t=0, \quad \bar{y}^tBy+6w + 6 \bar w=0.
$$
Given $y$ we want to define $v$ by $\bar{v}^t=-\frac{1}{6}By$. Since $By \in 3 \cdot \calE^3$ and $v$ must be in $\calE^3$ this requires that $y \in 2 \cdot \calE^3$. Note that we can then also find a suitable $w\in \calE$.
However, scaling the lattice  by a factor $2$ gives isomorphic quotients showing
$$
\CC^3/W(F)\cong (E_{\omega})^3.
$$
The rest of the argument is now again very close to Proposition~\ref{prop_structure_tor}.
Clearly, the subgroup
$$
\left\{ g \in \O(\Lambda): g= \left( \begin{array}{c|c|c}
1 & 0 & 0 \\ \hline
0 & X & 0 \\ \hline
0 & 0 & 1
\end{array}
\right)\right\}.
$$
acts on $(E_{\omega})^3$ as claimed in the proposition.

It remains to consider elements of the form
$$
g=\left(
\begin{array}{c|c|c}
u & 0 & 0 \\ \hline
0 & 1 & 0 \\ \hline
0 & 0 & s
\end{array}
\right) \in N(F).
$$
The condition that such a matrix lies in $\O(\Lambda)$ is that $s\bar{u}=1$ with $s\in \calE$. Hence $s$ is a power of $\omega$ and these elements act on $(E_{\omega})^3$ by multiplication with powers of $\omega$. But by
(\ref{equ:action2})  these elements are already accounted for by matrices with $u=s=1$ and
$X \in \O(3\calE)$ and hence we do not get a further quotient. Thus the claim follows.
\end{proof}

\subsection{The cohomology of the toroidal boundary divisor}
It is elementary to see from the descriptions above that
$$
(\calE_1\otimes_{\calE}E_\omega)/\operatorname{O}(\calE_1)\cong \PP^1.
$$
It follows that
$$
T_{3A_2}=(3\calE_1\otimes_{\calE}E_\omega)/\operatorname{O}(3\calE_1)=(\PP^1)^3/S_3=\PP^3.
$$
In particular, we get:
\begin{cor}\label{lem_coho_torbound_ball}
The Betti numbers of the toroidal boundary divisor $T_{3A_2}$ of $\overline{\calB_4/\Gamma_4}$ are given by
$b_0( T_{3A_2})=b_2( T_{3A_2})=b_4( T_{3A_2})=b_6( T_{3A_2})=1$ and $b_1( T_{3A_2})=b_3( T_{3A_2})=b_5( T_{3A_2})=0$.
\end{cor}

\subsection{The cohomology of the toroidal compactification}\label{completeproof}
We can now complete the proof of Theorem~\ref{T:CubSurfH} using the Decomposition Theorem for the morphism $\overline {\calB_4/\Gamma_4}\to (\calB_4/\Gamma_4)^*$.
We have

\begin{align*}
P_t(\overline {\calB_4/\Gamma_4}) \equiv\  &1+t^2+t^4& \text{($IP_t((\calB_4/\Gamma_4)^*)$,  from~\eqref{E:KirThmCubSurf})}\\
&\                         +( t^2  +t^4)&\text{($T_{3A_2}=\PP^3$ contribution, from Corollary \ref{lem_coho_torbound_ball})}\\
              \equiv\ & 1+2t^2+2t^4 \mod t^{5}\!\!\!\!\!\!\!\!
\end{align*}
by applying equation~\eqref{eq:IHblowup} to determine the contribution to the cohomology of $\overline {\calB_4/\Gamma_4}$ from the exceptional divisor.

\section{The cohomology of the Naruki compactification} For completeness, let us note that the cohomology of the Naruki compactification $\overline \calN$ coincides with the cohomology of toroidal and Kirwan compactifications for the moduli of cubic surfaces.
\begin{pro}\label{T:Naruki}
The Betti numbers of the Naruki compactification $\overline \calN=\widetilde \calN/W(E_6)$ of the moduli space of cubic surfaces are as follows:
\begin{equation}\label{E:CubSurfHNar}
\begin{array}{r|ccccc}
\hskip2cm j&0&2&4&6&8\\\hline
\dim H^j(\overline{\calN})&1&2&2&2&1 \rule{0pt}{2.8ex}
\end{array}
\end{equation}
while all the odd degree cohomology vanishes.
\end{pro}
\begin{proof}
The Naruki compactification $\widetilde \calN$ is a modular compactification for the moduli of marked cubic surfaces. Clearly, $W(E_6)$ acts on $\widetilde \calN$, and we have defined $\overline \calN=\widetilde \calN/W(E_6)$. On the other hand, as discussed above we recall that there exists a marked ball quotient model $\calB_4/\Gamma_4^m$, which is a $W(E_6)$ cover of $\calB_4/\Gamma_4$ (see~\eqref{markedball}).
Then, there exists a ($W(E_6)$-equivariant) period map $\widetilde \calN\to (\calB_4/\Gamma_4^m)^*$ contracting $40$ divisors $D_i$ in $\widetilde \calN$ to the $40$ cusps of the Baily--Borel compactification $(\calB_4/\Gamma_4^m)^*$ (see~\cite[\S2.9]{DvGK}\footnote{In~\cite{DvGK}, the image of  the (extended) period map $\widetilde \calN\to (\calB_4/\Gamma_4^m)^*$ is denoted by $\calN$. For consistency with our notations, a better notation would be $\calN^*(=\calN)$. Of course, $\calN^*=(\calB_4/\Gamma_4^m)^*$ as the period map is surjective.}). Furthermore (cf. loc. cit.), $D_i\cong (\PP^1)^3$ (and the singularities of $(\calB_4/\Gamma_4^m)^*$ at the $40$ cusps are cones over the Segre embedding of $(\PP^1)^3$). The $40$ exceptional divisors $D_i$ are conjugated under the action of $W(E_6)$. Thus, taking the quotient by $W(E_6)$, we obtain $\overline \calN\to (\calB_4/\Gamma_4)^*$, which contracts a divisor $D$ to the unique cusp of $(\calB_4/\Gamma_4)^*$. Since $D$ is a quotient of $D_i\cong (\PP^1)^3$ by a finite group that contains $S_3$ permuting the three $\PP^1$ factors, it is immediate to see that $D$ has the rational cohomology of $\PP^3$. The claim now follows as before (see \S\ref{completeproof}).
\end{proof}

%% file: cubiccoh.bbl
\providecommand{\bysame}{\leavevmode\hbox to3em{\hrulefill}\thinspace}
\providecommand{\MR}{\relax\ifhmode\unskip\space\fi MR }
\providecommand{\MRhref}[2]{%
  \href{http://www.ams.org/mathscinet-getitem?mr=#1}{#2}
}
\providecommand{\href}[2]{#2}
\begin{thebibliography}{CMGHL17}

\bibitem[AB83]{atiyahbott83}
M.~F. Atiyah and R.~Bott, \emph{The {Y}ang-{M}ills equations over {R}iemann
  surfaces}, Philos. Trans. Roy. Soc. London Ser. A \textbf{308} (1983),
  no.~1505, 523--615.

\bibitem[ACT02]{ACTsurf}
D.~Allcock, J.~A. Carlson, and D.~Toledo, \emph{The complex hyperbolic geometry
  of the moduli space of cubic surfaces}, J. Algebraic Geom. \textbf{11}
  (2002), no.~4, 659--724.

\bibitem[ACT11]{act}
\bysame, \emph{The moduli space of cubic threefolds as a ball quotient}, Mem.
  Amer. Math. Soc. \textbf{209} (2011), no.~985, xii+70.

\bibitem[All00]{ALeech}
D.~Allcock, \emph{The {L}eech lattice and complex hyperbolic reflections},
  Invent. Math. \textbf{140} (2000), no.~2, 283--301.

\bibitem[All03]{allcock}
\bysame, \emph{The moduli space of cubic threefolds}, J. Algebraic Geom.
  \textbf{12} (2003), no.~2, 201--223.

\bibitem[AMRT10]{AMRT}
A.~Ash, D.~Mumford, M.~Rapoport, and Y.S. Tai, \emph{Smooth compactifications
  of locally symmetric varieties}, second ed., Cambridge Mathematical Library,
  Cambridge University Press, Cambridge, 2010, With the collaboration of Peter
  Scholze.

\bibitem[Bas07]{basak}
T.~Basak, \emph{The complex {L}orentzian {L}eech lattice and the {B}imonster},
  J. Algebra \textbf{309} (2007), no.~1, 32--56.

\bibitem[BBD82]{bbdg}
A.~A. Be\u{\i}linson, J.~Bernstein, and P.~Deligne, \emph{Faisceaux pervers},
  Analysis and topology on singular spaces, {I} ({L}uminy, 1981),
  Ast\'{e}risque, vol. 100, Soc. Math. France, Paris, 1982, pp.~5--171.

\bibitem[Beh12]{beh}
N.~Behrens, \emph{Singularities of ball quotients}, Geom. Dedicata \textbf{159}
  (2012), 389--407.

\bibitem[BLMM17]{bergeron}
N.~Bergeron, Z.~Li, J.~Millson, and C.~Moeglin, \emph{The {N}oether-{L}efschetz
  conjecture and generalizations}, Invent. Math. \textbf{208} (2017), no.~2,
  501--552.

\bibitem[Bri73]{Bri}
E.~Brieskorn, \emph{Sur les groupes de tresses [d'apr\`es {V}. {I}. {A}rnold]},
  S\'{e}minaire {B}ourbaki, 24\`eme ann\'{e}e (1971/1972), {E}xp. {N}o. 401,
  Springer, 1973, pp.~21--44. Lecture Notes in Math., Vol. 317.

\bibitem[BS06]{BS}
J.~Bernstein and O.~Schwarzman, \emph{Chevalley's theorem for the complex
  crystallographic groups}, J. Nonlinear Math. Phys. \textbf{13} (2006), no.~3,
  323--351.

\bibitem[Car85]{carlson}
J.~A. Carlson, \emph{The one-motif of an algebraic surface}, Compositio Math.
  \textbf{56} (1985), no.~3, 271--314.

\bibitem[CG72]{cg}
C.~H. Clemens and P.~A. Griffiths, \emph{The intermediate {J}acobian of the
  cubic threefold}, Ann. of Math. (2) \textbf{95} (1972), 281--356.

\bibitem[CMGHL15]{cubics}
S.~Casalaina-Martin, S.~Grushevsky, K.~Hulek, and R.~Laza, \emph{Complete
  moduli of cubic threefolds and their intermediate {J}acobians},
  arXiv:1510.08891, 2015.

\bibitem[CMGHL17]{prym}
\bysame, \emph{Extending the {P}rym map to toroidal compactifications of the
  moduli space of abelian varieties (with an appendix by {M}athieu {D}utour
  {S}ikiri\'c)}, J. Eur. Math. Soc. (JEMS) \textbf{19} (2017), no.~3, 659--723.

\bibitem[CML09]{cml}
S.~Casalaina-Martin and R.~Laza, \emph{The moduli space of cubic threefolds via
  degenerations of the intermediate {J}acobian}, J. Reine Angew. Math.
  \textbf{633} (2009), 29--65.

\bibitem[DK07]{DK}
I.~V. Dolgachev and S.~Kondo, \emph{Moduli of {$K3$} surfaces and complex ball
  quotients}, Arithmetic and geometry around hypergeometric functions, Progr.
  Math., vol. 260, Birkh\"{a}user, Basel, 2007, pp.~43--100.

\bibitem[Dol08]{Dolgachev-ref}
I.~V. Dolgachev, \emph{Reflection groups in algebraic geometry}, Bull. Amer.
  Math. Soc. (N.S.) \textbf{45} (2008), no.~1, 1--60.

\bibitem[DvGK05]{DvGK}
I.~V. Dolgachev, B.~van Geemen, and S.~Kondo, \emph{A complex ball
  uniformization of the moduli space of cubic surfaces via periods of {$K3$}
  surfaces}, J. Reine Angew. Math. \textbf{588} (2005), 99--148.

\bibitem[FMW98]{FMW}
R.~Friedman, J.~W. Morgan, and E.~Witten, \emph{Principal {$G$}-bundles over
  elliptic curves}, Math. Res. Lett. \textbf{5} (1998), no.~1-2, 97--118.

\bibitem[Fri84]{Fannals}
R.~Friedman, \emph{A new proof of the global {T}orelli theorem for {$K3$}
  surfaces}, Ann. of Math. (2) \textbf{120} (1984), no.~2, 237--269.

\bibitem[Fri13]{F13}
\bysame, \emph{On the ample cone of a rational surface with an anticanonical
  cycle}, Algebra Number Theory \textbf{7} (2013), no.~6, 1481--1504.

\bibitem[GH94]{GrHa78}
P.~A. Griffiths and J.~Harris, \emph{Principles of algebraic geometry}, Wiley
  Classics Library, John Wiley \& Sons, Inc., New York, 1994, Reprint of the
  1978 original.

\bibitem[GH17]{GH-IHAg-17}
S.~Grushevsky and K.~Hulek, \emph{The intersection cohomology of the {S}atake
  compactification of {$\mathcal A_g$} for {$g \leq4$}}, Math. Ann.
  \textbf{369} (2017), no.~3-4, 1353--1381.

\bibitem[GHK15]{GHK}
M.~Gross, P.~Hacking, and S.~Keel, \emph{Moduli of surfaces with an
  anti-canonical cycle}, Compos. Math. \textbf{151} (2015), no.~2, 265--291.

\bibitem[HH09]{HH1}
B.~Hassett and D.~Hyeon, \emph{Log canonical models for the moduli space of
  curves: the first divisorial contraction}, Trans. Amer. Math. Soc.
  \textbf{361} (2009), no.~8, 4471--4489.

\bibitem[HH13]{HH2}
\bysame, \emph{Log minimal model program for the moduli space of stable curves:
  the first flip}, Ann. of Math. (2) \textbf{177} (2013), no.~3, 911--968.
  \MR{3034291}

\bibitem[HKN10]{HKN}
M.~Hentschel, A.~Krieg, and G.~Nebe, \emph{On the classification of lattices
  over {$\mathbb Q(\sqrt{-3})$}, which are even unimodular {$\mathbb
  Z$}-lattices}, Abh. Math. Semin. Univ. Hambg. \textbf{80} (2010), no.~2,
  183--192.

\bibitem[HKT09]{HKT09}
P.~Hacking, S.~Keel, and J.~Tevelev, \emph{Stable pair, tropical, and log
  canonical compactifications of moduli spaces of del {P}ezzo surfaces},
  Invent. Math. \textbf{178} (2009), no.~1, 173--227.

\bibitem[HT18]{hulektommasi}
K.~{Hulek} and O.~{Tommasi}, \emph{The topology of {$A_g$} and its
  compactifications}, Geometry of {M}oduli---the {A}bel {S}ymposium 2017 (J.~A.
  Christophersen and K.~Ranestad, eds.), Abel Symp., vol.~14, Springer, 2018,
  pp.~135--193.

\bibitem[Kir84]{kirwan84}
F.~C. Kirwan, \emph{Cohomology of quotients in symplectic and algebraic
  geometry}, Mathematical Notes, vol.~31, Princeton University Press,
  Princeton, NJ, 1984.

\bibitem[Kir85]{kirwanblowup}
\bysame, \emph{Partial desingularisations of quotients of nonsingular varieties
  and their {B}etti numbers}, Ann. of Math. (2) \textbf{122} (1985), no.~1,
  41--85.

\bibitem[Kir86]{kirwanrational1}
\bysame, \emph{Rational intersection cohomology of quotient varieties}, Invent.
  Math. \textbf{86} (1986), no.~3, 471--505.

\bibitem[Kir89]{kirwanhyp}
\bysame, \emph{Moduli spaces of degree {$d$} hypersurfaces in {${\bf P}_n$}},
  Duke Math. J. \textbf{58} (1989), no.~1, 39--78.

\bibitem[KL89a]{KL1}
F.~C. Kirwan and R.~Lee, \emph{The cohomology of moduli spaces of {$K3$}
  surfaces of degree {$2$}. {I}}, Topology \textbf{28} (1989), no.~4, 495--516.

\bibitem[KL89b]{KL2}
\bysame, \emph{The cohomology of moduli spaces of {$K3$} surfaces of degree
  {$2$}. {II}}, Proc. London Math. Soc. (3) \textbf{58} (1989), no.~3,
  559--582.

\bibitem[KLW87]{KLW}
F.~C. Kirwan, R.~Lee, and S.~H. Weintraub, \emph{Quotients of the complex ball
  by discrete groups}, Pacific J. Math. \textbf{130} (1987), no.~1, 115--141.

\bibitem[Laz09]{rthesis}
R.~Laza, \emph{Deformations of singularities and variation of {GIT} quotients},
  Trans. Amer. Math. Soc. \textbf{361} (2009), no.~4, 2109--2161.

\bibitem[Laz10]{laza}
\bysame, \emph{The moduli space of cubic fourfolds via the period map}, Ann. of
  Math. (2) \textbf{172} (2010), no.~1, 673--711.

\bibitem[LO18]{LOG2}
R.~Laza and K.~G. O'Grady, \emph{G{IT} versus {B}aily-{B}orel compactification
  for quartic {$K3$} surfaces}, Geometry of moduli (J.~A. Christophersen and
  K.~Ranestad, eds.), Abel Symp., vol.~14, Springer, Cham, 2018, pp.~217--283.

\bibitem[LO19]{LOG1}
\bysame, \emph{Birational geometry of the moduli space of quartic {$K3$}
  surfaces}, Compos. Math. \textbf{155} (2019), no.~9, 1655--1710.

\bibitem[Loo03]{l1}
E.~Looijenga, \emph{Compactifications defined by arrangements. {I}. {T}he ball
  quotient case}, Duke Math. J. \textbf{118} (2003), no.~1, 151--187.

\bibitem[Loo77]{Lroot}
\bysame, \emph{Root systems and elliptic curves}, Invent. Math. \textbf{38}
  (1976/77), no.~1, 17--32.

\bibitem[LS86]{LSa}
G.~I. Lehrer and L.~Solomon, \emph{On the action of the symmetric group on the
  cohomology of the complement of its reflecting hyperplanes}, J. Algebra
  \textbf{104} (1986), no.~2, 410--424.

\bibitem[LS07]{ls}
E.~Looijenga and R.~Swierstra, \emph{The period map for cubic threefolds},
  Compos. Math. \textbf{143} (2007), no.~4, 1037--1049.

\bibitem[LT09]{LT}
G.~I. Lehrer and D.~E. Taylor, \emph{Unitary reflection groups}, Australian
  Mathematical Society Lecture Series, vol.~20, Cambridge University Press,
  Cambridge, 2009.

\bibitem[LX19]{LiuXu}
Y.~Liu and C.~Xu, \emph{K-stability of cubic threefolds}, Duke Math. J.
  \textbf{168} (2019), no.~11, 2029--2073.

\bibitem[MFK94]{GIT}
D.~Mumford, J.~Fogarty, and F.~C. Kirwan, \emph{Geometric invariant theory},
  third ed., Ergebnisse der Mathematik und ihrer Grenzgebiete (2), vol.~34,
  Springer-Verlag, Berlin, 1994.

\bibitem[Mil56]{milnor56}
J.~Milnor, \emph{Construction of universal bundles. {II}}, Ann. of Math. (2)
  \textbf{63} (1956), 430--436.

\bibitem[Mok15]{mok}
C.~P. Mok, \emph{Endoscopic classification of representations of quasi-split
  unitary groups}, Mem. Amer. Math. Soc. \textbf{235} (2015), no.~1108, vi+248.

\bibitem[MOT15]{Ma}
S.~Ma, H.~Ohashi, and S.~Taki, \emph{Rationality of the moduli spaces of
  {E}isenstein {$K3$} surfaces}, Trans. Amer. Math. Soc. \textbf{367} (2015),
  no.~12, 8643--8679.

\bibitem[Muk03]{mukai}
S.~Mukai, \emph{An introduction to invariants and moduli}, Cambridge Studies in
  Advanced Mathematics, vol.~81, Cambridge University Press, Cambridge, 2003,
  Translated from the 1998 and 2000 Japanese editions by W. M. Oxbury.

\bibitem[Nar82]{naruki}
I.~Naruki, \emph{Cross ratio variety as a moduli space of cubic surfaces},
  Proc. London Math. Soc. (3) \textbf{45} (1982), no.~1, 1--30, With an
  appendix by Eduard Looijenga.

\bibitem[Nik79]{nikulin}
V.~V. Nikulin, \emph{Integer symmetric bilinear forms and some of their
  geometric applications}, Izv. Akad. Nauk SSSR Ser. Mat. \textbf{43} (1979),
  no.~1, 111--177, 238.

\bibitem[Pin74]{pinkham}
H.~C. Pinkham, \emph{Deformations of algebraic varieties with {$G_{m}$}
  action}, Soci\'{e}t\'{e} Math\'{e}matique de France, Paris, 1974,
  Ast\'{e}risque, No. 20.

\bibitem[Rei89]{reichstein}
Z.~Reichstein, \emph{Stability and equivariant maps}, Invent. Math. \textbf{96}
  (1989), no.~2, 349--383.

\bibitem[Sca87]{scattone}
F.~Scattone, \emph{On the compactification of moduli spaces for algebraic
  {$K3$} surfaces}, Mem. Amer. Math. Soc. \textbf{70} (1987), no.~374, x+86.

\bibitem[Wei94]{weibel94}
C.~A. Weibel, \emph{An introduction to homological algebra}, Cambridge Studies
  in Advanced Mathematics, vol.~38, Cambridge University Press, Cambridge,
  1994.

\bibitem[Yok02]{yokoyama}
M.~Yokoyama, \emph{Stability of cubic 3-folds}, Tokyo J. Math. \textbf{25}
  (2002), no.~1, 85--105.

\bibitem[Zha05]{ZhangCubic}
J.~Zhang, \emph{Geometric compactification of moduli space of cubic surfaces
  and {K}irwan blowup}, ProQuest LLC, Ann Arbor, MI, 2005, Thesis (Ph.D.)--Rice
  University.

\end{thebibliography}
